\documentclass[twoside]{amsart}

\usepackage{amsmath,amsfonts,amsthm,mathrsfs}
\usepackage{amssymb}

\usepackage[unicode,bookmarks]{hyperref} 
\usepackage[usenames,dvipsnames]{xcolor}
\hypersetup{colorlinks=false}

\usepackage{expl3}

\ExplSyntaxOn
\prop_new:N \g_cite_map_prop
\tl_new:N \l_citekey_result_tl

\cs_new:Npn \mapcitekey #1#2 {
  \clist_map_inline:nn {#2}
       {  \prop_gput:Nnn  \g_cite_map_prop  {##1} {#1}   }
}

\cs_new:Npn \getcitekey #1 {
   \prop_get:NoN \g_cite_map_prop{#1}  \l_citekey_result_tl
   \quark_if_no_value:NF \l_citekey_result_tl
       {  \tl_set_eq:NN #1  \l_citekey_result_tl  }
}

\cs_new:Npn \showcitekeymaps {\prop_show:N  \g_cite_map_prop }
\ExplSyntaxOff

\usepackage{etoolbox}
\makeatletter
\patchcmd{\@citex}{\if@filesw}{\getcitekey\@citeb \if@filesw}%
    {\typeout{*** SUCCESS ***}}{\typeout{*** FAIL ***}}
\patchcmd{\nocite}{\if@filesw}{\getcitekey\@citeb \if@filesw}%
    {\typeout{*** SUCCESS ***}}{\typeout{*** FAIL ***}}
\makeatother

\usepackage{enumitem}

\newenvironment{aenumerate}{%
	\begin{enumerate}[label=(\alph{*}), ref=(\alph{*})]
}{%
	\end{enumerate}%
}

\usepackage{chngcntr}

\usepackage{tikz}
\usepackage{tikz-cd}

\newcommand{\derR}{\mathbf{R}}

\newcommand{\norm}[1]{\lVert#1\rVert}

\newcommand{\abs}[1]{\lvert #1 \rvert}
\newcommand{\bigabs}[1]{\bigl\lvert #1 \bigr\rvert}

\newcommand{\eps}{\varepsilon}

\newcommand{\tensor}{\otimes}
\newcommand{\del}{\partial}

\newcommand{\dbar}{\bar{\partial}}

\newcommand{\dz}{\mathit{dz}}
\newcommand{\dzb}{d\bar{z}}

\newcommand{\shHom}{\mathcal{H}\hspace{-1pt}\mathit{om}}

\newcommand{\NN}{\mathbb{N}}

\newcommand{\QQ}{\mathbb{Q}}
\newcommand{\RR}{\mathbb{R}}
\newcommand{\CC}{\mathbb{C}}

\newcommand{\PP}{\mathbb{P}}

\newcommand{\menge}[2]{\bigl\{ \thinspace #1 \thinspace\thinspace \big\vert%
\thinspace\thinspace #2 \thinspace \bigr\}}
\newcommand{\Menge}[2]{\Bigl\{ \thinspace #1 \thinspace\thinspace \Big\vert%
\thinspace\thinspace #2 \thinspace \Bigr\}}

\newcommand{\MENGE}[2]{\left\{ \thinspace #1 \thinspace\thinspace \middle\vert%
\thinspace\thinspace #2 \thinspace \right\}}

\DeclareMathOperator{\id}{id}

\DeclareMathOperator{\Supp}{Supp}
\DeclareMathOperator{\codim}{codim}

\DeclareMathOperator{\End}{End}
\DeclareMathOperator{\Hom}{Hom}

\DeclareMathOperator{\Alb}{Alb}
\DeclareMathOperator{\Pic}{Pic}

\newcommand{\define}[1]{\emph{#1}}

\newcommand{\shf}[1]{\mathscr{#1}}
\newcommand{\OX}{\shf{O}_X}

\newcommand{\restr}[1]{\big\vert_{#1}}

\newcommand{\argbl}{-}

\def\overbar#1#2#3{{%
	\setbox0=\hbox{\displaystyle{#1}}%
	\dimen0=\wd0
	\advance\dimen0 by -#2 
	\vbox {\nointerlineskip \moveright #3 \vbox{\hrule height 0.3pt width \dimen0}%
		\nointerlineskip \vskip 1.5pt \box0}%
}}

\newcommand{\into}{\hookrightarrow}

\newcommand{\inner}[2]{\langle #1, #2 \rangle}

\newcommand{\fu}{f^{\ast}}
\newcommand{\fl}{f_{\ast}}

\newcommand{\pu}{p^{\ast}}
\newcommand{\pl}{p_{\ast}}

\newcommand{\tl}{t_{\ast}}
\newcommand{\gu}{g^{\ast}}

\newcommand{\shF}{\shf{F}}
\newcommand{\shG}{\shf{G}}

\newcommand{\shO}{\shf{O}}

\makeatletter
\let\@@seccntformat\@seccntformat
\renewcommand*{\@seccntformat}[1]{%
  \expandafter\ifx\csname @seccntformat@#1\endcsname\relax
    \expandafter\@@seccntformat
  \else
    \expandafter
      \csname @seccntformat@#1\expandafter\endcsname
  \fi
    {#1}%
}
\newcommand*{\@seccntformat@subsection}[1]{%
  \textbf{\csname the#1\endcsname.}
}
\makeatother

\makeatletter
\let\@paragraph\paragraph
\renewcommand*{\paragraph}[1]{%
	\vspace{0.3\baselineskip}%
	\@paragraph{\textit{#1}}%
}
\makeatother

\counterwithin{equation}{subsection}
\counterwithout{subsection}{section}
\counterwithin{figure}{subsection}

\newtheorem{theorem}[equation]{Theorem}
\newtheorem*{theorem*}{Theorem}
\newtheorem{lemma}[equation]{Lemma}
\newtheorem*{lemma*}{Lemma}
\newtheorem{corollary}[equation]{Corollary}
\newtheorem*{corollary*}{Corollary}
\newtheorem{proposition}[equation]{Proposition}
\newtheorem*{proposition*}{Proposition}

\newtheorem*{conjecture*}{Conjecture}

\theoremstyle{definition}
\newtheorem{definition}[equation]{Definition}
\newtheorem*{definition*}{Definition}
\theoremstyle{remark}
\newtheorem*{remark}{Remark}

\newtheorem{example}[equation]{Example}
\newtheorem*{example*}{Example}
\newtheorem*{problem*}{Problem}
\newtheorem*{note}{Note}

\theoremstyle{plain}

\newcommand{\theoremref}[1]{\hyperref[#1]{Theorem~\ref*{#1}}}
\newcommand{\lemmaref}[1]{\hyperref[#1]{Lemma~\ref*{#1}}}
\newcommand{\definitionref}[1]{\hyperref[#1]{Definition~\ref*{#1}}}
\newcommand{\propositionref}[1]{\hyperref[#1]{Proposition~\ref*{#1}}}
\newcommand{\conjectureref}[1]{\hyperref[#1]{Conjecture~\ref*{#1}}}
\newcommand{\corollaryref}[1]{\hyperref[#1]{Corollary~\ref*{#1}}}
\newcommand{\exampleref}[1]{\hyperref[#1]{Example~\ref*{#1}}}
\newcommand{\exerciseref}[1]{\hyperref[#1]{Exercise~\ref*{#1}}}

\makeatletter
\let\old@caption\caption
\renewcommand*{\caption}[1]{%
	\setcounter{figure}{\value{equation}}%
	\stepcounter{equation}%
	\old@caption{#1}\relax%
}
\makeatother

\newcounter{intro}

\newtheorem{intro-conjecture}[intro]{Conjecture}
\newtheorem{intro-corollary}[intro]{Corollary}
\newtheorem{intro-theorem}[intro]{Theorem}

\newcommand{\OA}{\mathscr{O}_A}

\newcommand{\Ah}{\hat{A}}
\newcommand{\omY}{\omega_Y}

\newcommand{\OY}{\shO_Y}

\newcommand{\parref}[1]{\hyperref[#1]{\S\ref*{#1}}}
\newcommand{\chapref}[1]{\hyperref[#1]{Chapter~\ref*{#1}}}

\makeatletter
\newcommand*\if@single[3]{%
  \setbox0\hbox{${\mathaccent"0362{#1}}^H$}%
  \setbox2\hbox{${\mathaccent"0362{\kern0pt#1}}^H$}%
  \ifdim\ht0=\ht2 #3\else #2\fi
  }
\newcommand*\rel@kern[1]{\kern#1\dimexpr\macc@kerna}
\newcommand*\widebar[1]{\@ifnextchar^{{\wide@bar{#1}{0}}}{\wide@bar{#1}{1}}}
\newcommand*\wide@bar[2]{\if@single{#1}{\wide@bar@{#1}{#2}{1}}{\wide@bar@{#1}{#2}{2}}}
\newcommand*\wide@bar@[3]{%
  \begingroup
  \def\mathaccent##1##2{%
    \if#32 \let\macc@nucleus\first@char \fi
    \setbox\z@\hbox{$\macc@style{\macc@nucleus}_{}$}%
    \setbox\tw@\hbox{$\macc@style{\macc@nucleus}{}_{}$}%
    \dimen@\wd\tw@
    \advance\dimen@-\wd\z@
    \divide\dimen@ 3
    \@tempdima\wd\tw@
    \advance\@tempdima-\scriptspace
    \divide\@tempdima 10
    \advance\dimen@-\@tempdima
    \ifdim\dimen@>\z@ \dimen@0pt\fi
    \rel@kern{0.6}\kern-\dimen@
    \if#31
      \overline{\rel@kern{-0.6}\kern\dimen@\macc@nucleus\rel@kern{0.4}\kern\dimen@}%
      \advance\dimen@0.4\dimexpr\macc@kerna
      \let\final@kern#2%
      \ifdim\dimen@<\z@ \let\final@kern1\fi
      \if\final@kern1 \kern-\dimen@\fi
    \else
      \overline{\rel@kern{-0.6}\kern\dimen@#1}%
    \fi
  }%
  \macc@depth\@ne
  \let\math@bgroup\@empty \let\math@egroup\macc@set@skewchar
  \mathsurround\z@ \frozen@everymath{\mathgroup\macc@group\relax}%
  \macc@set@skewchar\relax
  \let\mathaccentV\macc@nested@a
  \if#31
    \macc@nested@a\relax111{#1}%
  \else
    \def\gobble@till@marker##1\endmarker{}%
    \futurelet\first@char\gobble@till@marker#1\endmarker
    \ifcat\noexpand\first@char A\else
      \def\first@char{}%
    \fi
    \macc@nested@a\relax111{\first@char}%
  \fi
  \endgroup
}
\makeatother

\mapcitekey{Saito:HodgeModules}{Saito-HM}
\mapcitekey{Saito:MixedHodgeModules}{Saito-MHM}
\mapcitekey{Saito:Theory}{Saito-th}
\mapcitekey{Saito:Kollar}{Saito-K}
\mapcitekey{Saito:Kaehler}{Saito-Kae}
\mapcitekey{Schnell:GV_course}{Schnell}
\mapcitekey{Takegoshi:HigherDirectImages}{Takegoshi}
\mapcitekey{Wang:TorsionPoints}{Wang}
\mapcitekey{Schnell:lazarsfeld}{Schnell-laz}
\mapcitekey{Popa+Schnell:pluricanonical}{pluricanonical}
\mapcitekey{Green+Lazarsfeld:GV1}{GL1}
\mapcitekey{Green+Lazarsfeld:GV2}{GL2}
\mapcitekey{Simpson:Subspaces}{Simpson}
\mapcitekey{Kollar:DualizingII}{Kollar}
\mapcitekey{Beilinson+Bernstein+Deligne}{BBD}
\mapcitekey{Deligne:Finitude}{Deligne}
\mapcitekey{Deligne:HodgeII}{Deligne-H}
\mapcitekey{Lazarsfeld+Popa+Schnell:BGG}{LPS}
\mapcitekey{Schmid:VHS}{Schmid}
\mapcitekey{Chen+Jiang:VanishingEulerCharacteristic}{ChenJiang}
\mapcitekey{Ueno:ClassificationTheory}{Ueno}
\mapcitekey{Schnell:saito-vanishing}{Schnell-van}
\mapcitekey{Pareschi+Popa:cdf}{PP1}
\mapcitekey{Pareschi+Popa:regIII}{PP2}
\mapcitekey{Pareschi+Popa:GV}{PP3}
\mapcitekey{Pareschi+Popa:regI}{PP4}
\mapcitekey{Popa:perverse}{Popa}
\mapcitekey{Ben-Bassat+Block+Pantev:noncomm_FM}{BBP}
\mapcitekey{Pareschi:BasicResults}{Pareschi}
\mapcitekey{Varouchas:image}{Varouchas}
\mapcitekey{Ein+Lazarsfeld:Theta}{EL}
\mapcitekey{Chen+Hacon:abelian}{CH1}
\mapcitekey{Chen+Hacon:FiberSpaces}{CH2}
\mapcitekey{Chen+Hacon:Iitaka}{CH3}
\mapcitekey{Chen+Hacon:Kodaira}{CH-Kod}
\mapcitekey{Lazarsfeld+Popa:BGG}{LP}
\mapcitekey{Hacon:GV}{Hacon}
\mapcitekey{Birkenhake+Lange:ComplexTori}{BL}
\mapcitekey{Schmid+Vilonen:Unitary}{SV}
\mapcitekey{Cattani+Kaplan+Schmid:Degeneration}{CKS}
\mapcitekey{Zucker:DegeneratingCoefficients}{Zucker}
\mapcitekey{Zhang:Quaternions}{Zhang}
\mapcitekey{Debarre:Ample}{Debarre}
\mapcitekey{Grauert+Peternell+Remmert:SCV7}{SCV}
\mapcitekey{Mourougane+Takayama:Metric}{MT}
\mapcitekey{Arapura:GV}{Arapura}
\mapcitekey{Chen+Hacon:Varieties}{CH-pisa}
\mapcitekey{Guan+Zhou:Extension}{GZ}
\mapcitekey{Demailly:ComplexGeometry}{Demailly}
\mapcitekey{Gunning+Rossi:AnalyticFunctions}{GR}
\mapcitekey{Siu:TwistedPluricanonicalSections}{Siu}
\mapcitekey{Paun:InvariancePlurigenera}{Paun}
\mapcitekey{Paun:Survey}{Paun2}
\mapcitekey{Blocki:OhsawaTakegoshi}{Blocki}
\mapcitekey{Narasimhan+Simha:AmpleCanonicalClass}{NS}
\mapcitekey{Kawamata:Curves}{Ka-curve}
\mapcitekey{Lai:VarietiesFibered}{Lai}
\mapcitekey{Siu:SectionExtension}{Siu-ext}
\mapcitekey{Lazarsfeld:Positivity2}{Lazarsfeld}
\mapcitekey{Kollar:DualizingI}{Kollar}
\mapcitekey{Kollar:ShafarevichMaps}{Kollar95}
\mapcitekey{Mori:Bowdoin}{Mori}
\mapcitekey{Hacon+Pardini:BirationalGeometry}{HP}
\mapcitekey{Jiang:Kawamata}{Jiang}
\mapcitekey{Esnault+Viehweg:VanishingTheorems}{EV}

\mapcitekey{Raufi:SingularHermitianMetrics}{Raufi}
\mapcitekey{Berdtsson+Paun:BergmanKernels}{BP}
\mapcitekey{Cao+Paun:AbelianVarieties}{CP}
\mapcitekey{Paun+Takayama:Positivity}{PT}
\mapcitekey{Berndtsson:Curvature}{Berndtsson}
\mapcitekey{Tsuji:CanonicalSingularMetrics}{Tsuji}
\mapcitekey{Cao:OhsawaTakegoshi}{Cao}
\mapcitekey{Berdtsson+Lempert:OhsawaTakegoshi}{BL}
\mapcitekey{Demailly:ExtensionNonReduced}{Demailly-extension}

\newcommand{\omX}{\omega_X}
\newcommand{\omXm}{\omX^{\tensor m}}
\newcommand{\omXY}{\omega_{X/Y}}
\newcommand{\omXYm}{\omXY^{\tensor m}}
\newcommand{\shI}{\mathcal{I}}

\newcommand{\shFd}{\shF^{\ast}}

\DeclareMathOperator{\Mat}{Mat}
\newcommand{\wbar}[1]{\widebar{\vphantom{#1}#1}}

\newcommand{\Vfin}{V_{\mathit{fin}}}
\newcommand{\Vd}{V^{\ast}}
\newcommand{\hd}{h^{\ast}}
\newcommand{\Hd}{H^{\ast}}
\newcommand{\Ed}{E^{\ast}}
\newcommand{\shMs}{\mathscr{M}}
\newcommand{\dx}{\mathit{dx}}
\newcommand{\dy}{\mathit{dy}}
\newcommand{\dt}{\mathit{dt}}
\newcommand{\dtb}{d\bar{t}}
\newcommand{\Dbcoh}{D_{\mathit{coh}}^{\mathit{b}}}
\renewcommand{\restr}[1]{\vert_{#1}}

\DeclareMathOperator{\var}{var}

\newtheorem*{iitaka-conj}{Iitaka's conjecture}

\begin{document}

\title[Algebraic fiber spaces over abelian varieties]{%
	Algebraic fiber spaces over abelian varieties: \\
	around a recent theorem by Cao and P\u{a}un}
	
\dedicatory{To Lawrence Ein, on his $60$th birthday.}	

\author[C.~D.~Hacon]{Christopher Hacon}
\address{Department of Mathematics \\  
University of Utah\\  
Salt Lake City, UT 84112, USA}
\email{hacon@math.utah.edu}
\author[M.~Popa]{Mihnea Popa}
\address{Department of Mathematics, Northwestern University,
2033 Sheridan Road, Evanston, IL 60208, USA} 
\email{mpopa@math.northwestern.edu}

\author[Ch.~Schnell]{Christian Schnell}
\address{Department of Mathematics, Stony Brook University, Stony Brook, NY 11794-3651}
\email{cschnell@math.sunysb.edu}
\begin{abstract} 
We present a simplified proof for a recent theorem by
Junyan Cao and Mihai P\u{a}un, confirming a special case of Iitaka's
$C_{n,m}$ conjecture: if $f \colon X\to Y$ is an algebraic fiber space, and if the
Albanese mapping of $Y$ is generically finite over its image, then
we have the inequality of Kodaira dimensions $\kappa (X)\geq \kappa (Y)+\kappa (F)$,
where $F$ denotes a general fiber of $f$. We include a detailed survey of the main
algebraic and analytic techniques, especially the construction of singular hermitian
metrics on pushforwards of adjoint bundles (due to Berndtsson, P\u{a}un, and
Takayama).
\end{abstract}
\date{\today}
\maketitle

\section{Introduction}

\subsection{Main result}

In the classification of algebraic varieties up to birational equivalence, the most fundamental
invariants of a smooth projective variety $X$ are the spaces of global sections of
the pluricanonical bundles $\omXm$. The rate of growth of the plurigenera $P_m(X) =
\dim H^0(X, \omXm)$ is measured by the \define{Kodaira dimension}
\[
	\kappa(X) = \limsup_{m \to +\infty} \frac{\log P_m(X)}{\log m} 
		\in \{-\infty, 0, 1, \dotsc, \dim X\}.
\]
The following conjecture by Iitaka and Viehweg predicts the behavior of the Kodaira
dimension in families. Recall that an \define{algebraic fiber space} is a surjective morphism
with connected fibers between two smooth projective varieties.

\begin{conjecture*}
Let $f \colon X \to Y$ be an algebraic fiber space with general fiber $F$. Provided
that $\kappa(Y) \geq 0$, the Kodaira dimension of $X$ satisfies the inequality
\[
	\kappa(X) \geq \kappa(F) + \max \{ \kappa(Y), \var(f) \},
\]
where $\var(f)$ measures the birational variation in moduli of the fibers.
\end{conjecture*}

Using analytic techniques, Cao and P\u{a}un \cite{CP} have recently proved the
conjectured subadditivity of the Kodaira dimension in the case where $Y$ is an
abelian variety; as $\kappa(Y) = 0$, this amounts to the inequality
\[
	\kappa(X) \geq \kappa(F).
\]
With very little extra work, one can deduce the subadditivity of the Kodaira
dimension in any algebraic fiber space whose base $Y$ has maximal Albanese
dimension, meaning that the Albanese mapping $Y\to \Alb(Y)$ is generically finite
over its image. This includes of course the case where $Y$ is a curve of genus $\geq
1$, where the following result was first proved by Kawamata
\cite[Theorem~2]{Ka-curve}.

\begin{theorem}\label{main}
Let $f \colon X \to Y$ be an algebraic fiber space with general fiber $F$. Assume
that $Y$ has maximal Albanese dimension, then  $\kappa(X) \geq \kappa(F)+\kappa (Y)$. 
\end{theorem}

\begin{remark} 
A proof of this result is also claimed in \cite{CH-Kod}, however  the proof given there
is incomplete because of a serious mistake in \S4. 
\end{remark}

The purpose of this paper is to explain a simplified proof of the Cao-P\u aun theorem
that combines both analytic and algebraic techniques. We first reduce to the case
when $\kappa (X) = 0$ and $Y$ is an abelian variety, where we then prove a more
precise statement (Theorem \ref{t-pushforward}).  This is done in \chapref{scn:main}.
We then take the opportunity to provide a detailed survey of the results that are
used in the proof, for the benefit of those readers who are more familiar with one or
the other side of the story. 

In \chapref{GV}, we discuss the main algebraic tools, contained mostly in the papers
\cite{CH3,Hacon,Lai,PP3,pluricanonical}, namely results from generic vanishing
theory. The upshot of the discussion is that when $f\colon X \to A$ is a fiber space over
an abelian variety, with $\kappa (X) = 0$, then for all $m$ sufficiently large and
divisible, $\fl \omXm$ is a unipotent vector bundle on $A$, meaning a successive
extension of copies of $\shO_A$. This is as far as the algebraic techniques seem to
go at present.
While we recall the basic generic vanishing  and Fourier-Mukai tools involved, as this topic is well-established 
in the literature, we provide fewer background details. Sources where a comprehensive
treatment can be found include the lecture notes \cite{Schnell}, as well as 
\cite{Pareschi,Popa}.

In \chapref{analytic} and \chapref{chap:NS} we discuss the main analytic tools,
contained mostly in the papers \cite{BP,PT,CP}, namely the existence of
singular metrics with semi-positive curvature (in a suitable sense) on pushforwards
of pluricanonical bundles, and a very surprising criterion for such a  metric to be
smooth and flat. This time, the upshot is that when $f \colon X \to A$ is a fiber
space onto an abelian variety, with $\kappa (X) = 0$, then $\fl \omXm$ is a vector
bundle with a flat hermitian metric. Because a unipotent vector bundle with a flat
hermitian metric must be trivial, the algebraic and analytic results together lead to
the conclusion in \theoremref{main}. Since the analytic results are still new, and
are likely to be less familiar to algebraic geometers, we decided to include as many
details as possible. For another survey of these and related results, we recommend \cite{Paun2}.

\begin{remark}
For the sake of exposition, we present only the simplest version of the result by Cao
and P\u{a}un. One can tweak the proof of \theoremref{main} to show that 
the inequality in \theoremref{main} still holds when $X$ is replaced by a klt pair $(X,
\Delta)$, and $F$ by the pair $(F, \Delta_F)$; this is done in \cite[Theorem 4.22]{CP}.
\end{remark}

\subsection{What is new?}

The presentation in \chapref{analytic} contains various small improvements compared
to the original papers \cite{BP,PT,CP}. We briefly summarize the main points here.
Let $f \colon X \to Y$ be a projective and surjective holomorphic mapping between two
complex manifolds. Given a holomorphic line bundle $L$ on $X$, and a singular
hermitian metric $h$ on $L$ with semi-positive curvature, we construct a singular
hermitian metric on the torsion-free coherent sheaf
\[
	\shF = \fl \bigl( \omXY \tensor L \tensor \shI(h) \bigr),
\]
and show that this metric has semi-positive curvature (in the sense that the
logarithm of the norm of every local section of the dual sheaf is plurisubharmonic).
In \cite{PT}, P\u{a}un and Takayama constructed a singular hermitian metric with
semi-positive curvature on the larger sheaf $\fl \bigl( \omXY \tensor L \bigr)$,
under the additional assumption that the restriction of $(L, h)$ to a general fiber
of $f$ has trivial multiplier ideal. Another difference with \cite{PT} is that we do
not use approximation by smooth metrics or results about Stein manifolds; instead,
both the construction of the metric, and the proof that it has semi-positive
curvature, rely on the Ohsawa-Takegoshi extension theorem with sharp estimates, recently
proved by B{\l}ocki and Guan-Zhou \cite{Blocki,GZ}. This approach was suggested to
us by Mihai P\u{a}un.

\begin{note}
Berndtsson and Lempert \cite{BL} explain how one can use the curvature properties of
pushforwards of adjoint bundles to get a relatively short proof of (one version of)
the Ohsawa-Takegoshi theorem with sharp estimates. This suggests that the two results
are not so far from each other. That said, we hope that using the
Ohsawa-Takegoshi theorem as a black box will make the proof of the main result more
accessible to algebraic geometers than it would otherwise be.
\end{note}

We introduce what we call the ``minimal extension property'' for singular hermitian
metrics (see \parref{par:MEP}), and show that, as a consequence of the Ohsawa-Takegoshi theorem with sharp
estimates, the singular hermitian metric on $\shF$ always has this
property. We then use the minimal extension property, together with some basic
inequalities from analysis, to give an alternative proof for the following result by
Cao and P\u{a}un: when $Y$ is projective, $\shF$ is a hermitian flat bundle if and
only if the line bundle $\det \shF$ has trivial first Chern class in $H^2(Y, \RR)$.
The original argument in \cite{CP} relied on some results by Raufi about
curvature tensors for singular hermitian metrics \cite{Raufi}. We also show that when
$Y$ is projective, every nontrivial morphism of sheaves $\shF \to \OY$ is split
surjective; this result is new.

In \chapref{chap:NS}, we apply these results to construct canonical singular
hermitian metrics with semi-positive curvature on the sheaves $\fl \omXYm$ for $m
\geq 2$. Here, one small improvement over \cite{PT} is the observation that these
metrics are continuous on the Zariski-open subset of $Y$ where $f \colon X \to Y$ is
submersive.

Our discussion of generic vanishing theory in \chapref{GV} is fairly standard, but 
includes (in \parref{application}) a new result relating the structure of the
cohomological support loci $V^0(\omXm)$ for $m \geq 2$ to the Iitaka fibration of
$X$. Here the main \theoremref{t-pushforward} is one of the crucial ingredients.

\subsection{Acknowledgements}

We thank Mihai P\u{a}un for encouraging us to write this paper, and for many useful
discussions and advice about its contents. We also thank Dano Kim and Luigi Lombardi
for reading and commenting on a draft version. During the preparation of the paper, CH
was partially supported by NSF grants DMS-1300750 and DMS-1265285 and by a grant from
the Simons Foundation (Award \#256202). MP was partially supported by
NSF grant DMS-1405516 and by a Simons Fellowship. CS was partially supported by NSF
grants DMS-1404947 and DMS-1551677, and by a Centennial Fellowship from the American
Mathematical Society.

\section{Proof of the main statement}
\label{scn:main}

\subsection{Main analytic and algebraic input}\label{input}

In this section we lay out the tools neeed to prove the main result. We also give a
brief sketch of the proof, which is presented in more detail in \parref{par:proof}.
The rest of the paper will be devoted to a detailed survey of the results stated
here.

\medskip

We first note that one can reduce \theoremref{main} to the special case when  $\kappa(X) = 0$ and $Y = A$ is an abelian variety, 
with the help of the Iitaka fibration; the argument for this is recalled in \parref{par:proof} below. 
We will therefore make these assumptions in the remainder of this section. The condition $\kappa (X) = 0$ is 
equivalent to saying that $P_m(X)
\leq 1$ for all $m \in \NN$, with equality for $m$ sufficiently large and divisible.
Let $F$ be the general fiber of $f \colon X \to A$. Our goal is to prove that $\kappa(F) =
0$. What we will actually show is that $P_m(F) = 1$ whenever $P_m(X) = 1$; 
 this is enough to conclude that $\kappa(F) = 0$.

Fix now an integer $m \in \NN$ such that $P_m(X) = 1$, and consider the pushforward of the
$m$-th pluricanonical bundle
\[
	\shF_m = \fl \omXm.
\]
This is a torsion-free coherent sheaf on $A$, whose rank at the generic point of $A$ is
equal to $P_m(F)$. (In fact, this holds for every smooth fiber of $f$, by invariance
of plurigenera.) The space of global sections of $\shF_m$ has dimension
\[
	h^0\bigl( A, \fl \omXm \bigr) = \dim H^0(X, \omXm) = P_m(X) = 1.
\]
To obtain the conclusion, it is enough to show that $\shF_m$ has rank $1$ generically; we will in fact prove the 
stronger statement that $\shF_m \simeq \shO_A$. This uses both algebraic and analytic properties of 
$\shF_m$.

\medskip

\noindent
{\bf Generic vanishing and unipotency.}
We first explain the algebraic input.  We borrow an idea from generic
vanishing theory, initiated in \cite{GL1,GL2}, namely to consider the locus
\begin{align*}
	V^0(A, \shF_m) &= \menge{P \in \Pic^0(A)}{H^0(A, \shF_m \tensor P) \neq 0} \\	
		&= \menge{P \in \Pic^0(A)}{H^0 \bigl( X, \omXm \tensor \fu P \bigr) \neq 0}	
		\subseteq \Pic^0(A).
\end{align*}
The following result by Chen-Hacon \cite[\S3]{CH3}, Lai \cite[Theorem~3.5]{Lai} and Siu
\cite[Theorem~2.2]{Siu-ext} describes the
structure of $V^0(A, \shF_m)$; it is a generalization of a famous theorem
by Simpson \cite{Simpson}. The proof by Simpson (which applies when $m = 0, 1$)
relies on Hodge theory and the Gelfond-Schneider theorem in transcendence theory; the cited 
works use a construction with cyclic coverings, originally due to Viehweg, to reduce the
general case to the case $m = 1$. We review the argument in \parref{par:Lai}.

\begin{theorem} \label{thm:Lai}
Let $X$ be a smooth projective variety. For each $m \in \NN$, the locus
\[
	\menge{P \in \Pic^0(X)}{H^0 \bigl( X, \omXm \tensor P \bigr) \neq 0}
		\subseteq \Pic^0(X)
\]
is a finite union of abelian subvarieties translated by points of finite order.
\end{theorem}

This theorem implies that $V^0(A, \shF_m)$ is also a finite union of abelian
subvarieties translated by points of finite order. The reason is that, as $f \colon X \to A$ has
connected fibers, the pullback morphism $\fu \colon \Pic^0(A) \to \Pic^0(X)$ is
injective. Since we are assuming that $P_m(X) = 1$, we have $\OA \in
V^0(A, \shF_m)$; let $s_0 \in H^0(X, \omXm)$ be any nontrivial section. Now we
observe that $\kappa(X) = 0$ forces
\[
	V^0(A, \shF_m) = \{\OA\}.
\]
To see why, suppose that we had $P \in V^0(A, \shF_m)$ for some nontrivial line
bundle $P \in \Pic^0(A)$. By \theoremref{thm:Lai}, we can assume that $P$ has finite order
$d \neq 1$. Let
\[
	s_1 \in H^0(A, \shF_m \tensor P) 
		= H^0 \bigl( X, \omXm \tensor \fu P \bigr)
\]
be any nontrivial section; then $s_0^{\tensor d}$ and $s_1^{\tensor d}$ are two
linearly independent sections of $\omX^{\tensor dm}$, contradicting the fact that
$P_{dm}(X) = 1$.

Knowing the locus $V^0(A, \shF_m)$ gives us a lot of information about $\shF_m$, due
to the following result \cite[Theorem~1.10]{pluricanonical}. It is based on a
vanishing theorem for pushforwards of pluricanonical bundles, which is again proved
using Viehweg's cyclic covering construction; we review the argument in
\parref{par:GVm}.

\begin{theorem} \label{pluricanonical_GV}
Let $f \colon X \to A$ be a morphism from a smooth projective variety to an abelian
variety. For every $m \in \NN$, the sheaf $\shF_m = \fl \omXm$ is a GV-sheaf on $A$.
\end{theorem}

Recall that a coherent sheaf $\shF$ on an abelian variety $A$ is called a
\define{GV-sheaf} if its cohomology support loci
\[
	V^i(A, \shF) = \menge{P \in \Pic^0(A)}{H^i(A, \shF \tensor P) \neq 0}
\]
satisfy the inequalities $\codim V^i(A, \shF) \geq i$ for every $i \in \NN$. This property can be seen as
a variant of (semi-)positivity on abelian varieties; in fact every GV-sheaf on $A$ is nef, see 
\cite[Theorem 4.1]{PP2}. 

\begin{note}
A more conceptual description involves the Fourier-Mukai transform
\[
	\derR \Phi_P \colon \Dbcoh(\OA) \to \Dbcoh \bigl( \shO_{\Pic^0(A)} \bigr),
\]
which is an equivalence between the bounded derived categories of coherent sheaves on
$A$ and the dual abelian variety $\Pic^0(A)$. In terms of the Fourier-Mukai
transform, $\shF$ is a GV-sheaf if and only if the complex of sheaves
\[
	\derR \shHom \Bigl( \derR \Phi_P(\shF), \shO_{\Ah} \Bigr)
		\in \Dbcoh \bigl( \shO_{\Pic^0(A)} \bigr)
\]
is concentrated in degree $0$, and is therefore again a coherent sheaf $\hat{\shF}$
on $\Pic^0(A)$. By the base change theorem, the support of $\hat{\shF}$ is precisely
the locus $V^0(A, \shF)$.
\end{note}

In the case at hand, we have $V^0(A, \shF_m) = \{\OA\}$; consequently, $\hat{\shF_m}$
is a successive extension of skyscraper sheaves supported at the origin in
$\Pic^0(A)$.  We will use this via the following elementary consequence; see
\parref{par:FS-abelian} for details. Recall first from \cite{Mukai:duality} that a
vector bundle $U$ on $A$ is called \emph{unipotent} if it has a filtration $$0 = U_0
\subset U_1 \subset \cdots \subset U_n = U$$ such that $U_i / U_{i-1} \simeq \shO_A$
for all $i = 1, \ldots, n$. Note in particular that $\det U \simeq \shO_A$.  More
generally, $U$ is called \emph{homogeneous} if it has a filtration $$0 = U_0 \subset
U_1 \subset \cdots \subset U_n = U$$ such that $U_i / U_{i-1}$ is isomorphic to a
line bundle in $\Pic^0(A)$  for all $i = 1, \ldots, n$.  A homogeneous vector bundle
$U$ is called \emph{decomposable} if $U=U_1\oplus U_2$, where the $U_i$ are non-zero vector
bundles, and \emph{indecomposable} if this is not the case.

\begin{corollary}\label{pluricanonical_unipotent}
Let $X$ be a smooth projective variety with $\kappa (X) = 0$, and let $f\colon X \rightarrow A$ be an algebraic fiber space over an abelian variety.
\begin{aenumerate} 
\item If $\shF_m \neq 0$ for some $m \in \NN$, then the coherent sheaf $\shF_m$ is an
indecomposable homogeneous vector bundle.
\item If $H^0(X, \omXm) \neq 0$ for some $m \in \NN$, then the coherent sheaf
$\shF_m$ is an indecomposable unipotent vector bundle.
\end{aenumerate}
\end{corollary}

\noindent
{\bf Singular hermitian metrics on pushforwards of pluricanonical bundles.}
We now come to the analytic input.  
To motivate it, recall that the space of global sections of $\shF_m$ has dimension $P_m(X) = 1$.
In order to show that $P_m(F) = 1$, we therefore need to argue that the unipotent
vector bundle $\shF_m$ is actually the trivial bundle $\OA$. 
For the moment this seems quite hopeless with algebraic methods, so it is at this point that the analytic methods
take over.

The crucial development that allows us to proceed is recent work on the notion of a
singular hermitian metric on a torsion-free sheaf; the highlight of this study is the
following remarkable result by P\u{a}un and Takayama \cite[Theorem~3.3.5]{PT}.  In
order to state it, recall that to a singular hermitian metric $h$ on a line bundle
$L$, one associates the multiplier ideal sheaf $\mathcal{I}(h) \subseteq \shO_X$,
consisting of those functions that are locally square-integrable with respect to $h$.

\begin{theorem}\label{PT}
Let $f\colon X \to Y$ be a projective morphism of smooth varieties, and let $(L, h)$ be a line bundle on $X$ with a 
singular hermitian metric of semi-positive curvature. Then the torsion-free sheaf
$\fl \big(\omXY \otimes L \otimes \shI(h)\big)$ has a canonical singular hermitian
metric with semi-positive curvature.
\end{theorem}

The relevant definitions and the proof are described in \chapref{analytic} and
\chapref{chap:NS}, where we also present another key statement. Indeed, Cao and P\u
aun \cite[Corollary~2.9 and Theorem~5.23]{CP} show that their singular hermitian
metrics behave much like smooth metrics with Griffiths semi-positive curvature: if
the determinant line bundle $\det \shF$ has trivial first Chern class, then $\shF$ is
actually a hermitian flat bundle. This is (a) below; part (b) is new.

\begin{theorem}\label{CP}
Let $f\colon X \to Y$ be a surjective morphism of smooth projective varieties. Let
$(L, h)$ be a line bundle on $X$ with a singular hermitian metric of semi-positive
curvature, and define $\shF = \fl \big(\omXY \otimes L \otimes \shI(h)\big)$.
\begin{aenumerate}
\item If $c_1(\det \shF) = 0$ in $H^2(Y, \RR)$, then the torsion-free sheaf $\shF$ is
locally free, and
the singular hermitian metric in \theoremref{PT} is smooth and flat.
\item Every nonzero morphism $\shF \to \OY$ is split surjective.
\end{aenumerate}
\end{theorem}

The application of these results to \theoremref{main} stems from the fact that the
bundles $\shF_m$ with $m \ge 2$ naturally fit into this framework. Let us briefly
summarize how this works when $f \colon X \to Y$ is an algebraic fiber space with
general fiber $F$. For every $m \in \NN$ such that $P_m(F) \neq 0$, the spaces
of $m$-canonical forms on the smooth fibers of $f$ induce in a canonical way a
singular hermitian metric with semi-positive curvature on $\omXY$, called the
\define{$m$-th Narasimhan-Simha metric}. (For $m = 1$, the Narasimhan-Simha metric is
of course just the Hodge metric.) Denote by $h$ the induced singular hermitian
metric on the line bundle $L = \omXY^{\tensor(m-1)}$. Pretty much by construction,
the inclusion
\[
	\fl \bigl( \omXY \tensor L \tensor \shI(h) \bigr) \subseteq
		\fl \bigl( \omXY \tensor L \bigr) = \fl \omXYm
\]
is generically an isomorphism, and so \theoremref{PT} and \theoremref{CP} apply.

\begin{corollary}\label{metric_summary}
Let $f \colon X \to Y$ be an algebraic fiber space.
\begin{aenumerate}
\item For any $m \in \NN$, the torsion-free sheaf $\fl \omXYm$ has a canonical
singular hermitian metric with semi-positive curvature.
\item If $c_1(\det \fl \omXYm) = 0$ in $H^2(Y, \RR)$, then $\fl \omXYm$ is locally
free, and the
singular hermitian metric on it is smooth and flat. 
\item Every nonzero morphism $\fl \omXYm \to \OY$ is split surjective.
\end{aenumerate}
\end{corollary}

In our case, $\shF_m = \fl \omXm$ is a unipotent vector bundle by
\corollaryref{pluricanonical_unipotent}, and so the hypothesis in (b) is satisfied;
after this point, the proof of \theoremref{main} becomes straightforward.

\subsection{Proof of Theorem~\ref*{main}}\label{par:proof}
We now explain how \theoremref{main} follows quickly by combining the results outlined in the previous section. Recall that we are starting with an algebraic fiber space $f\colon X \rightarrow Y$, where $X$ is a smooth projective variety, and $Y$ is of maximal Albanese dimension. 
Let us note right away that one can perform a useful reduction, following in part the argument
in \cite[Theorem 4.9]{CH2}. 

\begin{lemma}\label{kodaira_reduction}
To prove \theoremref{main}, it is enough to assume that $\kappa (X) = 0$ and that $Y$ is an abelian variety. \end{lemma}
\begin{proof}
We begin by showing that if $\kappa (X) = - \infty$, then $\kappa (F) = - \infty$. If
this were not the case, then we could pick some $m>0$ such that $P_m(F)>0$ and hence $f_*\omega _X^{\otimes m}\ne 0$. Let $Y\to A$ be the Albanese morphism of $Y$, and 
$g\colon X\to A$ the induced morphism. Since $F$ is an irreducible component of the
general fiber of $X\to g(X)\subseteq A$, it follows that $g_*\omega _X^{\otimes m}\ne 0$. By \theoremref{pluricanonical_GV}, $g_*\omega _X^{\otimes m}$ is a GV-sheaf, and in particular by the general \lemmaref{GV_nonzero} below,  the set 
\[
	V^0(g_*\omega _X^{\otimes m})= \menge{P \in \Pic^0(A)}%
		{H^0(A, g_*\omega _X^{\otimes m}\otimes P )\ne 0}
\] 
is non-empty.
Now by \theoremref{thm:Lai} and the comments immediately after, 
$V^0(g_*\omega _X^{\otimes m})$ contains a torsion point $P \in \Pic^0(A)$, i.e.
there is an integer $k>0$ such that $P^{\otimes k}\simeq \shO _A$. But then 
$h^0(X, \omega ^{\otimes m}_X\otimes g^* P )=h^0(A, g_*\omega _X^{\otimes m}\otimes P )\ne 0$ and so 
$$h^0(X, \omega ^{\otimes km}_X )=h^0\big(X, (\omega ^{\otimes m}_X\otimes P )^{\otimes k}\big)\ne 0.$$ 
This contradicts the assumption  $\kappa (X) = - \infty$.

Assume now that $\kappa (X) \ge 0$. We will first prove the statement in the case that $\kappa (Y)=0$.
By Kawamata's theorem \cite[Theorem 1]{Kawamata:abelian},
since $Y$ is of maximal Albanese dimension, it is in fact birational to its Albanese variety and so we may assume that $Y$ is an abelian variety.
Let
$h\colon X \rightarrow Z$ the Iitaka fibration of $X$. 
 Since we are allowed to work birationally, we can assume that $Z$ is smooth.
We denote by $G$ its general fiber, so that in particular $\kappa (G) = 0$. By  the same result of Kawamata, the 
Albanese map of $G$ is surjective, so we deduce that $B = f(G) \subseteq Y$ is an abelian subvariety. If $G\to B'\to B$ is the Stein factorization, then $B'\to B$ is an \'etale map of abelian varieties.
We thus have an induced fiber space 
$$ G \longrightarrow B'$$
over an abelian variety, with $\kappa (G) = 0$, and whose general fiber is $H  = F
\cap G$. Assuming that \theoremref{main} holds for algebraic fibers spaces of Kodaira
dimension zero over abelian varieties, we obtain $\kappa (H) = 0$. Note however that $H$ is also an irreducible component of the general fiber of 
$$h_{|F} \colon F \longrightarrow h (F).$$
Considering the Stein factorization of this morphism, the easy addition formula  \cite[Corollary 2.3]{Mori}, implies that 
$$\kappa (F) \le \kappa (H) + \dim h(F) = \dim h(F).$$ 
(Note that we can assume that $g(F)$ is smooth, by passing to a birational model.)
Since  $\dim h(F) \le \dim Z = \kappa (X)$, we obtain the required inequality $\kappa (F) \le \kappa (X)$.

Finally we prove the general case. Since $Y$ has maximal Albanese dimension, after replacing it by a resolution of singularities of an \'etale cover of its Stein factorization, and 
$X$ by a resolution of the corresponding fiber product, by \cite[Theorem 13]{Kawamata:abelian} we may assume that $Y=Z\times K$ where $Z$ is of general type 
and $K$ is an abelian variety. In particular $\kappa (Y)=\dim Z=\kappa (Z)$.
If $E$ is the general fiber of the induced morphism $X\to Z$, then the induced morphism $E\to K$ has general fiber isomorphic to $F$. By what we have proven above, we deduce that $\kappa (E)\geq \kappa (F)$. We then have the required inequality
$$\kappa (X) =  \kappa (Z)+\kappa (E)\geq \kappa (Y)+\kappa (F),$$  
where the first equality is \cite[Theorem 3]{Kawamata:abelian}, since $Z$ is of general type.
\end{proof}

We may therefore assume that $f \colon X \to A$ is a fiber space onto an abelian variety, and $\kappa( X) = 0$. Note that this last condition means that we have $h^0(X, \omega_X^{\otimes m})=1$ for all sufficiently divisible integers  $m>0$. The task at hand is to show that $\kappa (F) = 0$. (It is a well known consequence of the easy addition formula \cite[Corollary 2.3]{Mori} that if $\kappa (F) = - \infty$, then $\kappa (X) = - \infty$.) 
We show in fact the following more precise statement:

\begin{theorem}\label{t-pushforward}
If $f\colon X \rightarrow A$ is an algebraic fiber space over an abelian variety,
with $\kappa (X) = 0$, then we have 
\[
	\shF_m = f_* \omega_X^{\otimes m} \simeq \shO_A
\]
for every $m \in \NN$ such that $H^0(X, \omega_X^{\otimes m}) \neq 0$.
\end{theorem}
\begin{proof}
From \corollaryref{pluricanonical_unipotent}, we know that $\shF_m$ is an indecomposable unipotent vector bundle on $A$. In particular, 
$$ \det  \shF_m  \simeq \shO_A.$$
\corollaryref{metric_summary} implies then that $\shF_m$ has a smooth hermitian metric that is flat. Thus $\shF_m$ is a 
successive extension of trivial bundles $\shO_A$  that can be split off as direct summands with the help of the flat metric. 
It follows that in fact $\shF_m\simeq \shO_A^{\oplus r}$, the trivial bundle of some rank $r \ge 1$. But then, since 
$$h^0(A, f_* \omega_X^{\otimes m} )= h^0 (X, \omega_X^{\otimes m}) = 1,$$ 
we obtain that $r =1$, which is the statement of the theorem.
\end{proof}

In the remaining chapters, we explain the material in \parref{input} in more detail.

\section{Generic vanishing}
\label{GV}

\subsection{Canonical bundles and their pushforwards}

As explained above, the algebraic tools used in this paper revolve around the topic of generic vanishing. This study was initiated by Green and Lazarsfeld 
\cite{GL1, GL2}, in part as an attempt to provide a useful weaker version of Kodaira Vanishing for the canonical bundle, in the absence of twists by positive line bundles. 
An important addition was provided in work of Simpson \cite{Simpson}.
The results of Green-Lazarsfeld were extended to include higher direct images of canonical bundles in \cite{Hacon}. From the point of view of this paper, the main statements to keep in mind are summarized in the following theorem. 
Recall that for any coherent sheaf $\shF$ on an abelian variety $A$, we consider for all $k \ge 0$ 
the \emph{cohomological support loci}
$$V^k (\shF) = \{ \, P \in \Pic^0(A) \, \mid \, H^k(X, \shF \otimes P) \neq
0 \, \}$$
They are closed subsets of $\Pic^0 (A)$, by the semi-continuity theorem for cohomology.

\begin{theorem}
If $f\colon X \to A$ is a morphism from a smooth projective variety to an abelian variety, then for any $j, k \ge 0$ 
we have
\begin{enumerate}
\item \cite{Hacon} $\codim_{\Pic^0 (A)}  V^k (R^j f_* \omega_X) \geq k$.
\item \cite{GL2,Simpson}
 Every irreducible component of $V^k (R^j f_* \omega_X)$ is a translate of an abelian subvariety of
$A$ by a point of finite order.
\end{enumerate}
\end{theorem}

What we use in this paper are (partial) extensions of these results to pushforwards
of pluricanonical bundles $f_* \omega_X^{\otimes m}$, for $m \ge 2$. We describe
these  in the following sections, beginning with an abstract study in  the next.

\subsection{The $GV$ property and unipotency}\label{GV-background}
Let $A$ be an abelian variety of dimension $g$.  The generic vanishing property (1) in the theorem above
can be formalized  into the following:

\begin{definition}\label{GV_definition}
The sheaf $\shF$ is called a \define{GV-sheaf} on $A$ if it satisfies
\[
	\codim_{\Pic^0 (A)} V^k (\shF) \geq k \qquad \text{for all $k\geq 0.$}
\]
\end{definition}

We will identify $\Pic^0 (A)$ with the dual abelian variety $\widehat A$, and denote by $P$ the normalized Poincar\'e bundle on $A \times \widehat A$. It induces the integral transforms
\[
	\derR \Phi_P \colon  \Dbcoh(\OA) \longrightarrow \Dbcoh \bigl( \shO_{\widehat A} \bigr), 
 	\quad
	\derR \Phi_P \shF = \derR {p_2}_* (p_1^* \shF \otimes P).
\]
and
\[
	\derR \Psi_P \colon \Dbcoh \bigl( \shO_{\widehat A} \bigr) \longrightarrow  \Dbcoh(\OA), 
		\quad \derR \Psi_P \shG = \derR {p_1}_* (p_2^* \shG \otimes P).
\]
These functors are known from \cite[Theorem 2.2]{Mukai:duality} to be equivalences of derived categories, usually called the Fourier-Mukai transforms;  moreover, 
\begin{equation}\label{FM-formulas}
\derR \Psi_P \circ \derR \Phi_P=(-1_A)^*[-g] \qquad \text{and}
\qquad  \derR \Phi_P \circ \derR \Psi_P=(-1_{\widehat A})^*[-g],
\end{equation}
where $[-g]$ denotes shifting $g$ places to the right.

Standard applications of base change (see e.g. \cite[Lemma 2.1]{PP2} and \cite[Proposition 3.14]{Pareschi+Popa:GV}) 
lead to the following basic properties of $GV$-sheaves:

\begin{lemma}\label{inclusion_chain}
Let $\shF$ be a coherent sheaf on $A$. Then:
\begin{enumerate}
\item
$\shF$  is a GV-sheaf if and only if 
$$\codim_{\widehat A} \Supp R^k \Phi_P \shF \ge k  \,\,\,\,\,\,{\rm  for~ all}\,\,\,\, k \geq 0.$$
\item
If $\shF$ is a GV-sheaf, then 
$$V^g (\shF) \subseteq \cdots \subseteq V^1 (\shF) \subseteq V^0 (\shF).$$
\end{enumerate}
\end{lemma}

To give a sense of what is going on,  here is a sketch of the proof of part (1): note
that the restriction of ${p_1}^*\shF \otimes P$ to a fiber $A \times \{\alpha\}$ of
$p_2$ is isomorphic to the sheaf $\shF \otimes \alpha$ on $A$, and so fiberwise we
are looking at the cohomology groups $H^k (A, \shF\otimes \alpha)$. A simple
application of the theorem on cohomology and base change then shows for every $m \ge
0$ that $$\bigcup_{k \ge m} \Supp R^k \Phi_P \shF =  \bigcup_{k \ge m} V^k
(\shF).$$ This implies the result by descending induction on $k$.

\begin{lemma}\label{GV_nonzero}
If $\shF$ is a GV-sheaf on $A$, then $\shF = 0$ if and only if $V^0 (\shF) =
\emptyset$.
\end{lemma} 
\begin{proof}
By \lemmaref{inclusion_chain}, we see that $V^0 (\shF) = \emptyset$ is equivalent to $V^k (\shF) = \emptyset$ for all $k\geq 0$, which 
by base change is in turn equivalent to $\derR \Phi_P \shF = 0$. By Mukai's derived equivalence, this is equivalent to
$\shF = 0$.
\end{proof}

The following proposition is the same as \cite[Corollary 3.2(4)]{Hacon}, since it can be seen that the assumption
on $\shF$  imposed there is equivalent to that of being a GV-sheaf. This is the main way in which generic vanishing is used in this paper;  for the definition of a unipotent vector bundle see \S \ref{input}.

\begin{proposition}\label{unipotent}
Let $\shF$ be a GV-sheaf on an abelian variety $A$. If $V^0 (\shF) = \{0\}$, 
then $\shF$ is a unipotent vector bundle.
\end{proposition}
\begin{proof}
By \cite[Example 2.9]{Mukai:duality}, if $g = \dim A$, then $\shF$ is a unipotent vector bundle if and only if 
\begin{equation}\label{unipotent-conditions}
R^i \Phi_P \shF = 0 \quad \text{for all $i\neq g$,} \qquad \text{and} \qquad
	R^g \Phi_P \shF = \shG,
\end{equation}
where $\shG$ is a coherent sheaf supported at the origin $0 \in \widehat A$. To
review the argument, notice that if this is the case, then if
$l=\operatorname{length}(\shG)>0$, we have $h^0(\widehat A, \shG)\ne 0$ and so there is a short exact sequence 
$$0\longrightarrow k(0)\longrightarrow \shG\longrightarrow \shG' \longrightarrow 0$$ 
where $\shG'$ is  a coherent sheaf supported at the origin $0 \in \widehat A$, with
$\operatorname{length}(\shG ')=l-1$. 
Applying $\derR \Psi_P$ we obtain a short exact sequence of vector bundles on $A$
$$0\longrightarrow \shO _A\longrightarrow R ^0\Psi_P \shG \longrightarrow R ^0\Psi_P \shG' \longrightarrow 0,$$
and by ($\ref{FM-formulas}$) we have $R ^0\Psi_P \shG = (-1_A)^* \shF$. 
It is then not hard to see that $\shF^\prime = R ^0\Psi_P \shG'$ also satisfies the hypotheses in ($\ref{unipotent-conditions}$) and so,
proceeding by induction on $l$, we may assume that $\shF^\prime$ is a unipotent vector bundle. It follows that 
$\shF$ is also a unipotent vector bundle as well (since it is an extension of a unipotent vector bundle by $ \shO _A$). 

We now check that the two conditions in ($\ref{unipotent-conditions}$) are satisfied.
 By \lemmaref{inclusion_chain} the hypothesis implies that 
\[
	V^i (\shF) \subseteq \{0\} \qquad \text{for all $i \geq 0$}.
\]
By base change one obtains that $R^i \Phi_P \shF$ is supported at most at $0 \in \widehat A$ for $0\leq i\leq g$. 
It remains to show that $R^i \Phi_P \shF = 0$ for $i\ne g$. Note that 
\[
H^j( \widehat A, R^i \Phi_P \shF\otimes \alpha )=0 \qquad
		\text{for all $j>0$, $0\leq i\leq g$, and $\alpha \in \Pic^0(\widehat A)$,}
\]
and so by base change we have
\[
	R^j\Psi_P  (R ^i \Phi_P \shF)=0 \qquad
		\text{for all $j>0$ and $0\leq i\leq g$.}
\]
By an easy argument involving the spectral sequence of the composition of two functors, since $\derR \Psi_P \circ \derR \Phi_P=(-1_A)^*[-g]$, it then follows that 
$R ^0\Psi_P  (R ^i \Phi_P \shF)=\mathcal H^i\big( (-1_A)^*\shF [-g]\big)$, and so in particular
$$R^0\Psi_P  (R ^i \Phi_P \shF)=0 \,\,\,\, {\rm for} \,\,\,\, i < g.$$ 
But then $\derR \Psi_P  (R ^i \Phi_P \shF)=0$ for $i < g$, and hence 
$R ^i \Phi_P \shF=0$ for $i < g$.

\end{proof}

For later use, we note that a very useful tool for detecting generic vanishing is a
cohomological criterion introduced in \cite[Corollary~3.1]{Hacon}. Before stating it, we recall that an ample line bundle $N$ on an abelian variety $B$ induces an isogeny
$$\varphi_N: B \longrightarrow \widehat B, \,\,\,\,\,\,x \to t_x^* N \otimes N^{-1},$$
where $t_x$ denotes translation by $x \in B$.

\begin{theorem} \label{vanishing_GV}
A coherent sheaf $\shF$ on $A$ is a GV-sheaf if and only if given any sufficiently large power $M$ of an ample line bundle on $\widehat A$, one has 
$$H^i \big(A, \shF \otimes R^g \Psi_P (M^{-1}) \big) = 0 \,\,\,\,\,\, {\rm for ~all} \,\,\,\, i >0.$$
If $\varphi_M\colon \widehat A \to A$ is the isogeny induced by $M$, this is also equivalent to
$$H^i (\widehat A, \varphi_M^*  \shF  \otimes M) = 0 \,\,\,\,\,\, {\rm for~ all}  \,\,\,\, i >0.$$
\end{theorem}

\begin{remark}
Note that since $M$ is ample, $H^i(\widehat A, M^{-1}\otimes \alpha )=0$ for all $i<g$ and $\alpha \in \Pic^0(\widehat A)\simeq A$, and so 
$R^i\Psi_P (M^{-1})  = 0$ for $i\ne g$. If we denote $R^g \Psi_P (M^{-1}) =\widehat {M^{-1}}$, then by \cite[Proposition  3.11]{Mukai:duality} we have $\varphi_M^* \widehat {M^{-1}}\simeq M^{\oplus h^0(M)}$, 
hence the second assertion.
\end{remark}

\subsection{Pushforwards of pluricanonical bundles}
\label{par:GVm}

In this section we explain the proof of \theoremref{pluricanonical_GV}, following
\cite[\S5]{pluricanonical}. In \emph{loc. cit.} we noted that a very quick proof can
be given based on the general effective vanishing theorem for pushforwards of
pluricanonical bundles proved there. However, another more self-contained, if less
efficient, proof using cyclic covering constructions is also given; we choose to
explain this here, as cyclic covering constructions are in the background of  many
arguments in this article. We first recall Koll\'ar's vanishing theorem
\cite[Theorem~2.1]{Kollar}.

\begin{theorem} \label{kollar_vanishing}
Let $f: X \rightarrow Y$ be a morphism of projective varieties, with $X$ smooth. If $L$ is an ample line bundle 
on $Y$, then
$$H^j (Y, R^i f_* \omega_X\otimes L) = 0 \,\,\,\,{\rm for~all~} i {\rm ~and ~all~} j>0.$$ 
\end{theorem}

\medskip

\begin{proof}[Proof of \theoremref{pluricanonical_GV}]
Let $M = L^{\otimes d}$, where $L$ is an  ample and globally generated line bundle on $\widehat A$, and $d$ is an integer that can be chosen arbitrarily large. Let 
$\varphi_M\colon \widehat A \to A$ be the isogeny induced by $M$. According to \theoremref{vanishing_GV}, it is 
enough to show that 
$$H^i (\widehat A, \varphi_M^*  f_*\omega_X^{\otimes m} \otimes M) = 0 \quad
\text{for all $i >0$.}$$
Equivalently, we need to show that 
$$H^i (\widehat A, h_* \omega_{X_1}^{\otimes m} \otimes L^{\otimes d}) = 0  \quad
\text{for all $i >0$,}$$
where $h\colon X_1 \to \widehat A$ is the base change of $f\colon X \to A $ via $\varphi_M$, as in the diagram
\[
\begin{tikzcd}
X_1  \dar{h}  \rar & X \dar{f} \\
 \widehat A \rar{\varphi_M}	& A	
\end{tikzcd}
\]
  
We can conclude immediately if we know that there exists a bound $d = d (g, m)$, i.e. depending only on $g = \dim A$ 
and $m$, such that the vanishing in question holds for any morphism $h$. (Note that we cannot apply Serre Vanishing here, as 
construction depends on the original choice of $M$.)
But \propositionref{summand} below shows that there exists a morphism  
$\varphi: Z \rightarrow \widehat A$ with $Z$ smooth projective, and $k \le g + m$, such that $h_* \omega_{X_1}^{\otimes m} \otimes L^{\otimes k(m-1)}$ is a direct summand of $\varphi_* \omega_Z$. Applying Koll\'ar vanishing, \theoremref{kollar_vanishing}, we deduce that 
$$H^i (\widehat A, h_*\omega_{X_1}^{\otimes m} \otimes L^{\otimes d}) = 0 \,\,\,\, {\rm for~ all} \,\,\,\, i
>0 \,\,\,\,{\rm and ~all} \,\,\,\, d \ge (g+m)(m -1) +1,$$
which finishes the proof. (The main result of \cite{pluricanonical} shows that one can in fact take $d \ge m (g+1) - g$.)
\end{proof}

\begin{proposition}\label{summand}
Let $f: X \rightarrow Y$ be a morphism of projective varieties, with $X$ smooth and $Y$ of dimension $n$.
Let $L$ be an ample and globally generated line bundle on $Y$, and $m\ge 1$ an integer. Then
there exists a smooth projective variety $Z$ with a morphism $\varphi: Z \rightarrow Y$, and an integer $0 \le k\le n+m$, such that $f_* \omega_X^{\otimes m} \otimes L^{\otimes k(m-1)}$ is a direct summand in $\varphi_* \omega_Z$.
\end{proposition}
\begin{proof}
The sheaf $f_*\omega_X^{\otimes m} \otimes L^{\otimes pm}$ is globally generated for some sufficiently 
large $p$. Denote by $k$ the minimal $p\ge 0$ for which this is satisfied. 

Consider now the adjunction morphism
\[
	\fu \fl \omX^{\otimes m} \to \omX^{\otimes m}.
\]
After blowing up $X$, if necessary, we can assume that the image sheaf is of the
form $\omX^{\otimes m} \tensor \OX(-E)$ for a divisor $E$ with normal crossing support. As
$\fl \omX^{\otimes m} \tensor L^{\otimes km}$ is globally generated, the line bundle
\[
	\omX^{\otimes m} \tensor \fu L^{\otimes km} \tensor \OX(-E)
\]
is globally generated as well. It is therefore isomorphic to $\OX(D)$, where $D$ is an
irreducible smooth divisor, not contained in the support of $E$, such that $D + E$
also has normal crossings. We have thus arranged that
\[
	(\omX \tensor \fu L^{\otimes k})^{\otimes m} \simeq \OX(D+E).
\] 
We can now take the associated covering of $X$ of degree $m$, branched along $D + E$,  and resolve its singularities. This gives us a generically finite morphism $g \colon Z \to X$ of degree $m$, and we denote $\varphi = f\circ g: Z \rightarrow Y$.

By a well-known calculation of Esnault and Viehweg, see e.g. \cite[Lemma 2.3]{EV}, the direct
image $g_* \omega_Z$ contains  the sheaf
\begin{align*}
	\omX \tensor \bigl( \omX \tensor \fu L^{\otimes k} \bigr)^{\otimes m-1} \tensor 
		&\OX \left(- \left\lfloor \tfrac{m-1}{m} \bigl( D+E \bigr) \right\rfloor \right) \\
	\simeq \omX^{\otimes m} \tensor \fu L^{\otimes k(m-1)} 
		\tensor &\OX \left(- \left\lfloor \tfrac{m-1}{m} E \right\rfloor \right)
\end{align*}
 as a direct summand. If we now apply $\fl$, we find that
$$
	\fl \Bigl( \omX^{\otimes m} \tensor \OX \left(- \left\lfloor \tfrac{m-1}{m} E \right\rfloor
		\right) \Bigr) \tensor L^{\otimes k(m-1)}
$$
is a direct summand of $\varphi_{\ast} \omega_Z$. Finally, $E$ is the relative base
locus of $\omega_X^{\otimes m}$, and so
\[
	\fl \Bigl( \omX^{\otimes m} \tensor \OX \left(- \left\lfloor \tfrac{m-1}{m} E \right\rfloor
		\right) \Bigr) \simeq \fl \omX^{\otimes m}.
\]
In other words, $f_* \omega_X^{\otimes m} \otimes L^{\otimes k(m-1)}$ is a direct summand in 
$\varphi_* \omega_Z$. By \theoremref{kollar_vanishing}, the sheaf $f_*
\omega_X^{\otimes m} \otimes L^{\otimes k(m-1) + n+1}$ 
is $0$-regular in the sense of Castelnuovo-Mumford, and hence globally generated.\footnote{Recall that a sheaf $\shF$ on $Y$ is \emph{$0$-regular}
with respect to an ample and globally generated line bundle $L$ if
$$H^i (Y, \shF \otimes L^{\otimes -i} ) = 0 \,\,\,\,{\rm ~for~all} \,\,\,\, i > 0.$$ 
The Castelnuovo-Mumford Lemma says that every $0$-regular sheaf is globally generated.}
By our minimal choice of $k$, this is only possible if
\[
	k(m-1)+ n + 1 \geq (k-1) m+1,
\]
which is equivalent to $k \leq n + m$. 
\end{proof}

\subsection{Fiber spaces over abelian varieties}
\label{par:FS-abelian}

Let $f \colon X \rightarrow A$ be a fiber space over an abelian variety. For simplicity, for each $m \ge 0$ we denote 
$$\shF_m  = f_* \omega_X^{\otimes m}.$$
Note that $\shF _0=\shO _A$.
Though this is not really necessary for the argument, we first remark that we can be precise about the values of 
$m \ge 1$ for which $\shF_m \neq 0$.

\begin{lemma}\label{nonzero}
We have $\shF_m \neq 0$ if and only if there exists $P \in \Pic^0 (A)$ such that $H^0 (X, \omega_X^{\otimes m}
\otimes f^* P) \neq 0$.
\end{lemma}
\begin{proof} 
By \theoremref{pluricanonical_GV} we know that $\shF_m$ is a GV-sheaf on $A$ for all $m \ge 1$. 
We conclude from \lemmaref{GV_nonzero} that  $\shF_m \neq 0$ if and only if $V^0 (\shF_m) \neq \emptyset$, which 
by the projection formula is precisely the statement of the lemma.
\end{proof}

The purpose of this subsection is to give the 

\begin{proof}[Proof of \corollaryref{pluricanonical_unipotent}]
We will only prove the second statement, since the first one is similar.
We fix an $m$ such that $H^0(A,\shF_m)=H^0 (X, \omega_X^{\otimes m}) \neq 0$. In particular $\shF_m$ is a non-trivial  GV-sheaf on $A$. Since $\kappa (X) = 0$, 
we have $h^0 (A, \shF_m) = 1$, and in particular $0 \in V^0 (\shF_m)$. We claim that
$$V^0 (\shF_m) = \{ 0 \},$$
which implies that $\shF_m$ is unipotent by \propositionref{unipotent}.

To see this, note first that by  \theoremref{thm:Lai} and the comments immediately after, 
$V^0 (\shF_m)$ is a  union of torsion translates of abelian subvarieties of $\Pic^0 (A)$. 
Then, proceeding as in \cite[Lemma 2.1]{CH1}, if there were two distinct points 
$P, Q \in V^0 (\shF_m)$ we could assume that they are both torsion of the same order $k$. Since $f$ is a fiber space, the mapping
$$f^* \colon \Pic^0 (A) \longrightarrow \Pic^0 (X)$$
is injective, and so $f^*P$ and $f^*Q$ are distinct as well.  Now if $P \in V^0 (\shF_m)$, then 
$$H^0 (X, \omega_X^{\otimes m} \otimes f^*P) \simeq H^0(A,\shF_m \otimes P)\neq 0,$$ 
and similarly for $Q$. Let $D\in |mK_X+f^*P|$ and $G\in |mK_X+f^*Q|$, so that  
$kD,kG\in |mkK_X|$. Since $h^0(X, \omega_X^{\otimes mk})=1$, 
it follows that $kD=kF$, and hence $f^*P=f^*Q$. This is the required  contradiction.

Finally, since $h^0(A, \shF_m)=1$, it is clear that $\shF_m$ is indecomposable.
\end{proof}

\subsection{Cohomological support loci for pluricanonical bundles} 
\label{par:Lai}

In this section we explain an important ingredient used in \corollaryref{pluricanonical_unipotent}, namely \theoremref{thm:Lai}, the analogue of Simpson's theorem for the 
$0$-th cohomological support locus of a pluricanonical bundle. 
We give a slight generalization, emphasizing again the ubiquitous cyclic covering trick.

For a coherent sheaf $\shF$ on an abelian variety $A$, for each $k \ge 1$ we consider the following refinement of $V^0 (\shF)$, namely 
$$V^0_k (\shF) = \{ \, P \in \Pic^0(A) \, \mid \, h^ 0(X, \shF \otimes P) \ge k \, \}.$$

\begin{theorem}\label{torsion_pluricanonical}
Let $f: X \rightarrow A$ be a morphism from a smooth projective variety to an abelian variety, and fix integers $m, k \ge 1$. 
Then every irreducible component of $V^0_k (f_* \omega_X^{\otimes m})$ is a torsion subvariety, i.e. a translate of an abelian subvariety 
of $\Pic^0 (A)$ by a torsion point.
\end{theorem}

To prove \theoremref{torsion_pluricanonical}, we first collect a few lemmas. 
Given a smooth projective variety $X$, and a line bundle $L$ on $X$ with
$\kappa(L) \geq 0$, recall that the \emph{asymptotic multiplier ideal} of $L$ is defined as
\[
	\shI \bigl( \norm{L} \bigr) = 
		 \shI \Bigl( \tfrac{1}{p} D \Bigr) \subseteq \shO_X,
\]
where $p$ is any sufficiently large and divisible integer, $D$ is the divisor of
a general section in $H^0(X, L^{\otimes p})$, and the ideal sheaf on the right is the
multiplier ideal of the $\QQ$-divisor $\frac{1}{p}D$; see \cite[Ch. 11]{Lazarsfeld}.
It is easy to see that the ideal sheaf $\shI \bigl( \norm{L} \bigr)$ is
independent of the choice of $p$ and $D$.  Further properties of asymptotic
multiplier ideals appear in the proof of \lemmaref{lem:Lai} below.

\begin{lemma} \label{lem:EV}
There exists a morphism $g\colon Y \rightarrow X$ with $Y$ smooth and projective, such 
that the sheaf $\omX \tensor L \tensor \shI \bigl( \norm{L} \bigr) $ is a direct
summand of $g_* \omega_Y$.
\end{lemma}

\begin{proof}
Take $D$ as above, and let $\mu \colon X' \to X$ be a log resolution of $(X,D)$
such that $X'$ is smooth and $\mu ^* D$ plus the exceptional divisor of $\mu$ is a
divisor with simple normal crossing support. Then
\[
	\mu^* L^{\otimes p} = \shO_{X'}(\mu^* D),
\]
and we let $f \colon Y \to X'$ be a resolution of singularities of the degree $p$
branched covering of $X'$ defined by $\mu^* D$. According to the calculation of
Esnault and Viehweg recalled in the proof of \propositionref{summand}, $\fl \omY$
contains as a direct summand the sheaf
\[
	\omega_{X'} \tensor \mu^* L \tensor 
		\shO_{X'} \Bigl(- \bigl\lfloor \tfrac{1}{p} \mu^* D \bigr\rfloor \Bigr)
	\simeq \mu^* \bigl( \omX \tensor L \bigr) \tensor \shO_{X'} 
		\Bigl( K_{X'/X} - \bigl\lfloor \tfrac{1}{p} \mu^* D \bigr\rfloor \Bigr),
\]
and so $\mu_* \fl \omY$ contains as a direct summand the sheaf
\[
	\omX \tensor L \tensor \mu_* \shO_{X'}
		\Bigl( K_{X'/X} - \bigl\lfloor \tfrac{1}{p} \mu^* D \bigr\rfloor \Bigr)	
	= \omX \tensor L \tensor \shI \bigl( \norm{L} \bigr).
\]
We can therefore take $g = \mu \circ f$.
\end{proof}

\begin{lemma} \label{lem:summand}
If $\shF$ and $\shG$ are two coherent sheaves on an abelian variety $A$, and $\shF$
is a direct summand of $\shG$, then every irreducible component of $V_k^0 (\shF)$ is
also an irreducible component of $V_{\ell}^0 ( \shG)$ for some $\ell \geq k$.
\end{lemma}

\begin{proof}
Let $Z \subseteq V_k^0 (\shF)$ be an irreducible component. We can assume without
loss of generality that $k = \min \menge{ h^0(X, \shF \tensor \alpha)}{\alpha \in
Z}$. By assumption, we have
a decomposition $\shG \simeq \shF \oplus \shF'$. We define
\[
	\ell = k + \min \menge{ h^0(X, \shF' \tensor \alpha)}{\alpha \in Z}.
\]
By the semicontinuity of $h^0(A, \shF '\otimes \alpha )$ and $h^0(A, \shF \otimes \alpha )$,  
it follows that there is a neighborhood $U$ of the generic point of $Z$ such that 
$h^0(\shF '\otimes \alpha )\leq \ell -k$  and $h^0(\shF \otimes \alpha )\leq k$ for
any $\alpha \in U$. Since $h^0(\shF \otimes \alpha )< k$ for any $\alpha \in U
\setminus (U\cap Z)$ it is easy to see that $Z$ is an irreducible component of
$V_{\ell}^0 ( \shG)$.
\end{proof}

\begin{lemma} \label{lem:Lai}
If $V^0_k (f_* \omega_X^{\otimes m})$ contains a point, then it also contains a torsion subvariety
through that point. 
\end{lemma}

\begin{proof}
Take any point in $V^0_k (f_* \omega_X^{\otimes m})$. Since $\Pic^0(X)$ is divisible, we may assume
that our point is of the form $L_0^{\otimes m}$ for some $L_0 \in \Pic^0(X)$. This means that
\[
	h^0 \bigl( X,   \omega_X^{\otimes m} \tensor f^* L_0^{\otimes m} \bigr) \geq k.
\]
For $r \geq 0$, set $\shI_r = \shI \bigl( \norm{\omX^{\otimes r} \tensor f^*
L_0^{\otimes r}} \bigr)$. According to \lemmaref{lem:EV}, there exists a morphism $g
\colon Y \rightarrow X$ such that 
\[
	\omX \tensor (\omX \tensor f^* L_0)^{\otimes (m-1)} \tensor \mathcal{I}_{m-1} = 
	\omX^{\otimes m} \tensor L_0^{\otimes (m-1)} \tensor \mathcal{I}_{m-1}
\]
is a direct summand of $g_* \omega_Y$. Consequently, $f_* (\omX^{\otimes m} \tensor L_0^{m-1} \tensor \mathcal{I}_{m-1})$ is a 
direct summand of $h_* \omega_Y$, where $h = f\circ g\colon Y \rightarrow A$. 
By Simpson's theorem we know that, for any $\ell$, every irreducible component of $V^0_{\ell} (h_* \omega_Y)$ is a
torsion subvariety. Together with \lemmaref{lem:summand}, this shows that every irreducible component of 
\[
	V_k^0 \bigl( f_* (\omX^{\otimes m} \tensor L_0^{\otimes (m-1)} \tensor \mathcal{I}_{m-1}) \bigr)
\]
is a torsion subvariety. We observe that this set contains $L_0$: the reason is that since
$ \mathcal{I}_m\subseteq  \mathcal{I}_{m-1}$ (see \cite[Theorem 11.1.8]{Lazarsfeld}), we have
\begin{align*}
	H^0 \bigl( X, (\omX \tensor f^*L_0)^{\otimes m} \tensor \mathcal{I}_m \bigr) &\subseteq
	H^0 \bigl( X, (\omX \tensor f^*L_0)^{\otimes m} \tensor \mathcal{I}_{m-1} \bigr) \\
	&\subseteq H^0 \bigl( X, (\omX \tensor f^*L_0)^{\otimes m}\bigr),
\end{align*}
and the two spaces on the outside are equal because the subscheme defined by $\mathcal{I}_m$
is contained in the base locus of $(\omX \tensor f^*L_0)^{\otimes m}$ by
\cite[Theorem 11.1.8]{Lazarsfeld}.

Now let $W$ be an irreducible component of $V_k^0 \bigl( f_* (\omX^{\otimes m} \tensor L_0^{\otimes (m-1)}
\tensor \mathcal{I}_{m-1}) \bigr)$ passing through the point $L_0$. For every $L \in W$, we
have 
\[
	h^0 \bigl( X, \omX^{\otimes m} \tensor L_0^{\otimes (m-1)} \tensor L \bigr) \geq
	h^0 \bigl( X, \omX^{\otimes m} \tensor L_0^{\otimes (m-1)} \tensor \mathcal{I}_{m-1} \tensor L \bigr)
	\geq k,
\]
and so $L_0^{\otimes (m-1)} \tensor W \subseteq V_k^0 ( f_*\omX^{\otimes m})$. As noted above, $L_0^{\otimes (m-1)}
\tensor W$ contains the point $L_0^{\otimes m}$; it is also a torsion subvariety, because $W$
is a torsion subvariety and $L_0 \in W$.
\end{proof}

\begin{proof}[Proof of \theoremref{torsion_pluricanonical}]
Let $Z \subseteq V^0_k (f_* \omega_X^{\otimes m})$
be an irreducible component; we have to show that $Z$ is a torsion
subvariety. In case $Z$ is a point, this follows directly from \lemmaref{lem:Lai}, so
let us assume that $\dim Z \geq 1$. Let $Z_0 \subseteq Z$ denote the Zariski-open
subset obtained by removing the intersection with the other irreducible components of
$V^0_k (f_* \omega_X^{\otimes m})$. Then again by \lemmaref{lem:Lai}, every point of $Z_0$ lies on a
torsion subvariety that is contained in $Z$. Because there are only countably many
torsion subvarieties in $\Pic^0(X)$, Baire's theorem implies that $Z$ itself must be a
torsion subvariety.  
\end{proof}

\subsection{Iitaka fibration and cohomological support loci}
\label{application}
In this section, we use \theoremref{t-pushforward} to give a precise description of
the cohomological support loci
\[
	V^0(\omega _X ^{\otimes m}) = \menge{P \in \Pic^0(X)}%
		{H^0(X, \omXm \tensor P) \neq 0}
\]
for all $m\geq 2$, in terms of the Iitaka fibration of $X$. After a birational
modification of $X$, the Iitaka fibration can be realized as a morphism $f \colon X
\to Y$, where $Y$ is smooth projective of dimension $\kappa(X)$. By the universal
property of the Albanese mapping, we obtain a commutative diagram
\[
\begin{tikzcd}
X \rar{a_X} \dar{f} & A_X \dar{a_f} \\
Y \rar{a_Y} & A_Y
\end{tikzcd}
\]
where $A_X = \Alb(X)$ and $A_Y = \Alb(Y)$ are the two Albanese varieties. The following
simple lemma appears in \cite[Lemma~2.6]{Chen+Hacon:Iitaka}.

\begin{lemma}\label{l-iitaka} 
With notation as above, the following things are true:
\begin{aenumerate} 
\item The homomorphism $a_f$ is surjective with connected fibers.
\item Setting  $K = \ker(a_f)$, we have a short exact sequence
$$0\to \Pic^0(Y)\to \Pic^0(X)\to \Pic^0(K)\to 0.$$
\item If $F$ is a general fiber of $f$, then the kernel of the natural homomorphism
$\Pic^0(X)\to \Pic^0(F)$ is a finite union of torsion translates of  $\Pic^0(Y)$.
\end{aenumerate}
\end{lemma}

Using this lemma and the results of Green-Lazarsfeld \cite{GL1,GL2} and Simpson
\cite{Simpson}, one can prove the following results about the locus $V^0(\omX)$:
\begin{enumerate}
\item There are finitely many quotient abelian varieties $\Alb(X) \to B_i$ and
finitely many torsion points $\alpha_i \in \Pic^0(X)$ such that
\[
	V^0(\omega _X) = \bigcup_{i=1}^n \bigl( \alpha_i + \Pic^0(B_i) \bigr).
\]
This is proved in \cite[Theorem~0.1]{GL2} and \cite{Simpson}. Note that
$V^0(\omega_X)$ may be empty; in that case, we take $n = 0$.
\item We have $\Pic^0(B_i) \subseteq \Pic^0(Y)$, where $f \colon X\to Y$ is the Iitaka
fibration; when $X$ is of maximal Albanese dimension, then the union of the
$\Pic^0(B_i)$ generates $\Pic^0(Y)$. This is proved in
\cite{Chen+Hacon:Iitaka,Chen+Hacon:Pluricanonical}.  
\item At a general point $P$ of the $i$-th irreducible component $\alpha_i +
\Pic^0(B_i)$, one has $s \cup v = 0$ for all $s \in H^0(X, \omX \tensor P)$ and all $v \in H^1(B_i, \shO_{B_i})$;
conversely, if $s \in H^0(X, \omX \tensor P)$ is nonzero and $s \cup v = 0$ for some $v
\in H^1(X, \OX)$, then necessarily $v \in H^1(B_i, \shO_{B_i})$.
\end{enumerate}

\begin{note}
One can interpret property (3) as follows. If $P\in \Pic^0(X)$ is a general point of
a component of $V^0(\omX)$, and  we identify the tangent space to $\Pic^0(X)$ at
the point $P$ with the vector space $H^1(X, \OX)$, then $s \cup v=0\in H^1(X,\omega
_X\otimes P)$ if and only if $s$ deforms to first order in the direction of $v$.
Property (3) then says that if $s$ deforms to first order in the direction of $v$,
then it deforms to arbitrary order.
\end{note}

It turns out that the cohomology support loci $V^0(\omXm)$ for $m \geq 2$ are
governed by the Iitaka fibration $f \colon X \to Y$: in contrast to the case $m = 1$,
every irreducible component is now simply a torsion translate of $\Pic^0(Y)$.

\begin{theorem} Let $X$ be a smooth complex projective variety, and let $F$ be a
general fiber of the Iitaka fibration $f \colon X \to Y$. Let $m \geq 2$.
\begin{aenumerate}
\item For every torsion point $\alpha \in \Pic^0(X)$, and every $\beta \in
\Pic^0(Y)$, we have
\[
	h^0 (X, \omega_X^{\otimes m} \otimes \alpha) = 
		h^0 (X, \omega_X^{\otimes m} \otimes \alpha \otimes \fu \beta).
\]
\item There exist finitely many torsion points $\alpha_1,\ldots , \alpha_n\in
\Pic^0(X)$ such that 
\[
	V^0(\omega _X ^{\otimes m}) 
		= \bigcup_{i=1}^n \bigl( \alpha _i+\Pic^0(Y) \bigr).
\]
\item At every point $\alpha \in V^0(\omega _X ^{\otimes m})$, one has $s \cup v = 0$
for all $s\in H^0(X,\omega _X^{\otimes m}\otimes \alpha)$ and all $v
\in H^1(Y, \OY)$; conversely, if $s \in H^0(X, \omXm \tensor \alpha)$ is nonzero and
$s \cup v = 0$ for some $v\in  H^1(X, \OX)$, then necessarily $v \in H^1(Y, \OY)$.
\end{aenumerate}
\end{theorem}
\begin{proof}
We begin by proving (a), following Jiang's version \cite[Lemma~3.2]{Jiang} of
the original argument in \cite[Proposition~2.12]{HP}.  Let $g\colon X\to A_Y$ be the
morphism induced by composing $f$ with the Albanese map of $Y$. 
\[
\begin{tikzcd}
X \rar{a_X} \dar{f} \arrow{dr}{g} & A_X \dar{a_f} \\
Y \rar{a_Y} & A_Y
\end{tikzcd}
\]
Let $H$ be an ample divisor on $A_Y$. By construction, $g$ factors through the
Iitaka fibration of $X$, and so there is an integer $d \gg 0$ such that 
\begin{equation} \label{eq:equiv-1}
	d K_X \sim \gu H + E
\end{equation}
for some effective divisor $E$ on $X$. In particular, all sufficiently large and
divisible powers of $\omX$ have nontrivial global sections.

Now consider the torsion-free coherent sheaf
\[
 	\shF = g_* \Bigl( \omXm \tensor \alpha \tensor
		\shI \bigl( \norm{\omX^{\tensor(m-1)}} \bigr) \Bigr)
\]
on the abelian variety $A_Y$. From our discussion of asymptotic multiplier ideals in
\parref{par:Lai}, it is easy to see that we have inclusions
\[
	\shI \bigl( \norm{\omXm \tensor \alpha} \bigr) =
	\shI \bigl( \norm{\omXm} \bigr) \subseteq
	\shI \bigl( \norm{\omX^{\tensor(m-1)}} \bigr),
\]
by choosing the integer $p \in \NN$ in the definition of the asymptotic multiplier
ideal as a multiple of the order of the torsion point $\alpha \in \Pic^0(X)$. Since
\[
	H^0 \Bigl( X, \omXm \tensor \alpha 
		\tensor \shI \bigl( \norm{\omXm \tensor \alpha} \bigr) \Bigr) 
	= H^0 \bigl( X, \omXm \tensor \alpha \bigr),
\]
this shows that $H^0(A_Y, \shF) = H^0(X, \omXm \tensor \alpha)$. For $p \in \NN$
sufficiently large and divisible, we have
\[
	\shI \bigl( \norm{\omX^{\tensor(m-1)}} \bigr) 
		= \mu_{\ast} \shO_{X'} \Bigl( 	
			K_{X'/X} - \bigl\lfloor \tfrac{1}{p} F \bigr\rfloor \Bigr).
\]
Here $\mu \colon X'\to X$ is a log resolution of the complete linear system of
$\omX^{\tensor p(m-1)}$: the divisor $F + D$ has simple normal crossing support, the
linear system $\abs{D}$ is base point free, and 
\begin{equation} \label{eq:equiv-2}
	p(m-1) \mu^{\ast} K_X \sim F + D.
\end{equation}
(see \cite[9.2.10]{Lazarsfeld}). We may also assume that the larger divisor $F + D +
\mu^{\ast} E$ has simple normal crossing support.

Now $\shF$ is the pushforward, via the mapping $g \circ \mu \colon X' \to A_Y$, of
the line bundle 
\[
	\mu^{\ast} \alpha \tensor \shO_{X'} \Bigl( K_{X'} + (m-1) \mu^{\ast} K_X 
		- \bigl\lfloor \tfrac{1}{p} F \bigr\rfloor \Bigr),
\]
and for any $\eps \in \QQ$, we have the $\QQ$-linear equivalence of $\QQ$-divisors
\[
	(m-1) \mu^{\ast} K_X - \bigl\lfloor \tfrac{1}{p} F \bigr\rfloor
	\sim _\QQ \tfrac{1-\eps}{p}(F + D) + \tfrac{\eps(m-1)}{d} \bigl( 
		\mu^{\ast} g^{\ast} H + \mu^{\ast} E) - \bigl\lfloor \tfrac{1}{p} F \bigr\rfloor
\]
by combining \eqref{eq:equiv-1} and \eqref{eq:equiv-2}. This allows us to write
\[
	K_{X'} + (m-1) \mu^{\ast} K_X - \bigl\lfloor \tfrac{1}{p} F \bigr\rfloor
		\sim _\QQ  K_{X'} + \Delta + \tfrac{\eps(m-1)}{d} \mu^{\ast} g^{\ast} H,
\]
where $\Delta$ is the $\QQ$-divisor on $X'$ given by the formula
\[
	\Delta = \tfrac{1-\eps}{p} D + \tfrac{\eps(m-1)}{d} \mu^{\ast} E
		+ \tfrac{1-\eps}{p} F - \bigl\lfloor \tfrac{1}{p} F \bigr\rfloor.
\]
By construction, the support of $\Delta$ is a divisor with simple normal crossings;
and if we choose $\eps > 0$ sufficiently small, then $\Delta$ is a boundary divisor,
meaning that the coefficient of every irreducible component belongs to the interval
$[0,1)$. To see that $\Delta \geq 0$ it suffices to observe that $p(m-1)\mu ^*E\geq dF$ and to check that $\lfloor \Delta \rfloor =0$ it suffices to observe that the coefficients of $ \frac 1 p D +\{\frac  1 p F\}$ are $<1$ and apply continuity.
In particular, the pair $(X', \Delta)$ is klt. We can now apply the version
for $\QQ$-divisors of Koll\'ar's vanishing theorem \cite[\S10]{Kollar95} and conclude
that the pushforward of 
\[
	\mu^{\ast} (\alpha \tensor g^{\ast} \beta) \tensor 
	\shO_{X'} \Bigl( K_{X'} + (m-1) \mu^{\ast} K_X 
		- \bigl\lfloor \tfrac{1}{p} F \bigr\rfloor \Bigr)
\]
under the map $g \circ \mu \colon X' \to A_Y$ has vanishing higher cohomology for
all $\beta \in \Pic^0(Y)$. Together with the projection formula, this shows that 
\[
	H^i \bigl( A_Y, \shF \tensor \beta \bigr) = 0
\]
for every $i > 0$ and every $\beta \in \Pic^0(Y)$. It follows that
\[
	h^0\bigl( A_Y, \shF \tensor \beta \bigr) = 
		\chi \bigl( A_Y, \shF \tensor \beta \bigr)
\]
has the same value for every $\beta \in \Pic^0(Y)$. But then
\begin{align*}
	h^0\bigl( X, \omXm \tensor \alpha \bigr) &= h^0\bigl( A_Y, \shF \bigr)  
		= h^0\bigl( A_Y, \shF \tensor g^{\ast} \beta \bigr) \\
	&= h^0\Bigl( X, \omXm \tensor \alpha \tensor g^{\ast} \beta \tensor
		\shI \bigl( \norm{\omX^{\tensor(m-1)}} \bigr) \Bigr) \\
	&\leq h^0\bigl( X, \omXm \tensor \alpha \tensor g^{\ast} \beta \bigr),
\end{align*}
and by semicontinuity, we conclude that in fact 
\[
	h^0\bigl( X, \omXm \tensor \alpha \bigr) = 
	h^0\bigl( X, \omXm \tensor \alpha \tensor g^{\ast} \beta \bigr).
\]

We next prove (b).
Assuming that  $V^0(\omega _X ^{\otimes m})$ is nonempty, there are by
\theoremref{torsion_pluricanonical} finitely many distinct torsion elements 
$\alpha _1,\ldots , \alpha _s\in \Pic^0(X)$, and abelian subvarieties $B_i\subset \Pic^0(X)$,  such that  
\[
	V^0(\omega _X ^{\otimes m})=\bigcup_{i=1}^{s} \bigl( \alpha _i+B_i \bigr).
\]
By (a) we know that $\Pic^0(Y) \subseteq B_i$; indeed, (a) implies that for every
torsion point $\alpha \in V^0 (\omega_X^{\otimes m})$ we have $\alpha + \Pic^0 (Y)
\subseteq V^0 (\omega_X^{\otimes m})$, and this applies of course to 
$\alpha_i$. To prove (b), it is therefore enough to show that $B_i \subseteq
\Pic^0(Y)$. Take an arbitrary element $\alpha \in V^0(\omXm)$, and let $s \in H^0(X,
\omXm \tensor \alpha)$ be a nonzero global section. The restriction of $s$ to a
general fiber $F$ of the Iitaka fibration $f \colon X \to Y$ is then a nonzero global
section of 
\[
	\omega_F^{\tensor m} \tensor \alpha \restr{F},
\]
and because $\kappa(F) = 0$, it follows that $\alpha \restr{F}$ is torsion in
$\Pic^0(F)$. According to \lemmaref{l-iitaka}, a nonzero multiple of $\alpha$
therefore belongs to $\Pic^0(Y)$. This is enough to conclude that $B_i \subseteq
\Pic^0(Y)$, and so (b) is proved.

To prove (c), note first that standard arguments (see for instance \cite[Lemma
12.6]{EV}) imply the first half, namely that if $v\in  H^1(Y, \shO _{Y})$ then 
\[
	s \cup v = 0 \qquad 
		\text{for all $s\in H^0 \bigl( X,\omega _X^{\otimes m} \otimes \alpha \bigr)$.}
\]
For the second half, suppose that $s \in H^0(X, \omXm \tensor \alpha)$ is a nonzero
global section such that $s \cup v = 0$ for some $v \in H^1(X, \OX)$. Restricting to
a general fiber $F$ of the Iitaka fibration $f \colon X \to Y$, we get
\[
	0= s \restr{F} \cup v \restr{F} \in 
		H^1 \bigl( F,\omega _F^{\otimes m}\otimes \alpha \restr{F} \bigr),
\]
where $s \restr{F} \in H^0 \bigl( F,\omega_F^{\otimes m}\otimes \alpha \restr{F}
\bigr)$ is nonzero, and $v \restr{F} \in H^1(F, \shO _F)$. Since $\alpha \in
V^0(\omXm)$, we have $\alpha \restr{F} = \alpha_i \restr{F}$ for some $i = 1, \dotsc,
s$, as a consequence of \lemmaref{l-iitaka} and (b). In particular, $\alpha
\restr{F}$ is torsion, say of order $k$, and so 
\[
	s^{\tensor k} \restr{F} \in H^0 \bigl( F, \omega_F^{\tensor km} \bigr).
\]
Let $a_F \colon F \to A_F$ denote the Albanese mapping of $F$. Recalling that
$\kappa(F) = 0$, we get from \theoremref{t-pushforward} that $(a_F)_{\ast}
\omega_F^{\tensor km} = \shO_{A_F}$. Under the isomorphism
\[
	H^0 \bigl( F, \omega_F^{\tensor km} \bigr) = H^0 \bigl( A_F, \shO_{A_F} \bigr),
\]
our nonzero section $s^{\tensor k} \restr{F}$ therefore corresponds to a nonzero
constant $\sigma \in \CC$; likewise, under the isomorphism
\[
	H^1(F, \shO_F) = H^1 \bigl( A_F, \shO_{A_F} \bigr),
\]
the vector $v \restr{F}$ corresponds to a vector $u \in H^1 \bigl( A_F, \shO_{A_F}
\bigr)$. It is not hard to see that the two isomorphisms are compatible with cup
product; consequently, $s^{\tensor k} \cup v = 0$ implies that $\sigma u = 0$, and
hence that $u = 0$. By the infinitesimal version of \lemmaref{l-iitaka}, this means
that $v \in H^1(Y, \shO_Y)$, as asserted.
\end{proof}

\section{Singular metrics on pushforwards of adjoint line bundles}
\label{analytic}

\subsection{Plurisubharmonic functions}

Let $X$ be a complex manifold. We begin our survey of the analytic
techniques by recalling the following important definition; see for example
\cite[I.5]{Demailly} for more details.

\begin{definition}
A function $\varphi \colon X \to [-\infty, +\infty)$ is called
\define{plurisubharmonic} if it is upper semi-continuous, locally integrable, and
satisfies the mean-value inequality
\begin{equation} \label{eq:MVI-Delta}
	(\varphi \circ \gamma)(0) 
		\leq \frac{1}{\pi} \int_{\Delta} (\varphi \circ \gamma) \, d\mu
\end{equation}
for every holomorphic mapping $\gamma \colon \Delta \to X$ from the open unit disk
$\Delta \subseteq \CC$.
\end{definition}

Suppose that $\varphi$ is plurisubharmonic. From \eqref{eq:MVI-Delta}
one can deduce, by integrating over the space of lines through a given point, that
the mean-value inequality
\[
	(\varphi \circ \iota)(0) \leq
		\frac{1}{\mu(B)} \int_B (\varphi \circ \iota) \, d\mu
\]
also holds for any open embedding $\iota \colon B \into X$ of the open unit ball $B
\subseteq \CC^n$; here $n$ is the local dimension of $X$ at the point $\iota(0)$.
In other words, every plurisubharmonic function is also \define{subharmonic}.
Together with local integrability, this implies that $\varphi$ is locally bounded
from above. 

\begin{lemma} \label{lem:constant}
Every plurisubharmonic function on a compact complex manifold is locally constant.
\end{lemma}

\begin{proof}
Let $\varphi$ be a plurisubharmonic function on a compact complex manifold $X$.
As $\varphi$ is upper semi-continuous and locally bounded from above, it achieves
a maximum on every connected component of $X$. The mean-value inequality then forces
$\varphi$ to be locally constant.
\end{proof}

Observe that a plurisubharmonic function is uniquely determined by its values on any
subset whose complement has measure zero. Indeed, the mean-value inequality provides
an upper bound on the value at any point $x \in X$, and the upper semi-continuity a
lower bound. One also has the following analogue of the Riemann and Hartogs extension
theorems for holomorphic functions \cite[Theorem~I.5.24]{Demailly}; by
what we have just said, there can be at most one extension in each case.

\begin{lemma} \label{lem:Riemann-Hartogs}
Let $Z \subseteq X$ be a closed analytic subset, and let $\varphi$ be a
plurisubharmonic function on $X \setminus Z$. 
\begin{aenumerate}
\item If $\codim Z \geq 2$, then $\varphi$ extends to a plurisubharmonic
function on $X$.
\item If $\codim Z = 1$, then $\varphi$ extends to a plurisubharmonic function on $X$
if and only if it is locally bounded near every point of $Z$.
\end{aenumerate}
\end{lemma}

A plurisubharmonic function $\varphi$ determines a coherent sheaf of ideals
$\shI(\varphi) \subseteq \OX$, called the \define{multiplier ideal sheaf},
whose sections over any open subset $U \subseteq X$
consist of those holomorphic functions $f \in H^0(U, \OX)$ for which the function
$\abs{f}^2 e^{-\varphi}$ is locally integrable. We use the convention that the value
of the product is $0$ at points $x \in X$ where $f(x) = 0$ and $\varphi(x) = -\infty$.

Since plurisubharmonic functions are locally bounded from above, Montel's theorem in
several variables implies the following compactness property.

\begin{proposition} \label{prop:Montel}
Let $\varphi \colon B \to [-\infty, +\infty)$ be a plurisubharmonic function on the
open unit ball $B \subseteq \CC^n$. Consider the collection of holomorphic functions
\[
	H_K(\varphi) = \MENGE{f \in H^0(B, \shO_B)}%
		{\int_B \abs{f}^2 e^{-\varphi} \, d\mu \leq K}.
\]
Any sequence of functions in $H_K(\varphi)$ has a subsequence that converges
uniformly on compact subsets to an element of $H_K(\varphi)$.
\end{proposition}

\begin{proof}
The mean-value inequality for holomorphic functions implies that all functions in
$H_K(\varphi)$ are uniformly bounded on every closed ball of radius $R < 1$. Let us
briefly review the argument. Because $\varphi$ is locally bounded from above, there
is a constant $C \geq 0$ such that $\varphi \leq C$ on the closed ball of radius
$(R+1)/2$. Fix a point $z \in \wbar{B}_R(0)$ in the closed ball of radius $R$, and a
holomorphic function $f \in H_K(\varphi)$. By the mean-value inequality,
\[
	\abs{f(z)}^2 
		\leq \frac{1}{r^n \mu(B)} \int_{B_r(z)} \abs{f}^2 \, d\mu
		\leq \frac{e^C}{r^n \mu(B)} \int_{B_r(z)} \abs{f}^2 e^{-\varphi} \, d\mu
		\leq \frac{K \cdot e^{C}}{r^n \mu(B)},
\]
where $r = (1-R)/2$. By the $n$-dimensional version of Montel's theorem
\cite[Theorem~I.A.12]{GR}, this uniform
bound implies that any sequence $f_0, f_1, f_2, \dotsc \in H_K(\varphi)$ has a
subsequence that converges uniformly on compact subsets to a holomorphic function $f
\in H^0(B, \shO_B)$. By Fatou's lemma,
\[
	\int_B \abs{f}^2 e^{-\varphi} \, d\mu 
	\leq \liminf_{k \to +\infty} \int_B \abs{f_k}^2 e^{-\varphi} \, d\mu
	\leq K,
\]
which means that $f \in H_K(\varphi)$.
\end{proof}

\begin{note}
The example of an orthonormal sequence in $H_K(\varphi)$ shows that the convergence
need not be with respect to the $L^2$-norm.
\end{note}

\subsection{Singular hermitian metrics on line bundles}

Many of the newer applications of analytic techniques in algebraic geometry -- such as
Siu's proof of the invariance of plurigenera -- rely on the notion of
\define{singular hermitian metrics} on holomorphic line bundles. The word
``singular'' here means two things at once: first, that the metric is not necessarily
$C^{\infty}$; second, that certain vectors in the fibers of the line bundle are
allowed to have either infinite length or length zero.

Let $X$ be a complex manifold, and let $L$ be a holomorphic line bundle on $X$ with a
singular hermitian metric $h$.  In any local trivialization of $L$, such a metric is
represented by a ``weight function'' of the form $e^{-\varphi}$, where $\varphi$ is a
measurable function with values in $[-\infty, +\infty]$.  More precisely,
suppose that the restriction of $L$ to an open subset $U \subseteq X$
is trivial, and that $s_0 \in H^0(U, L)$ is a nowhere vanishing holomorphic section.
Then any other holomorphic section $s \in H^0(U, L)$ can be written as
$s = f s_0$ for a unique holomorphic function $f$ on $U$, and the length squared of
$s$ with respect to the singular hermitian metric $h$ is 
\begin{equation} \label{eq:product}
	\abs{s}_h^2 = \abs{f}^2 e^{-\varphi}.
\end{equation}
The points where $\varphi$ is not finite correspond to singularities of the metric:
at points where $\varphi(x) = -\infty$, the metric becomes infinite; at points where
$\varphi(x) = +\infty$, the metric stops being positive definite.

\begin{note}
At points $x \in U$ where $\varphi(x) = -\infty$, we use the following convention:
the product in \eqref{eq:product} equals $0$ if $f(x) = 0$; otherwise, it
equals $+\infty$. With this rule in place, $\abs{s}_h$ is a well-defined
measurable function on $U$ with values in $[0, +\infty]$. 
\end{note}

We say that a singular hermitian metric $h$ is \define{continuous} if the local
weight functions $\varphi$ are continuous functions with values in $[-\infty,
+\infty]$.  This is equivalent to asking that, for every open subset $U \subseteq X$
and every section $s \in H^0(U, L)$, the function $\abs{s}_h \colon U \to [0,
+\infty]$ should be continuous.

We say that the pair $(L, h)$ has \define{semi-positive curvature} if the local
weight functions $\varphi$ are plurisubharmonic. In that case, $\varphi$ is locally
integrable, and the curvature current of $(L, h)$ can be defined, in the
sense of distributions, by the formula
\[
	\Theta_h = \frac{\sqrt{-1}}{2 \pi} \del \dbar \varphi.
\]
It is easy to see that $\Theta_h$ is a well-defined closed positive $(1,1)$-current
on $X$; its cohomology class in $H^2(X, \RR)$ equals the first Chern class
$c_1(L)$. Conversely, if the current $\Theta_h$ is positive, then one can make the
local weight functions $\varphi$ plurisubharmonic by modifying them on a set of
measure zero \cite[Theorem~I.5.8]{Demailly}.

\begin{note}
Most authors include the condition of local integrability into the definition of a
singular hermitian metric. We use a different convention, so as to be consistent
with the definition of singular hermitian metrics on vector bundles later on.
\end{note}

A singular hermitian metric of semi-positive curvature is automatically positive
definite at every point. Indeed, $\varphi$ is locally bounded from above, and so 
the factor $e^{-\varphi}$ in the local expression for $h$ may equal $+\infty$ at
certain points, but has to be locally bounded from below by a positive
constant. Moreover, $\varphi$ is upper semi-continuous, and so the function $\abs{s}_h
\colon U \to [0, +\infty]$ is not just measurable, but even lower semi-continuous,
for every holomorphic section $s \in H^0(U, L)$ on some open subset $U \subseteq X$.

\begin{lemma} \label{lem:trivial}
Suppose that $X$ is compact, and that $h$ is a singular hermitian metric with
semi-positive curvature on a holomorphic line bundle $L$. If $c_1(L) = 0$ in $H^2(X,
\RR)$, then $h$ is actually a smooth metric with zero curvature.
\end{lemma}

\begin{proof}
The cohomology class of the closed positive $(1,1)$-current $\Theta_h$ equals zero in
$H^2(X, \RR)$, and so there is a globally defined plurisubharmonic function $\psi$
in $X$ such that $\Theta_h = \frac{\sqrt{-1}}{2\pi} \del \dbar \psi$. By
\lemmaref{lem:constant}, $\psi$ is locally constant, and so $\Theta_h = 0$. Now all
the local weight functions $\varphi$ satisfy $\del \dbar \varphi = 0$, and are
therefore smooth functions; but this means exactly that $h$ is a smooth metric.
\end{proof}

The curvature assumption implies that the \define{multiplier ideal sheaf}
$\shI(h) \subseteq \OX$ is a coherent sheaf of ideals on $X$; in the notation from
above, a holomorphic function $f \colon U \to \CC$ is a section of $\shI(h)$ if and
only if the function $\abs{f}^2 e^{-\varphi}$ is locally integrable. Consequently,
the subspace
\[
	H^0 \bigl( X, L \tensor \shI(h) \bigr)
		\subseteq H^0(X, L)
\]
consists of all global holomorphic sections of $L$ for which the lower
semi-continuous function $\abs{s}_h^2 \colon X \to [0, +\infty]$ is locally
integrable.

\subsection{The Ohsawa-Takegoshi extension theorem}

It is known that a line bundle on a projective complex manifold admits a singular
hermitian metric with semi-positive curvature if and only if it is pseudo-effective.
The power of the metric approach to positivity comes from fact that one can extend
holomorphic sections from submanifolds with precise bounds on the norm of the
extension. The most important result in this direction is the famous
\define{Ohsawa-Takegoshi theorem}.\footnote{We use the name ``Ohsawa-Takegoshi
theorem'' for convenience only; in reality, there is a large collection of different
$L^2$-extension theorems in complex analysis, of which \theoremref{thm:OT} below is
an important but nevertheless special case. For more on this topic, see for example
\cite{Demailly-extension}.}

Let $X$ be a complex manifold of dimension $n$, and let $(L, h)$ be a
holomorphic line bundle with a singular hermitian metric of semi-positive curvature.
What we actually need is the ``adjoint version'' of the Ohsawa-Takegoshi theorem,
which is about extending sections of the adjoint bundle $\omX \tensor L$, or
equivalently, holomorphic $n$-forms with coefficients in $L$. Before we can state the
theorem, we first have to introduce some notation. 

Given $\beta \in H^0(X, \omX \tensor L)$, we define a nonnegative measurable
$(n,n)$-form $\abs{\beta}_h^2$ as
follows: view $\beta \wedge \wbar{\beta}$ as a smooth $(n,n)$-form with coefficients
in $L \tensor \wbar{L}$, compose with the singular hermitian metric $h$, and then
multiply by the factor $c_n = 2^{-n} (-1)^{n^2/2}$. Locally, we can write $\beta =
f s_0 \tensor \dz_1 \wedge \dotsm \wedge \dz_n$, and then 
\begin{equation} \label{eq:h-beta}
	\abs{\beta}_h^2 = \abs{f}^2 e^{-\varphi} 
		(\dx_1 \wedge \dy_1) \wedge \dotsm \wedge (\dx_n \wedge \dy_n),
\end{equation}
where $z_1 = x_1 + y_1 \sqrt{-1}, \dotsc, z_n = x_n + y_n \sqrt{-1}$ are local
holomorphic coordinates on $U$. Using this notation, we have
\begin{align*}
	H^0 \bigl( X, \omX &\tensor L \tensor \shI(h) \bigr) \\
		&= \menge{\beta \in H^0(X, \omX \tensor L)}%
			{\text{$\abs{\beta}_h^2$ is locally integrable}}.
\end{align*}
We also define the $L^2$-norm of the element $\beta \in H^0(X, \omX \tensor L)$ to be
\begin{equation} \label{eq:norm-h}
	\norm{\beta}_h^2 = \int_X \abs{\beta}_h^2
		\, \in \, [0, +\infty].
\end{equation}
When $X$ is compact, $\beta \in H^0 \bigl( X, \omX \tensor L \tensor \shI(h) \bigr)$
is equivalent to having $\norm{\beta}_h^2 < +\infty$; in general, finiteness of the
$L^2$-norm is a much stronger requirement.

Now suppose that $f \colon X \to B$ is a holomorphic mapping to the open unit ball $B
\subseteq \CC^r$. We assume that $f$ is projective and that $0 \in B$ is a regular
value of $f$; the central fiber $X_0 = f^{-1}(0)$ is therefore a projective complex
manifold of dimension $n - r = \dim X - \dim B$. We denote by $(L_0, h_0)$ the
restriction of $(L, h)$ to $X_0$. As long as $h_0$ is not identically equal to
$+\infty$, it defines a singular hermitian metric with semi-positive curvature on
$L_0$, and we have the space 
\begin{equation} \label{eq:sections-X0}
	H^0 \bigl( X_0, \omega_{X_0} \tensor L_0 \tensor \shI(h_0) \bigr)
\end{equation}
of holomorphic $(n-r)$-forms with coefficients in $L_0$ that are square-integrable
with respect to $h_0$; as before, the defining condition is that the integral
\[
	\norm{\alpha}_{h_0}^2 = \int_{X_0} \abs{\alpha}_{h_0}^2
\]
should be finite; note that the definition of $\abs{\alpha}_{h_0}^2$ involves the
constant $c_{n-r}$.

The Ohsawa-Takegoshi theorem says that every section of $\omega_{X_0} \tensor L_0
\tensor \shI(h_0)$ can be extended to a section of $\omX \tensor L \tensor \shI(h)$
with finite $L^2$-norm -- and, crucially, it provides a universal upper bound on
the $L^2$-norm of the extension. (If $h_0 \equiv +\infty$, then the space in
\eqref{eq:sections-X0} is trivial and the extension problem is not interesting.) Here
$\beta \in H^0(X, \omX \tensor L)$ is an extension of $\alpha \in H^0(X_0,
\omega_{X_0} \tensor L_0)$ if
\[
	\beta \restr{X_0} = \alpha \wedge df = \alpha \wedge (df_1 \wedge \dotsb \wedge df_r),
\]
where $f = (f_1, \dotsc, f_r)$. That said, the precise statement of the
Ohsawa-Takegoshi extension theorem is the following.

\begin{theorem} \label{thm:OT}
Let $f \colon X \to B$ be a projective morphism such that
$0 \in B$ is a regular value. Let $(L, h)$ be a holomorphic line bundle with a singular
hermitian metric of semi-positive curvature. Denote by $(L_0, h_0)$ the restriction
to the central fiber $X_0 = f^{-1}(0)$, and suppose that $h_0 \not\equiv +\infty$. 
Then for every $\alpha \in H^0 \bigl( X_0, \omega_{X_0} \tensor L_0 \tensor \shI(h_0)
\bigr)$, there exists at least one $\beta \in H^0 \bigl( X, \omX \tensor L \tensor
\shI(h) \bigr)$ with
\[
	\beta \restr{X_0} = \alpha \wedge df \quad \text{and} \quad
	\norm{\beta}_h^2 \leq \mu(B) \cdot \norm{\alpha}_{h_0}^2.
\]
\end{theorem}

The special thing about this form of the extension theorem is that the constant
$\mu(B) = \pi^r/r!$ in the estimate is the volume of the unit ball $B \subseteq
\CC^r$; the example of a product $X = B \times X_0$ shows that this is optimal.
Earlier proofs of the Ohsawa-Takegoshi theorem, for example by Siu or P\u{a}un
\cite{Siu, Paun}, only gave a weaker estimate, in which $\mu(B)$ had
to be replaced by a certain constant $C_0 \leq 200$. The proof of the sharp estimate
is due to B{\l}ocki and Guan-Zhou \cite{Blocki,GZ}. There is also a
(weaker) version of the Ohsawa-Takegoshi theorem for the case where the fibers are
compact K\"ahler manifolds, proved by Cao \cite{Cao}.

\begin{proof}[Proof of \theoremref{thm:OT}]
In \cite[\S3.12]{GZ}, the result is stated only for ``projective families'', meaning
in the case where $f \colon X \to B$ is smooth and everywhere submersive, but the
same proof works as long as $0 \in B$ is a regular value. Guan and Zhou
have $(2 \pi)^r/r!$ as the constant, but the extra factor of $2^r$ goes away because our
definition of $\norm{\beta}_h^2$ and $\norm{\alpha}_h^2$ involves dividing by $2^n$
and $2^{n-r}$, respectively.

For the reader who wants to look up the result in \cite{GZ}, we briefly explain how to
deduce \theoremref{thm:OT} from Guan and Zhou's main theorem. Choose an embedding $X
\into B \times \PP^N$, and let $H \subseteq X$ be the preimage of a sufficiently
general hyperplane in $\PP^N$. Then $X
\setminus H$ is a Stein manifold and $X_0 \setminus X_0 \cap H$ a closed submanifold. 
We can now apply \cite[Theorem~2.2]{GZ} to the pair $(X, X_0)$, taking $A = 0$,
$c_A(t) \equiv 1$, and $\Psi = r \log \abs{f}^2 = r \log \bigl( \abs{f_1}^2 + \dotsb
+ \abs{f_r}^2 \bigr)$.
\end{proof}

\begin{note}
Observe that if we write the inequality in \theoremref{thm:OT} in the form
\[
	\frac{1}{\mu(B)} \norm{\beta}_h^2 \leq \norm{\alpha}_{h_0}^2,
\]
then it looks like a mean-value inequality; this fact will
play a crucial role later or, when we construct singular hermitian metrics on
pushforwards of adjoint bundles.
\end{note}

\subsection{Coherent sheaves and Fr\'echet spaces}

In this section, we briefly review some fundamental results about section spaces of coherent
sheaves on complex manifolds. Recall that a \define{Fr\'echet space} is a 
Hausdorff topological vector space, whose topology is induced by a countable
family of semi-norms, and which is complete with respect to this family of semi-norms.
Most of the familiar theorems about Banach spaces, such as the open mapping
theorem or the closed graph theorem, remain true for Fr\'echet spaces.

\begin{example}
On a complex manifold $X$, the vector space $H^0(X, \OX)$ of all holomorphic
functions on $X$ is a Fr\'echet space, under the topology of uniform convergence on
compact subsets. More precisely, each compact subset $K \subseteq X$ gives rise to 
a semi-norm
\[
	\norm{f}_K = \sup_{x \in K} \abs{f(x)}
\]
on the space $H^0(X, \OX)$; to get a countable family, write $X$ as a
countable union of compact subsets. The same construction works for any open subset
$U \subseteq X$, and when $U \subseteq V$, the
restriction mapping $H^0(V, \shO _X) \to H^0(U, \shO _X)$ is continuous.
\end{example}

In fact, the section spaces of \emph{all} coherent sheaves on a given complex manifold can
be made into  Fr\'echet spaces in a consistent way; the construction is explained for
example in \cite[Ch. VIII, \S{}A]{GR}. Let $\shF$ be a coherent sheaf on a complex
manifold $X$. Then for every open subset $U \subseteq X$, the space of sections $H^0(U,
\shF)$ has the structure of a Fr\'echet space, in such a way that if $U \subseteq V$,
the restriction mapping $H^0(V, \shF) \to H^0(U, \shF)$ is continuous.  Moreover, if
$\phi \colon \shF \to \shG$ is any morphism between two coherent sheaves, then
the induced mappings $\phi_U \colon H^0(U, \shF) \to H^0(U, \shG)$ are all continuous.
The Fr\'echet space topology has several other good properties, such as the following
\cite[Proposition~VIII.A.2]{GR}.

\begin{proposition}
If $\shF \subseteq \shG$, then $H^0(U, \shF)$ is a closed subspace of $H^0(U, \shG)$.
\end{proposition}

Let $f \colon X \to Y$ be a proper holomorphic mapping between complex manifolds, and
let $\shF$ be a coherent sheaf on $X$. By Grauert's coherence theorem, the
pushforward sheaf $\fl \shF$ is a coherent sheaf on $Y$. The vector space
\[
	H^0(Y, \fl \shF) = H^0(X, \shF)
\]
therefore has two (a priori different) Fr\'echet space topologies, one coming from
$Y$, the other from $X$. 

\begin{proposition} \label{prop:proper-mapping}
In the situation just described, the two Fr\'echet space topologies on $H^0(Y, \fl
\shF) = H^0(X, \shF)$ are equal.
\end{proposition}

\begin{proof}
Since the problem is local, we may replace $Y$ by a Stein open subset and assume
that we have a surjective morphism 
\[
	\OY^{\oplus m} \to \fl \shF.
\]
The induced mapping $H^0(Y, \OY)^{\oplus m} \to H^0(Y, \fl \shF)$ is continuous and
surjective; by the open mapping theorem, the topology on $H^0(Y, \fl \shF)$ must be
the quotient topology. We also get a morphism $\OX^{\oplus m} \to \shF$, and
therefore a factorization
\[
	H^0(Y, \OY)^{\oplus m} \to H^0(X, \OX)^{\oplus m} \to H^0(X, \shF).
\]
Both mappings are continuous: the first because, $f$ being proper, uniform convergence
on compact subsets of $Y$ implies uniform convergence on compact subsets of $X$; the
second because $\OX^{\oplus m} \to \shF$ is a morphism. It follows that the identity
mapping
\[
	H^0(Y, \fl \shF) \to H^0(X, \shF)
\]
is continuous; by the open mapping theorem, it must be a homeomorphism.
\end{proof}

\subsection{Singular hermitian inner products}
\label{par:SHIP}

Before we can talk about singular hermitian metrics on vector bundles, we first have
to be clear about what we mean by a ``singular'' hermitian inner product on a vector
space. The purpose of this section is to define this notion with some care. 
Throughout, we let $V$ be a finite-dimensional complex vector space. There are two
ways in which a hermitian inner product can be singular: there may be vectors whose
length is $+\infty$, and others whose length is $0$. The best way to formalize this
is to work not with the inner product itself, but with the associated length
function \cite[\S3]{BP}.

\begin{definition}
A \define{singular hermitian inner product} on a finite-dimensional complex vector
space $V$ is a function $\abs{\argbl}_h \colon V \to [0, +\infty]$ with the following properties:
\begin{enumerate}
\item $\abs{\lambda v}_h = \abs{\lambda} \cdot \abs{v}_h$ for every $\lambda \in \CC
\setminus \{0\}$ and every $v \in V$, and $\abs{0}_h = 0$.
\item $\abs{v + w}_h \leq \abs{v}_h + \abs{w}_h$ for every $v, w \in V$.
\item $\abs{v+w}_h^2 + \abs{v-w}_h^2 = 2 \abs{v}_h^2 + 2 \abs{w}_h^2$ for every $v,w \in V$.
\end{enumerate}
\end{definition}

Our convention is that an inequality is satisfied if both sides are equal to
$+\infty$. It is easy to deduce from the axioms that both
\[
V_0 = \menge{v \in V}{\abs{v}_h = 0} \quad \text{and} \quad
\Vfin = \menge{v \in V}{\abs{v}_h < +\infty}
\]
are linear subspaces of $V$. We say that $h$ is \define{positive definite} if $V_0 =
0$; we say that $h$ is \define{finite} if $\Vfin = V$. Clearly, $\abs{\argbl}_h$ is a
semi-norm on $\Vfin$; it is a norm if and only if $V_0 = 0$. The third axiom is
the parallelogram law for this semi-norm. The formula
\[
	\inner{v}{w}_h = \frac{1}{4} \sum_{k=0}^3 
		(\sqrt{-1})^k \cdot \bigabs{v + (\sqrt{-1})^k w}_h
\]
therefore defines a semi-definite hermitian inner product on the subspace $\Vfin$; it
is positive definite if and only if $V_0 = 0$. We use the same notation for the
induced hermitian inner product on the quotient space $\Vfin / V_0$.

Given a singular hermitian inner product $h$ on $V$, we obtain a singular hermitian inner
product $\hd$ on the dual space $\Vd = \Hom_{\CC}(V, \CC)$ by setting
\[
	\abs{f}_{\hd} = \sup \MENGE{\frac{\abs{f(v)}}{\abs{v}_h}}{%
		\text{$v \in V$ with $\abs{v}_h \neq 0$}}
\]
for any linear functional $f \in \Vd$, with the understanding that a fraction with
denominator $+\infty$ is equal to $0$. (If $V_0 = V$, then we define $\abs{f}_{\hd} =
0$ for $f = 0$, and $\abs{f}_{\hd} = +\infty$ otherwise.) It is easy to see
that $\abs{f}_{\hd} = 0$ if and only if $f$ annihilates the subspace $\Vfin$, and
that $\abs{f}_{\hd} < +\infty$ if
and only if $f$ annihilates the subspace $V_0$. One then checks that $\hd$ is again a
singular hermitian inner product on $\Vd$, and that the resulting hermitian inner
product $\inner{\argbl}{\argbl}_{\hd}$ on 
\[
	\frac{\menge{f \in \Vd}{\abs{f}_{\hd} < +\infty}}%
		{\menge{f \in \Vd}{\abs{f}_{\hd} = 0}}
		\simeq \Hom_{\CC} \bigl( \Vfin / V_0, \CC \bigr)
\]	
agrees with the one naturally induced by $\inner{\argbl}{\argbl}_h$. Here is another
way to think about $\hd$. From a nonzero linear functional $f \colon V \to \CC$, we
get an induced singular hermitian inner product on $\CC$ by setting
\[
	\abs{\lambda}_{h, f} = \inf \menge{\abs{v}_h}%
		{\text{$v \in V$ satisfies $f(v) = \lambda$}}
\]
If $\lambda \neq 0$, this quantity is $+\infty$ unless the restriction of $f$ to
the subspace $\Vfin$ is nonzero; if $V_0 = V$, then $\abs{\lambda}_{h,f} = 0$ for
every $\lambda \in \CC$. Taking into account various special cases, the
following result is immediate from the definition.

\begin{lemma} \label{lem:induced}
Let $f \colon V \to \CC$ be a nonzero linear functional. Then
\[
	\abs{\lambda}_{h,f} = \frac{\abs{\lambda}}{\abs{f}_{\hd}}
\]
for every nonzero $\lambda \in \CC$.
\end{lemma}

Let $r = \dim V$. Since the product of $0$ and $+\infty$ is undefined, we do not get
a singular hermitian inner product on
\[
	\det V = \bigwedge^r V
\]
unless $V_0 = 0$ or $\Vfin = V$. But when $h$ is either positive definite or finite,
there is a well-defined singular hermitian inner product $\det h$ on the
one-dimensional vector space $\det V$. If $\Vfin = V$, we declare that
\[
	\abs{v_1 \wedge \dotsb \wedge v_r}_{\det h} = 
	\det \begin{pmatrix}
		\inner{v_1}{v_1}_h & \cdots & \inner{v_1}{v_r}_h \\
		\vdots & \ddots & \vdots \\
		\inner{v_r}{v_1}_h & \cdots & \inner{v_r}{v_r}_h
	\end{pmatrix}.
\]
If $\Vfin \neq V$ and $V_0 = 0$, we let $\abs{\argbl}_{\det h}$ equal $+\infty$ on
all nonzero elements of $\det V$.

\subsection{Singular hermitian metrics on vector bundles}

The purpose of this section is to extend the concept of singular hermitian metrics
from holomorphic line bundles to holomorphic vector bundles of arbitrary rank. Let
$X$ be a complex manifold, and let $E$ be a holomorphic vector bundle on $X$ of some
rank $r \geq 1$.

\begin{definition} \label{def:SHIP}
A \define{singular hermitian metric} on $E$ is a function $h$ that associates to every
point $x \in X$ a singular hermitian inner product $\abs{\argbl}_{h,x} \colon E_x \to
[0, +\infty]$ on the complex vector space $E_x$, subject to the following two
conditions:
\begin{enumerate}
\item $h$ is finite and positive definite almost everywhere, meaning that for all $x$
outside a set of measure zero, $\abs{\argbl}_{h,x}$ is a hermitian inner product on $E_x$.  
\item $h$ is measurable, meaning that the function 
\[
	\abs{s}_h \colon U \to [0, +\infty], 
		\quad x \mapsto \bigabs{s(x)}_{h,x},
\]
is measurable whenever $U \subseteq X$ is open and $s \in H^0(U, E)$.
\end{enumerate}
\end{definition}

In the case $r = 1$, this specializes to the definition of singular hermitian
metrics on holomorphic line bundles. The requirement that $h$ be measurable is
extremely weak: the singular hermitian metrics that we will actually encounter below
are at least semi-continuous. The advantage of the above definition is that it
behaves well under duality. By applying the general construction from the previous
section, we obtain on each fiber
\[
	\Ed_x = \Hom_{\CC}(E_x, \CC)
\]
of the dual bundle $\Ed$ a singular hermitian inner product $\abs{\argbl}_{\hd,x}$.
The following result shows that these form a singular hermitian metric on $\Ed$.

\begin{proposition}
A singular hermitian metric $h$ on a holomorphic vector bundle $E$ induces a singular
hermitian metric $\hd$ on the dual bundle $\Ed$.
\end{proposition}

\begin{proof}
If $\abs{\argbl}_{h,x}$ is finite and positive definite, then $\abs{\argbl}_{\hd,x}$
is also finite and positive
definite, and so the first condition in the definition is clearly satisfied. The
second condition is of a local nature, and so we may assume without loss of
generality that $E$ is the trivial bundle of rank $r$. Denote by $s_1, \dotsc, s_r
\in H^0(X, E)$ the global sections corresponding to a choice of trivialization. The
expression
\[
	H_{i,j}(x) = \inner{s_i(x)}{s_j(x)}_{h,x}
\]
is well-defined outside a set of measure zero, and the resulting function is
measurable. 
Denote by $H \in \Mat_{r \times r}(\CC)$ the $r \times r$-matrix with entries
$H_{i,j}$.  Then $\hd$ is represented by the transpose of the matrix $H^{-1}$, in the
natural trivialization of $\Ed$; the usual formula for the inverse of a matrix shows
that all entries of this matrix are again measurable functions.
\end{proof}

\begin{note}
In more sheaf-theoretic terms, a singular hermitian metric on a holomorphic vector
bundle $E$ is a morphism of sheaves of sets
\[
	\abs{\argbl}_h \colon E \to \shMs_X
\]
from $E$ to the sheaf of measurable functions on $X$ with values in $[0, +\infty]$.
The following conditions need to be satisfied:
\begin{enumerate}
\item One has $\abs{fs}_h = \abs{f} \cdot \abs{s}_h$ for every $s \in H^0(U, E)$ and
every $f \in H^0(U, \shO)$.
\item If $s \in H^0(U, E)$ and $\abs{s}_h = 0$ almost everywhere, then $s = 0$.
\item For almost every point $x \in X$, the function $\abs{\argbl}_{h,x} \colon E_x
\to [0, +\infty]$ is a singular hermitian inner product (in the sense of
\definitionref{def:SHIP}).
\end{enumerate}
Again, we use the convention that $\abs{f} \cdot \abs{s}_h = 0$ at points where $f$
is zero.
\end{note}

\subsection{Semi-positive curvature}

Let $h$ be a singular hermitian metric on a holomorphic vector bundle $E$, and denote
by $\hd$ the induced singular hermitian metric on the dual bundle $\Ed$. Suppose for
a moment that $h$ is smooth, and denote by $\Theta_h$ the curvature tensor of the
Chern connection; it is a $(1,1)$-form with coefficients in the bundle $\End(E)$.
One says that $(E, h)$ has \define{semi-positive curvature in the sense of
Griffiths} if, for every choice of holomorphic tangent vector $\xi \in T_x X$, the
matrix $\Theta_h(\xi, \wbar{\xi})$ is positive semi-definite \cite[VII.6]{Demailly}.
This is known to be equivalent to the condition that the function $\log
\abs{f}_{\hd}$ is plurisubharmonic for every local section $f \in H^0(U, \Ed)$. In
the singular case, we use this condition as the definition.

\begin{definition}
We say that the pair $(E, h)$ has \define{semi-positive curvature} if the function
$\log \abs{f}_{\hd}$ is plurisubharmonic for every $f \in H^0(U, \Ed)$.
\end{definition}

The point of this definition is that it allows us to talk about the curvature of a
singular hermitian metric without mentioning the curvature tensor: unlike in the case
of line bundles, the curvature tensor of $h$ does not in general make sense even as a
distribution \cite[Theorem~1.3]{Raufi}. The following lemma gives an equivalent
formulation of the definition.

\begin{lemma} \label{lem:quotient}
Let $h$ be a singular hermitian metric on $E$. Then $(E, h)$ has semi-positive
curvature if, and only if, for every open subset $U \subseteq X$ and every nonzero
morphism $E \restr{U} \to L$ to a line bundle, the induced singular hermitian metric
on $L$ has semi-positive curvature.
\end{lemma}

\begin{proof}
The construction of the induced singular hermitian metric on $L$ works as in
\lemmaref{lem:induced}. At each point $x \in U$, the linear mapping $E_x \to L_x$
between fibers induces a singular hermitian inner product on the one-dimensional
complex vector space $L_x$: the length of a vector $\lambda \in L_x$ is the infimum
of $\abs{e}_{h,x}$ over all $e \in E_x$ that map to $\lambda$. (If $E_x \to L_x$ is
zero, then the infimum equals $+\infty$ whenever $\lambda \neq 0$.)

Let us compute the curvature of the induced metric. After replacing $X$ by an open
neighborhood of a given point in $U$, we may assume that $L$ is trivial; our morphism
$E \to \OX$ is then given by a linear functional $f \in H^0(X, \Ed)$. Let
$e^{-\varphi}$ be the weight function of the induced metric. The formula in
\lemmaref{lem:induced} says that 
\[
	e^{-\varphi(x)} = \frac{1}{\abs{f(x)}_{\hd, x}^2}
\]
for every $x \in X$. Taking logarithms, we get $\varphi = 2 \log \abs{f}_{\hd}$, which
is plurisubharmonic because the pair $(E, h)$ has semi-positive curvature.
\end{proof}

Suppose that $(E, h)$ has semi-positive curvature. Since plurisubharmonic functions
are locally bounded from above, the singular hermitian inner product
$\abs{\argbl}_{\hd, x}$ on $\Ed_x$ must be finite for every $x \in X$; dually, every
$\abs{\argbl}_{h,x}$ must be positive definite. The determinant line bundle $\det E$
therefore has a well-defined singular hermitian metric that we denote by the
symbol $\det h$. We will prove later (in \propositionref{prop:determinant}) that the
pair $(\det E, \det h)$ again has semi-positive curvature.

When $(E, h)$ has semi-positive curvature, the pointwise length of any holomorphic
section of $\Ed$ is an upper semi-continuous function. Likewise, the pointwise length
of any holomorphic section of $E$ is a lower semi-continuous function.

\begin{lemma} \label{lem:lower-semi-continuous}
If $(E, h)$ has semi-positive curvature, then for any $s \in H^0(X, E)$, the function
$\abs{s}_h \colon X \to [0, +\infty]$ is lower semi-continuous.
\end{lemma}

\begin{proof}
Since the question is local, we may assume without loss of generality that $X$ is the
open unit ball in $\CC^n$, and $E$ the trivial bundle of rank $r \geq 1$. We have
\[
	\abs{s}_h \geq \frac{\abs{f(s)}}{\abs{f}_{\hd}}
\]
for every $f \in H^0(X, \Ed)$, and it is easy to see that $\abs{s}_h$ is the
pointwise supremum of the collection of functions on the right-hand side. Because
$\log \abs{f}_{\hd}$ is upper semi-continuous, each 
\[
	\frac{\abs{f(s)}}{\abs{f}_{\hd}} = \abs{f(s)} \cdot e^{-\log \abs{f}_{\hd}}
\]
is a lower semi-continuous function from $X$ to $[0, +\infty]$; their 
pointwise supremum is therefore also lower semi-continuous.
\end{proof}

\begin{example}
The following example, due to Raufi \cite[Theorem~1.3]{Raufi}, shows that the
function $\abs{s}_h$ can indeed have jumps. Let $E$ be the trivial bundle of rank $2$
on $\CC$. We first define a singular hermitian metric $\hd$ on the dual bundle $\Ed$:
at each point $z \in \CC$, it is represented by the matrix
\[
	\begin{pmatrix}
		1 + \abs{z}^2 & z \\
		\bar{z} & \abs{z}^2
	\end{pmatrix}.
\]	
From this, one computes that the singular hermitian metric $h$ on $E$ is given by
\[
	\frac{1}{\abs{z}^4} \begin{pmatrix}
		\abs{z}^2 & -z \\
		-\bar{z} & 1 + \abs{z}^2	
	\end{pmatrix}
\]
as long as $z \neq 0$. Contrary to what this formula might suggest, one has
\[
	\bigabs{(1,0)}_{h, 0} = 1;
\]
the length of the vector $(1,0)$ is thus $\abs{z}^{-2}$ for $z \neq 0$, but $1$ for
$z = 0$.
\end{example}

\subsection{Singular hermitian metrics on torsion-free sheaves}

Let $X$ be a complex manifold, and let $\shF$ be a torsion-free coherent sheaf on
$X$. Let $X(\shF) \subseteq X$ denote the maximal open subset where $\shF$ is locally free;
then $X \setminus X(\shF)$ is a closed analytic subset of codimension $\geq 2$. If $\shF
\neq 0$, then the restriction of $\shF$ to the open subset $X(\shF)$ is a holomorphic
vector bundle $E$ of some rank $r \geq 1$.

\begin{definition}
A \define{singular hermitian metric} on $\shF$ is a singular hermitian metric $h$ on
the holomorphic vector bundle $E$. We say that such a metric has
\define{semi-positive curvature} if the pair $(E, h)$ has semi-positive curvature.
\end{definition}

Suppose that $\shF$ has a singular hermitian metric with semi-positive curvature.
Since $X \setminus X(\shF)$ has codimension $\geq 2$, every holomorphic section of
the dual bundle $\Ed$ extends to a holomorphic section of the reflexive coherent sheaf
\[
	\shFd = \shHom(\shF, \OX),
\]
and every plurisubharmonic function on $X(\shF)$ extends to a plurisubharmonic
function on $X$ (see \lemmaref{lem:Riemann-Hartogs}). For every open subset $U
\subseteq X$ and every holomorphic section $f \in H^0(U, \shFd)$, we thus obtain
a well-defined plurisubharmonic function
\[
	\log \abs{f}_{\hd} \colon U \to [-\infty, +\infty).
\]
Note that the function $\abs{f}_{\hd}$ is upper semi-continuous.

What about holomorphic sections of the sheaf $\shF$ itself? For any $s \in H^0(U,
\shF)$, the function $\abs{s}_h$ is lower semi-continuous on $U \cap X(\shF)$. In a
suitable neighborhood of every point in $U$, we can imitate the proof of
\lemmaref{lem:lower-semi-continuous} and take the pointwise supremum of the functions
\[
	\abs{f(s)} \cdot e^{-\log \abs{f}_{\hd}},
\]
where $f$ runs over all sections of $\shFd$. Since the pointwise supremum of a
family of lower semi-continuous functions is again lower semi-continuous, we obtain
in this manner a distinguished extension 
\[
	\abs{s}_h \colon U \to [0, +\infty]
\]
to a lower semi-continuous function on $U$. 

\begin{definition}
We say that a singular hermitian metric on $\shF$
is \define{continuous} if, for every open subset $U \subseteq X$ and every
holomorphic section $s \in H^0(U, \shF)$, the function $\abs{s}_h \colon U \to [0,
+\infty]$ is continuous.  
\end{definition}


\begin{proposition} \label{prop:shF-shG}
Let $\phi \colon \shF \to \shG$ be a morphism between two torsion-free coherent
sheaves that is generically an isomorphism. If $\shF$ has a singular hermitian metric 
with semi-positive curvature, then so does $\shG$.
\end{proposition}

\begin{proof}
Let $h$ denote the singular hermitian metric on $\shF$. On the open subset of
$X(\shF) \cap X(\shG)$ where $\phi$ is an isomorphism, $\shG$ clearly acquires a
singular hermitian metric that we also denote by $h$ for simplicity. Because the dual
morphism $\phi^{\ast} \colon \shG^{\ast} \to \shFd$ is injective, the function $\log
\abs{f}_{h^*}$ is plurisubharmonic for every $f \in H^0(U, \shG^{\ast}) \subseteq H^0(U,
\shFd)$.  Consequently, $h$ extends to a singular hermitian metric with semi-positive
curvature on all of $X(\shG)$.
\end{proof}

\begin{example}
If $\shF$ has a singular hermitian metric of semi-positive curvature, then the same
is true for the reflexive hull $\shF^{\ast\ast}$.
\end{example}

\subsection{The minimal extension property}
\label{par:MEP}

The Ohsawa-Takegoshi theorem leads us to consider the following ``minimal extension
property'' for singular hermitian metrics. To keep the statement simple, let us
assume that $X$ is a connected complex manifold of dimension $n$, and denote by $B
\subseteq \CC^n$ the open unit ball.

\begin{definition}
We say that a singular hermitian metric on $\shF$ has the \define{minimal extension
property} if there exists a nowhere dense closed analytic subset $Z \subseteq X$ with
the following two properties:
\begin{enumerate}
\item $\shF$ is locally free on $X \setminus Z$, or equivalently, $X \setminus Z
\subseteq X(\shF)$.
\item For every embedding $\iota \colon B \into X$ with $x =  \iota(0) \in X \setminus
Z$, and every $v \in E_x$ with $\abs{v}_{h,x} = 1$, there is
a holomorphic section $s \in H^0 \bigl( B, \iota^{\ast} \shF \bigr)$ such that
\[
	s(0) = v \quad \text{and} \quad 
	\frac{1}{\mu(B)} \int_B \abs{s}_h^2 \, d\mu \leq 1;
\]
here $(E, h)$ denotes the restriction to the open subset $X(\shF)$.
\end{enumerate}
\end{definition}

The point of the minimal extension property is the ability to extend sections
over the ``bad'' locus $Z$, with good control on the norm of the extension. 
We will see later that pushforwards of adjoint line bundles always have this
property, as a consequence of the Ohsawa-Takegoshi theorem. 

\begin{example}
The minimal extension property rules out certain undesirable examples like the
following. Let $Z \subseteq X$ be a closed analytic subset of codimension $\geq 2$,
and let $\shI_Z \subseteq \OX$ denote the ideal sheaf of $Z$. Then $\shI_Z$ is 
trivial on $X \setminus Z$, and the constant hermitian metric on this trivial
bundle is a singular hermitian metric with semi-positive curvature on $\shI_Z$. But
this metric does not have the minimal extension property, because a holomorphic
function $f \colon B \to \CC$ with $f(0) = 1$ and 
\[
	\frac{1}{\mu(B)} \int_B \abs{f}^2 d\mu \leq 1
\]
must be constant.
\end{example}

\subsection{Pushforwards of adjoint line bundles}

Let $X$ be a complex manifold of dimension $n$, and let $(L, h)$ be a holomorphic
line bundle with a singular hermitian metric of semi-positive curvature. If $X$ is
compact, the space $H^0(X, \omX \tensor L)$ is finite-dimensional, and the formula
\[
	\norm{\beta}_h^2 = \int_X \abs{\beta}_h^2
\]
endows it with a positive definite singular hermitian inner product that is
finite on the subspace $H^0 \bigl( X, \omX \tensor L \tensor \shI(h) \bigr)$. We are
now going to analyze how this construction behaves in families. 

Suppose then that $f \colon X \to Y$ is a projective surjective morphism between two
connected complex manifolds, with $\dim X = n$ and $\dim Y = r$; the general fiber of
$f$ is a projective complex manifold of dimension $n-r$, but there may be singular
fibers. Let $(L, h)$ be a holomorphic line bundle with a singular hermitian metric of
semi-positive curvature on $X$. The following important theorem was essentially
proved by P\u{a}un and Takayama \cite[Theorem~3.3.5]{PT}, building on earlier results
for smooth morphisms by Berndtsson and P\u{a}un \cite{Berndtsson, BP}.

\begin{theorem} \label{thm:pushforward}
Let $f \colon X \to Y$ be a projective surjective morphism between two connected
complex manifolds. If $(L, h)$ is a holomorphic line bundle with a singular
hermitian metric of semi-positive curvature on $X$, then the pushforward sheaf
\[
	\shF = \fl \bigl( \omXY \tensor L \tensor \shI(h) \bigr)
\]
has a canonical singular hermitian metric $H$. This metric has semi-positive
curvature and satisfies the minimal extension property.
\end{theorem}

The metric in the theorem is uniquely characterized by a simple property that we now
describe. Recall from \eqref{eq:norm-h} and \eqref{eq:h-beta} that any $\beta
\in H^0 \bigl( X, \omX \tensor L \tensor \shI(h) \bigr)$ gives rise to a locally
integrable $(n,n)$-form $\abs{\beta}_h^2$. Any such form can be integrated against compactly
supported smooth functions, and therefore defines a current of bidegree
$(n,n)$ on $X$. If we use brackets to denote the evaluation pairing between
$(n,n)$-currents and compactly supported smooth functions, then
\[
	\bigl\langle \abs{\beta}_h^2, \phi \bigr\rangle 
		= \int_X \phi \cdot \abs{\beta}_h^2.
\]
By the same token, any section $\beta \in H^0(Y, \omY \tensor \shF)$ defines a
current of bidegree $(r,r)$ on $Y$ that we denote by the symbol $\abs{\beta}_H^2$.
Now suppose that 
\[
	\beta \in H^0(U, \omY \tensor \shF) \simeq
		H^0 \bigl( f^{-1}(U), \omX \tensor L \tensor \shI(h) \bigr).
\]
The singular hermitian metric $H$ is uniquely characterized by the condition that
\[
	\bigl\langle \abs{\beta}_H^2, \phi \bigr\rangle
	= \bigl\langle \abs{\beta}_h^2, \fu \phi \bigr\rangle
\]
for every compactly supported smooth function $\phi \in A_c(U)$. Said differently,
$\abs{\beta}_H^2$ is the pushforward of the current $\abs{\beta}_h^2$ under the
proper mapping $f$.

\begin{corollary} \label{cor:pushforward}
In the situation of \theoremref{thm:pushforward}, suppose that the inclusion
\[
	\fl \bigl( \omXY \tensor L \tensor \shI(h) \bigr) \into
		\fl(\omXY \tensor L)
\]
is generically an isomorphism. Then $\fl(\omXY \tensor L)$ also has a singular
hermitian metric with semi-positive curvature and the minimal extension
property.
\end{corollary}

\begin{proof}
The existence of the metric follows from \propositionref{prop:shF-shG}. The minimal
extension property continues to hold because every section of $\fl \bigl( \omXY
\tensor L \tensor \shI(h) \bigr)$ is of course also a section of $\fl(\omXY \tensor L)$.
\end{proof}

\begin{example}
If we apply \theoremref{thm:pushforward} to the identity morphism $\id \colon X \to
X$, we only get a singular hermitian metric on $L \tensor \shI(h)$. To recover the
singular hermitian metric on $L$ that we started from, we can use
\corollaryref{cor:pushforward}.
\end{example}

The proof of \theoremref{thm:pushforward} gives the following additional information
about the singular hermitian metric on $\shF = \fl \bigl( \omXY \tensor L \tensor
\shI(h) \bigr)$ (see the end of \parref{subsec:proof-III}).

\begin{corollary} \label{cor:continuous}
In the situation of \theoremref{thm:pushforward}, suppose that $f \colon X \to Y$ is
submersive and that the singular hermitian metric $h$ on the line bundle $L$ is
continuous. Then the singular hermitian metric $H$ on $\shF$ is also continuous.
\end{corollary}

The following three sections explain the proof of \theoremref{thm:pushforward}. In a
nutshell, it is an application of the Ohsawa-Takegoshi extension theorem. We
present the argument in three parts that rely on successively stronger versions of
the extension theorem: first the ability to extend sections from a fiber; then the
fact that there is a universal bound on the norm of the extension; and finally the
optimal bound in \theoremref{thm:OT}. 

\subsection{Proof of the pushforward theorem, Part I}
\label{par:pushforward-proof}

Our first goal is to define the singular hermitian metric on
\[
	\shF = \fl \bigl( \omXY \tensor L \tensor \shI(h) \bigr),
\]
and to establish a few basic facts about it. In this part of the proof, we only use
the weakest version of the Ohsawa-Takegoshi extension theorem, namely the ability to
extend sections from a fiber.

The idea is to construct the metric first over a Zariski-open subset $Y \setminus Z$
where everything is nice, and then to extend it over the bad locus $Z$. To begin
with, choose a nowhere dense closed analytic subset $Z \subseteq Y$ with the
following three properties:
\begin{enumerate}
\item The morphism $f$ is submersive over $Y \setminus Z$.
\item Both $\shF$ and the quotient sheaf $\fl(\omXY \tensor L) / \shF$ are locally free
on $Y \setminus Z$.
\item On $Y \setminus Z$, the locally free sheaf $\fl(\omXY \tensor L)$ has the base
change property.
\end{enumerate}

By the base change theorem, the third condition will hold as long as the coherent
sheaves $R^i \fl(\omXY \tensor L)$ are locally free on $Y \setminus Z$. The
restriction of $\shF$ to the open subset $Y \setminus Z$ is a holomorphic vector
bundle $E$ of some rank $r \geq 1$. The second and third condition together guarantee
that
\[
	E_y = \shF \restr{y} \subseteq \fl(\omXY \tensor L) \restr{y} 
	= H^0 \bigl( X_y, \omega_{X_y} \tensor L_y \bigr)
\]
whenever $y \in Y \setminus Z$. As before, $(L_y, h_y)$ denotes the restriction of $(L,
h)$ to the fiber $X_y = f^{-1}(y)$; it may happen that $h_y \equiv +\infty$. The
Ohsawa-Takegoshi theorem gives us the following additional information about $E_y$.

\begin{lemma} \label{lem:Fubini}
For any $y \in Y \setminus Z$, we have inclusions
\[
	H^0 \bigl( X_y, \omega_{X_y} \tensor L_y \tensor \shI(h_y) \bigr)
		\subseteq E_y \subseteq H^0 \bigl( X_y, \omega_{X_y} \tensor L_y \bigr).
\]
\end{lemma}

\begin{proof}
If $h_y \equiv +\infty$, then the subspace of the left is trivial, which means that
the asserted inclusion is true by default. If $h_y$ is not identically equal to
$+\infty$, then given $\alpha \in H^0 \bigl( X_y, \omega_{X_y} \tensor L_y
\tensor \shI(h_y) \bigr)$ and a suitable open neighborhood $U$ of the point $y$, 
there is by \theoremref{thm:OT} some 
\[
	\beta \in H^0 \bigl( U, \omY \tensor \shF \bigr)
		\simeq H^0 \bigl( f^{-1}(U), \omX \tensor L \tensor \shI(h) \bigr)
\]
such that $\beta \restr{X_y} = \alpha \wedge df$. Since $\omY$ is trivial on
$U$, this gives us a section of $\shF$ in a neighborhood of the fiber $X_y$ whose
restriction to $X_y$ agrees with $\alpha$.
\end{proof}

\begin{note}
We will see in a moment that the two subspaces
\[
	H^0 \bigl( X_y, \omega_{X_y} \tensor L_y \tensor \shI(h_y) \bigr) \subseteq E_y
\]
are equal for almost every $y \in Y \setminus Z$. But unless $\shF = 0$, the two
subspaces are different for example at points where $h_y$ is identically equal to
$+\infty$.
\end{note}

We can now define on each $E_y$ with $y \in Y \setminus Z$ a singular hermitian
inner product in the following manner. Given an element
\[
	\alpha \in E_y \subseteq H^0 \bigl( X_y, \omega_{X_y} \tensor L_y \bigr),
\]
we can integrate over the compact complex manifold $X_y$ and define
\[
	\abs{\alpha}_{H,y}^2 = \int_{X_y} \abs{\alpha}_{h_y}^2 
		\, \in \, [0, +\infty].
\]
It is easy to see that $\abs{\argbl}_{H,y}$ is a positive definite singular hermitian
inner product. Clearly $\abs{\alpha}_{H,y} < +\infty$ if and only if $\alpha \in H^0
\bigl( X_y, \omega_{X_y} \tensor L_y \tensor \shI(h_y) \bigr)$; in light of
\lemmaref{lem:Fubini}, our singular hermitian inner product $\abs{\argbl}_{H,y}$ is
therefore finite precisely on the subspace $H^0 \bigl( X_y, \omega_{X_y} \tensor L_y
\tensor \shI(h_y) \bigr) \subseteq E_y$. 

Let us now analyze how the individual singular hermitian inner products
$\abs{\argbl}_{H,y}$ fit together on $Y \setminus Z$. Fix a point $y \in Y \setminus
Z$ and an open neighborhood $U \subseteq Y \setminus Z$ biholomorphic to the open unit
ball $B \subseteq \CC^r$; after pulling everything back to $U$, we may assume without loss of
generality that $Y = B$ and $Z = \emptyset$ and $y = 0$. Denote by $t_1, \dotsc, t_r$
the standard coordinate system on $B$; then the canonical bundle $\omega_B$ is
trivialized by the global section $\dt_1 \wedge \dotsb \wedge \dt_r$, and the volume
form on $B$ is 
\[
	d\mu = c_r (\dt_1 \wedge \dotsb \wedge \dt_r) 
		\wedge (\dtb_1 \wedge \dotsb \wedge \dtb_r).
\]
Fix a holomorphic section $s \in H^0(B, E)$, and denote by 
\[
	\beta = s \wedge (\dt_1 \wedge \dotsb \wedge \dt_r) \in 
		H^0 \bigl( B, \omega_B \tensor E \bigr)
		\simeq H^0 \bigl( X, \omX \tensor L \tensor \shI(h) \bigr)
\]
the corresponding holomorphic $n$-form on $X$ with coefficients in $L$. Since $f
\colon X \to B$ is smooth, Ehresmann's fibration theorem shows that $X$ is
diffeomorphic to the product $B \times X_0$. After choosing a K\"ahler metric
$\omega_0$ on $X_0$, we can write
\begin{equation} \label{eq:Ehresmann}
	\abs{\beta}_h^2 = F \cdot d\mu \wedge \frac{\omega_0^{n-r}}{(n-r)!},
\end{equation}
where $F \colon B \times X_0 \to [0, +\infty]$ is lower semi-continuous and locally
integrable; the reason is of course that the local weight functions for $(L, h)$ are
upper semi-continuous functions. At every point $y \in B$, we then have by construction
\begin{equation} \label{eq:fiber-integral}
	\abs{s(y)}_{H,y}^2 = \int_{X_0} F(y, \argbl) \, \frac{\omega_0^{n-r}}{(n-r)!}.
\end{equation}
By Fubini's theorem, the function $y \mapsto \abs{s(y)}_{H,y}$ is measurable;
moreover, since $F$ is locally integrable and $X_0$ is compact, we must have
$\abs{s(y)}_{H,y} < +\infty$ for almost every $y \in B$. Being coherent, $E$ is
generated over $B$ by a finite number of global sections; the singular hermitian
inner product $\abs{\argbl}_{H,y}$ is therefore finite and positive-definite for
almost every $y \in B$, hence for almost every $y \in Y \setminus Z$. In particular,
the first inclusion in \lemmaref{lem:Fubini} is an equality for almost every $y \in Y
\setminus Z$. We may summarize the conclusion as follows.

\begin{proposition}
On $Y \setminus Z$, the singular hermitian inner products $\abs{\argbl}_{H,y}$
determine a singular hermitian metric on the holomorphic vector bundle $E$.
\end{proposition}

While we are not yet ready to show that $(E, H)$ has semi-positive curvature, we can
already show that the function $\abs{s}_H$ is always lower semi-continuous.

\begin{proposition} \label{prop:LSC}
For any open subset $U \subseteq Y \setminus Z$ and any section $s \in H^0(U, E)$,
the function $\abs{s}_H \colon U \to [0, +\infty]$ is lower semi-continuous.
\end{proposition}

\begin{proof}
As before, we may assume that $U = B$ is the open unit ball in $\CC^m$; it is clearly
sufficient to show that $\abs{s}_H$ is lower semi-continuous at the origin. In other
words, we need to argue that
\[
	\abs{s(0)}_{H, 0} \leq \liminf_{k \to +\infty} \, \abs{s(y_k)}_{H, y_k}
\]
holds for every sequence $y_0, y_1, y_2, \dotsc \in B$ that converges to the
origin. As in \eqref{eq:Ehresmann}, the given section $s \in H^0(B, E)$ determines a
lower semi-continuous function $F \colon B \times X_0 \to [0, +\infty]$ such that
\eqref{eq:fiber-integral} is satisfied. By the lower
semi-continuity of $F$ and  Fatou's lemma, we obtain
\begin{align*}
	\int_{X_0} F(0, \argbl) \frac{\omega_0^{n-r}}{(n-r)!}
		&\leq \int_{X_0} \liminf_{k \to +\infty} F(y_k, \argbl)
\frac{\omega_0^{n-r}}{(n-r)!} \\
		&\leq \liminf_{k \to +\infty} 
			\int_{X_0} F(y_k, \argbl) \frac{\omega_0^{n-r}}{(n-r)!},
\end{align*}
which is the desired inequality up to taking square roots.
\end{proof}

\subsection{Proof of the pushforward theorem, Part II}

Having defined $(E, H)$ on the open subset $Y \setminus Z$, our next task is to
say something about the induced singular hermitian metric $\Hd$ on the dual
vector bundle $\Ed$. In particular, we need to prove that the norm of any local
section of $\shFd$ is uniformly bounded in the neighborhood of any point in $Z$, and
that its logarithm is an upper semi-continuous function. This part of the argument
relies on the existence of a uniform bound in the Ohsawa-Takegoshi theorem, but not
on the precise value of the constant. Let us start by reformulating the statement of
the Ohsawa-Takegoshi in terms of the pair $(E, H)$.

\begin{lemma} \label{lem:extension}
For every embedding $\iota \colon B \into Y$ with $y = \iota(0) \in Y \setminus Z$,
and for every $\alpha \in E_y$ with $\abs{\alpha}_{H,y} = 1$, there is a holomorphic
section $s \in H^0(B, \iota^{\ast} \shF)$ with
\[
	s(0) = \alpha \quad \text{and} \quad 
	\int_B \abs{s}_H^2 \, d\mu \leq C_0,
\]
where $C_0$ is the same constant as in the Ohsawa-Takegoshi theorem.
\end{lemma}

\begin{proof}
After pulling everything back to $B$, we may assume that $Y = B$ and $y = 0$. Since
$\abs{\alpha}_{H,0} = 1$, by \theoremref{thm:OT} there exists an element $\beta \in
H^0 \bigl( X, \omX \tensor L \tensor \shI(h) \bigr)$ with
\[
	\beta \restr{X_0} = \alpha \wedge df \quad \text{and} \quad
		\norm{\beta}_h^2 = \int_X \abs{\beta}_h^2 \leq C_0.
\]
In fact, one can take $C_0 = \mu(B)$, but the exact value of the constant is not
important here.  Using $\dt_1 \wedge \dotsb \wedge \dt_r$ as a trivialization of the
canonical bundle $\omega_B$, we may consider $\beta$ as a holomorphic section $s \in
H^0(B, \shF)$; the two conditions from above then turn into
\[
	s(0) = \alpha \quad \text{and} \quad
		\int_B \abs{s}_H^2 \, d\mu \leq C_0,
\]	
due to the fact that $d\mu = c_r (\dt_1 \wedge \dotsb \wedge \dt_r) \wedge (\dtb_1
\wedge \dotsb \wedge \dtb_r)$.
\end{proof}

Fix an open subset $U \subseteq Y$ and a holomorphic section $g \in H^0(U, \shFd)$;
after replacing $Y$ by the open subset $U$, we may assume without loss of generality
that $g \in H^0(Y, \shFd)$. Consider the measurable function
\begin{equation} \label{eq:psi}
	\psi = \log \abs{g}_{\Hd} \colon Y \setminus Z \to [-\infty, +\infty].
\end{equation}
Ultimately, our goal is to show that $\psi$ extends to a plurisubharmonic function on
all of $Y$. The following boundedness result is the crucial step in this direction.

\begin{proposition} \label{prop:boundedness}
Every point in $Y$ has an open neighborhood $U \subseteq Y$ such that $\psi = \log
\abs{g}_{\Hd}$ is bounded from above by a constant on $U \setminus U \cap Z$.
\end{proposition}

\begin{proof}
Choose two sufficiently small open neighborhoods $U \subseteq V \subseteq Y$ of the
given point, such that $\wbar{V}$ is compact, $\wbar{U} \subseteq V$, and for every
point $y \in \wbar{U}$, there is an embedding $\iota \colon B \into Y$ of the unit ball $B
\subseteq \CC^r$ with $\iota(0) = y$ and $ \iota(B) \subseteq V$. We shall argue that
there is a constant $C \geq 0$ such that $\psi \leq C$ on $U \setminus U \cap Z$.

Fix a point $y \in \wbar{U} \setminus Z$. If $\psi(y) = -\infty$, there is nothing to
prove, so let us suppose from now on that $\psi(y) \neq -\infty$. By definition of
the metric on the dual bundle, we can then find a vector $\alpha \in E_y$ with
$\abs{\alpha}_{H,y} = 1$ such that 
\[
	\psi(y) = \log \bigabs{g(\alpha)}.
\]
Choose an embedding $\iota \colon B \into Y$ such that $\iota(0) = y$ and $ \iota(B) \subseteq V$. 
Using \lemmaref{lem:extension}, we obtain a holomorphic section
$s \in H^0(V, \shF)$ with $s(0) = \alpha$ and
\[
	\int_V \abs{s}_H^2 \, d\mu \leq C_0;
\]
the integrand is of course only defined on the subset $V \setminus V \cap Z$, but
this does not matter because $V \cap Z$ has measure zero. It follows that $\psi(y)$
is equal to the value of $\log \abs{g(s)}$ at the point $y$, and so the desired upper
bound for $\psi$ is a consequence of
\lemmaref{lem:compactness} below.
\end{proof}

\begin{lemma} \label{lem:compactness}
Fix $K \geq 0$, and consider the set
\[
	S_K = \MENGE{s \in H^0(V, \shF)}{\int_V \abs{s}_H^2 \, d\mu \leq K}.
\]
There is a constant $C \geq 0$ such that, for every section $s \in S_K$, the
holomorphic function $g(s)$ is uniformly bounded by $C$ on the compact set $\wbar{U}$.
\end{lemma}

\begin{proof}
Since $g(s)$ is holomorphic on $V$, it is clear that each individual function $g(s)$
is bounded on $\wbar{U}$. To get an upper bound that works for every $s \in
S_K$ at once, we use a compactness argument. 
Given a section $s \in H^0(V, \shF)$, we invert the process from above and define
\[
	\beta = s \tensor \bigl( \dt_1 \wedge \dotsb \wedge \dt_r \bigr)
		\in H^0(V, \omY \tensor \shF) 
		= H^0 \bigl( f^{-1}(V), \omX \tensor L \tensor \shI(h) \bigr).
\]
If $s \in S_K$, then one has
\[
	\norm{\beta}_h^2 = \int_V \abs{s}_H^2 \, d\mu \leq K.
\]
Because $\wbar{V}$ is compact and $f$ is proper, we can cover $f^{-1}(V)$ by
finitely many open sets $W$ that are biholomorphic to the open unit ball in $\CC^n$,
and on which $L$ is trivial. Let $z_1, \dotsc, z_n$ be a holomorphic coordinate
system on $W$, choose a nowhere vanishing holomorphic section $s_0 \in H^0(W, L)$,
and write $\abs{s_0}_h^2 = e^{-\varphi}$, with $\varphi$ plurisubharmonic on $W$. Then
$\beta \restr{W} = b s_0 \tensor \dz_1 \wedge \dotsb \wedge \dz_n$ for some
holomorphic function $b \in H^0(W, \shO_W)$, and 
\[
	\int_W \abs{b}^2 e^{-\varphi} (\dx_1 \wedge \dy_1) \wedge \dotsb \wedge (\dx_n \wedge \dy_n)
	= \int_W \abs{\beta}_h^2 \leq K.
\]	
As we are dealing with finitely many open sets,
\propositionref{prop:Montel} shows that every sequence in $S_K$ has a subsequence that
converges uniformly on compact subsets to some $\beta \in H^0 \bigl( f^{-1}(V), \omX
\tensor L \tensor \shI(h) \bigr)$. This is all that we need.

Indeed, suppose that the assertion was false. Then we could find a sequence $s_0, s_1,
s_2, \dotsc \in S_K$ such that the maximum value of $\abs{g(s_k)}$ on the compact set
$\wbar{U}$ was at least $k$. Let $\beta_0, \beta_1, \beta_2, \dotsc$ denote the
corresponding sequence of holomorphic sections of $\omX \tensor L \tensor \shI(h)$ on
the open set $f^{-1}(V)$; after passing to a subsequence, the $\beta_k$ will converge
uniformly on compact subsets to $\beta \in H^0 \bigl( f^{-1}(V), \omX \tensor L
\tensor \shI(h) \bigr)$. Let $s \in H^0(V, \shF)$ be the unique section of $\shF$
such that 
\[
	\beta = s \tensor \bigl( \dt_1 \wedge \dotsb \wedge \dt_r \bigr).
\]
By \propositionref{prop:proper-mapping}, the $s_k$ converge to $s$ in the
Fr\'echet space topology on $H^0(V, \shF)$. Since $g \colon \shF \to \OY$ is a
morphism, the holomorphic functions $g(s_k)$ therefore converge uniformly on
compact subsets to $g(s)$. But then $\abs{g(s_k)}$ must be
uniformly bounded on $\wbar{U}$, contradicting our initial choice.
\end{proof}

The next step is to show that the function $\psi = \log \abs{g}_{\Hd}$ is upper
semi-continuous on $Y \setminus Z$. The proof is similar to that of
\propositionref{prop:LSC}.

\begin{proposition} \label{prop:USC}
For every $g \in H^0(Y, \shFd)$, the function $\psi = \log \abs{g}_{\Hd}$ is upper
semi-continuous on $Y \setminus Z$.
\end{proposition}

\begin{proof}
After restricting everything to a suitable open neighborhood of any given point $y
\in Y \setminus Z$, we may assume without loss of generality that $Y = B$ and $Z =
\emptyset$ and $y = 0$. Then $g \in H^0(B, \Ed)$, and it will be enough to show that
$\psi = \log \abs{g}_{\Hd}$ is upper semi-continuous at the origin. In other
words, we need to argue that
\begin{equation} \label{eq:limsup}
	\limsup_{k \to +\infty} \psi(y_k) \leq \psi(0)
\end{equation}
for every sequence $y_0, y_1, y_2, \dotsc \in B$ that converges to the origin. We may assume that $\psi(y_k) \neq -\infty$ for all $k \in \NN$,
and that the sequence $\psi(y_k)$ actually has a limit. 

As we saw before, there is, for each $k \in \NN$, a holomorphic section $s_k \in
H^0(B, E)$ such that $\psi(y_k)$ equals the value of $\log \abs{g(s_k)}$ at the point
$y_k$; the Ohsawa-Takegoshi theorem allows us to choose these sections in such a way that
\[
	\abs{s_k(y_k)}_{H, y_k} = 1 \quad \text{and} \quad
		\int_B \abs{s_k}_H \, d\mu \leq K
\]
for some constant $K \geq 0$. Passing to a subsequence, if necessary, we can arrange
that the $s_k$ converge uniformly on compact subsets to some $s \in H^0(B, E)$. Then
the holomorphic functions $g(s_k)$ converge uniformly on compact subsets to $g(s)$,
and \eqref{eq:limsup} reduces to showing that the value at the origin of $\log
\abs{g(s)}$ is less or equal to $\psi(0)$. By definition of the dual metric $\Hd$, we
have
\[
	\psi \geq \log \abs{g(s)} - \log \abs{s}_H,
\]
and so this is equivalent to proving that $\abs{s(0)}_{H, 0} \leq 1$. 
As in \eqref{eq:Ehresmann} and \eqref{eq:fiber-integral}, each $s_k$ determines a lower
semi-continuous function $F_k \colon B \times X_0 \to [0, +\infty]$ with
\[
	1 = \abs{s_k(y_k)}_{H,y_k}^2 
		= \int_{X_0} F_k(y_k, \argbl) \frac{\omega_0^{n-r}}{(n-r)!}.
\]
Likewise, $s$ determines a lower semi-continuous function $F \colon B \times X_0 \to [0,
+\infty]$. Since the local weight functions $e^{-\varphi}$ of the pair
$(L, h)$ are lower semi-continuous, and since $s_k$ converges uniformly on compact
subsets to $s$, we get
\[
	F(0, \argbl) \leq \liminf_{k \to +\infty} F_k(y_k, \argbl).
\]	
We can now apply Fatou's lemma and conclude the proof in the same way as in
\propositionref{prop:LSC}.
\end{proof}

\subsection{Proof of the pushforward theorem, Part III}
\label{subsec:proof-III}

In this section, we complete the proof of \theoremref{thm:pushforward} by showing
that the pair $(E, H)$ has semi-positive curvature, and that $H$ extends to a
singular hermitian metric on $\shF$ with the minimal extension property. The key
point is that we can prove the required mean-value inequalities because the optimal
value of the constant in the Ohsawa-Takegoshi theorem is exactly the volume of the
unit ball. To illustrate how this works, let us first show that the singular
hermitian metric $H$ on $Y \setminus Z$ has the minimal extension property (see
\parref{par:MEP}). For the statement, recall that $r = \dim Y$, and that $B \subseteq
\CC^r$ is the open unit ball.

\begin{proposition} \label{prop:minimal-extension}
For every embedding $\iota \colon B \into Y$ with $y = \iota(0) \in Y \setminus Z$,
and for every $\alpha \in E_y$ with $\abs{\alpha}_{H,y} = 1$, there is a holomorphic
section $s \in H^0(B, \iota^{\ast} \shF)$ with
\[
	s(0) = \alpha \quad \text{and} \quad 
	\frac{1}{\mu(B)} \int_B \abs{s}_H^2 \, d\mu \leq 1.
\]
\end{proposition}

\begin{proof}
The proof is the same as that of \lemmaref{lem:extension}; we only need to replace
the constant $C_0$ by its optimal value $\mu(B)$.
\end{proof}

Now let us prove that $H$ extends to a singular hermitian metric on $\shF$ with
semi-positive curvature. Keeping the notation from above, this amounts to proving
that the function $\psi \colon Y \setminus Z \to [-\infty, +\infty)$ in
\eqref{eq:psi} extends to a plurisubharmonic function on $Y$. We already know
that $\psi$ is upper semi-continuous (by \propositionref{prop:USC}) and bounded from
above in a neighborhood of every point in $Y$ (by \propositionref{prop:boundedness}).
What we need to prove is the mean-value inequality along holomorphic arcs in $Y
\setminus Z$.  The Ohsawa-Takegoshi theorem with sharp estimates
renders the proof of the mean-value inequality almost a triviality.

\begin{proposition} \label{prop:mean-value}
For every holomorphic mapping $\gamma \colon \Delta \to Y \setminus Z$, the function
$\psi = \log \abs{g}_{\Hd}$ satisfies the mean-value inequality
\[
	(\psi \circ \gamma)(0) 
		\leq \frac{1}{\pi} \int_{\Delta} (\psi \circ \gamma) \, d\mu.
\]
\end{proposition}

\begin{proof}
If $h$ is identically equal to $+\infty$ on the preimage of $\gamma(\Delta)$, the
inequality is clear, so we may assume that this is not the case.
Since $f \colon X \to Y$ is submersive over $Y \setminus Z$, we may then pull
everything back to $\Delta$ and reduce the problem to the case $Y = \Delta$. 
If $\psi(0) = -\infty$, then the mean-value inequality holds by
default. Assuming from now on that $\psi(0) \neq -\infty$, we choose an element
$\alpha \in E_0$ with $\abs{\alpha}_{H,0} = 1$, such that
\[
	\psi(0) = \log \abs{g}_{\Hd,0} = \log \abs{g(\alpha)}.
\]
Using the minimal extension property (in \propositionref{prop:minimal-extension},
with $m = 1$), there is a holomorphic section $s \in H^0(\Delta, E)$ such that
\[
	s(0) = \alpha \quad \text{and} \quad 
	\frac{1}{\pi} \int_{\Delta} \abs{s}_H^2 \, d\mu \leq 1.
\]
The existence of this section is all that we need to prove the mean-value inequality.
By definition of the metric $\Hd$ on the dual bundle, we have the pointwise
inequality
\[
	\abs{g}_{\Hd} \geq \frac{\abs{g(s)}}{\abs{s}_H}
\]
and therefore $2\psi \geq \log \abs{g(s)}^2 - \log \abs{s}_H^2$; here $g(s)$ is a
holomorphic function on $\Delta$, whose value at the origin equals $g(\alpha)$.
Integrating, we get
\[
	\frac{1}{\pi} \int_{\Delta} 2\psi \, d\mu
	\geq \frac{1}{\pi} \int_{\Delta} \log \abs{g(s)}^2 \, d\mu
		- \frac{1}{\pi} \int_{\Delta} \log \abs{s}_H^2 \, d\mu
\]
Now $\log \abs{g(s)}^2$ satisfies the mean-value inequality, and so the first term on
the right-hand side is at least $\log \abs{g(\alpha)}^2 = 2\psi(0)$. Since the
function $x \mapsto -\log x$ is convex, and since the function $\abs{s}_H^2$ is
integrable, the second term can be estimated by Jensen's inequality to be at least
\[
	- \log \left( \frac{1}{\pi} \int_{\Delta} \abs{s}_H^2 \, d\mu \right) 
		\geq - \log 1 = 0.
\]
Putting everything together, we obtain
\[
	\frac{1}{\pi} \int_{\Delta} 2\psi \, d\mu \geq 2\psi(0),
\]
which is the mean-value inequality (up to a factor of $2$).
\end{proof}

We have verified that $\psi$ is plurisubharmonic on $Y \setminus Z$. We already know from
\propositionref{prop:boundedness} that $\psi$ is locally bounded from above in a
neighborhood of every point in $Y$;
consequently, it extends uniquely to a plurisubharmonic function on all of $Y$, using
\lemmaref{lem:Riemann-Hartogs}. By duality, the singular hermitian metric $H$ is
therefore well-defined on the entire open set $Y(\shF)$ where the sheaf
$\shF = \fl \bigl( \omX \tensor L \tensor \shI(h) \bigr)$ is locally free. We have
already shown that $H$ has the minimal extension property. This
finishes the proof of \theoremref{thm:pushforward}. \qed

\begin{proof}[Proof of \corollaryref{cor:continuous}]
Suppose that $f \colon X \to Y$ is submersive and that the singular hermitian metric
$h$ on the line bundle $L$ is continuous. To prove that $H$ is continuous, it
suffices to show that for every locally defined section $s \in H^0(U, \shF)$, the
function $\abs{s}_H^2$ on $U \setminus U \cap Z$ admits a continuous extension to all
of $U$. This is a local problem, and so we may assume that $Y = B$ is the open unit
ball in $\CC^r$, with coordinates $t_1, \dotsc, t_r$, and that $s \in H^0(B, \shF)$. Define 
\[
	\beta = s \wedge (\dt_1 \wedge \dotsb \wedge \dt_r) 
		\in H^0(B, \omega_B \tensor \shF) 
		= H^0 \bigl( X, \omX \tensor L \tensor \shI(h) \bigr).
\]
By Ehresmann's fibration theorem, $X$ is diffeomorphic to the product $B \times X_0$,
and as in \eqref{eq:Ehresmann}, we can write
\[
	\abs{\beta}_h^2 = F \cdot d\mu \wedge \frac{\omega_0^{n-r}}{(n-r)!}
\]
with $F \colon B \times X_0 \to [0, +\infty]$ continuous. Now
\[
	y \mapsto \int_{X_0} F(y, \argbl) \frac{\omega_0^{n-r}}{(n-r)!}
\]
defines a continuous function on $B$ that agrees with $\abs{s}_H^2$ on the complement
of the bad set $Z$, due to \eqref{eq:fiber-integral}.
\end{proof}

\subsection{Positivity of the determinant line bundle}
\label{subsec:determinant}

In this section, we show that if a holomorphic vector bundle $E$ has a singular
hermitian metric with semi-positive curvature, then the determinant line bundle $\det
E$ has the same property. The proof in \cite[Proposition~1.1]{Raufi} relies on
locally approximating a given singular hermitian metric from below by smooth hermitian metrics
\cite[Proposition~3.1]{BP}.  

\begin{proposition} \label{prop:determinant}
If $(E, h)$ has semi-positive curvature, so does $(\det E, \det h)$.
\end{proposition}

Let us first analyze what happens over a point. Let $V$ be a complex vector space of
dimension $r$, and $\abs{\argbl}_h$ a positive definite singular hermitian inner
product on $V$; in the notation of \parref{par:SHIP}, we have $V_0 = 0$. 
Let $\PP(V)$ be the projective space parametrizing
one-dimensional quotient spaces of $V$, and denote by $\shO(1)$ the universal line
bundle on $\PP(V)$. We have a surjective morphism $V \tensor \shO \to \shO(1)$, and
so $h$ induces a singular hermitian metric on $\shO(1)$, with singularities along
the subspace $\PP(V/\Vfin) \subseteq \PP(V)$. To see this, choose a basis $e_1,
\dotsc, e_r \in V$ such that $e_1, \dotsc, e_k$ form an orthonormal basis of $\Vfin$
with respect to the inner product $\inner{\argbl}{\argbl}_h$, and denote by $[z_1,
\dotsc, z_r]$ the resulting homogeneous coordinates on $\PP(V)$. Then the local
weight functions of the metric on $\shO(1)$ are given by the formula
\[
	\log \bigl( \abs{z_1}^2 + \dotsb + \abs{z_k}^2 \bigr),
\]
with the convention that $z_i = 1$ on the $i$-th standard affine open subset. 

Now the one-dimensional complex vector space $\det V = \bigwedge^r V$ is naturally
the space of global sections of an adjoint bundle on $\PP(V)$, because
\[
	\det V \simeq H^0 \bigl( \PP(V), \omega_{\PP(V)} \tensor \shO(r) \bigr).
\]
The isomorphism works as follows. The element $e_1 \wedge \dotsb \wedge e_r \in \det
V$ determines a holomorphic $r$-form $\dz_1 \wedge \dotsb \wedge \dz_r$ on the dual
vector space $\Vd$; after contraction with the Euler vector field $z_1
\partial/\partial z_1 + \dotsb + z_r \partial/\partial z_r$, we
get a holomorphic $(r-1)$-form 
\[
	\Omega = \sum_{i=1}^r (-1)^{i-1} z_i \dz_1 \wedge \dotsb \wedge 
		\widehat{\dz_i} \wedge \dotsb \wedge \dz_r
\]
on $\PP(V)$ that is homogeneous of degree $r$, hence a global section of the
holomorphic line bundle $\omega_{\PP(V)} \tensor \shO(r)$. Integration over $\PP(V)$
therefore defines a positive definite singular hermitian inner product $H$ on $\det
V$. We have
\[
	\abs{e_1 \wedge \dotsb \wedge e_r}_H^2 = 
		\int_{\PP(V)} \frac{c_{r-1} \cdot \Omega \wedge \wbar{\Omega}}%
		{\bigl( \abs{z_1}^2 + \dotsb + \abs{z_k}^2 \bigr)^r},
\]
which simplifies in the affine chart $z_1 = 1$ to
\[
	\abs{e_1 \wedge \dotsb \wedge e_r}_H^2 = 
	\int_{\CC^{r-1}} \frac{d\mu}{\bigl( 1 + \abs{z_2}^2 + \dotsb + \abs{z_k}^2 \bigr)^r}.
\]
Now there are two cases. If $\Vfin \neq V$, then $k < r$, and the integral is easily
seen to be $+\infty$. If $\Vfin = V$, then $k = r$, and the integral evaluates to
$\pi^{r-1}/(r-1)!$, the volume of the open unit ball in $\CC^{r-1}$. In conclusion,
we always have
\[
	\abs{e_1 \wedge \dotsb \wedge e_r}_H^2 = \frac{\pi^{r-1}}{(r-1)!}
		\cdot \abs{e_1 \wedge \dotsb \wedge e_r}_{\det h}^2.
\]
With this result in hand, we can now prove \propositionref{prop:determinant}.

\begin{proof}
Let $p \colon \PP(E) \to X$ denote the associated $\PP^{r-1}$-bundle, and
let $\shO_E(1)$ be the universal line bundle on $\PP(E)$. We have a surjective
morphism $\pu E \to \shO_E(1)$, and by \lemmaref{lem:quotient}, the singular
hermitian metric on $E$ induces a singular hermitian metric on the line bundle
$\shO_E(1)$, still with semi-positive curvature. We have
\[
	\omega_{\PP(E)/X} \simeq \pu \det E \tensor \shO_E(-r),
\]
and therefore $\det E \simeq \pl \bigl( \omega_{\PP(E)/X} \tensor \shO_E(r) \bigr)$
is the pushforward of an adjoint bundle. The calculation above shows that,
up to a factor of $\pi^{r-1}/(r-1)!$, the resulting singular hermitian metric on
$\det E$ agrees with $\det h$ pointwise. The assertion about the curvature of $(\det
E, \det h)$ is therefore a consequence of \corollaryref{cor:pushforward}.
\end{proof}

\subsection{Consequences of the minimal extension property}

In this section, we derive a few interesting consequences from the minimal extension
property. All of the results below are true for smooth hermitian metrics with Griffiths
semi-positive curvature on holomorphic vector bundles; the minimal extension property
is what makes them work even in the presence of singularities.

Let $\shF$ be a torsion-free coherent sheaf on $X$, of generic rank $r \geq
1$, and suppose that $\shF$ has a singular hermitian metric with semi-positive
curvature and the minimal extension property. 
Let $E$ be the holomorphic vector bundle of rank $r$ obtained by
restricting $\shF$ to the open subset $X(\shF)$; by assumption, the pair $(E, h)$ has
semi-positive curvature.  \propositionref{prop:determinant} shows that $(\det E, \det h)$
also has semi-positive curvature. Let $\det \shF$ be the holomorphic line bundle obtained as
the double dual of $\bigwedge^r \! \shF$; its restriction to $X(\shF)$ agrees with $\det
E$. Since $X \setminus X(\shF)$ has codimension $\geq 2$, the singular hermitian
metric on $\det E$ extends uniquely to a singular hermitian metric on $\det \shF$.
The following result is due to Cao and P\u{a}un \cite[Theorem 5.23]{CP}, who proved
it using results by Raufi \cite{Raufi}.

\begin{theorem} \label{thm:det-trivial}
Suppose that $X$ is compact and that $c_1(\det \shF) = 0$ in $H^2(X, \RR)$. Then 
$\shF$ is locally free, and $(E, h)$ is a hermitian flat bundle on $X = X(\shF)$.
\end{theorem}

\begin{proof}
Since $X$ is compact, the singular hermitian metric on $\det \shF$ is smooth and has 
zero curvature (by \lemmaref{lem:trivial}). Restricting to the open subset $X(\shF)$,
we see that the same is true for $(\det E, \det h)$. Now the idea is to use the
minimal extension property to construct, locally on $X$, a collection of $r$ sections
of $\shF$ that are orthonormal with respect to $h$.

We can certainly cover $X$ by open subsets that are isomorphic to the open unit ball
$B \subseteq \CC^n$ and are centered at points $x \in X \setminus Z$ where the
singular hermitian inner product $\abs{\argbl}_{h,x}$ is finite and positive definite. 
After restricting everything to
an open subset of this kind, we may assume that $X = B$, that the point $0 \in B$
lies in the subset $B \setminus Z$, and that $\abs{\argbl}_{h,0}$ is a genuine
hermitian inner product on the $r$-dimensional complex vector space $E_0$. Choose an
orthonormal basis $e_1, \dotsc, e_r \in E_0$. By the minimal extension property for
$\shF$, we can find $r$ holomorphic sections $s_1, \dotsc, s_r \in H^0(B, \shF)$ such that 
\[
	s_i(0) = e_i \quad \text{and} \quad
	\frac{1}{\mu(B)} \int_B \abs{s_i}_h^2 \, d\mu \leq 1.
\]
Since the logarithm function is strictly concave, Jensen's inequality shows that
\begin{equation} \label{eq:inequality-1}
	\frac{1}{\mu(B)} \int_B \log \abs{s_i}_h^2 \, d\mu \leq
	\log \left( \frac{1}{\mu(B)} \int_B \abs{s_i}_h^2 \, d\mu \right) \leq 0,
\end{equation}
with equality if and only if $\abs{s_i}_h = 1$ almost everywhere.

Now let us analyze the singular hermitian metric on $\det E$. The expression
\[
	H_{i,j}(x) = \inner{s_i(x)}{s_j(x)}_{h, x}
\]
is well-defined outside a set of measure zero, and the resulting function is
locally integrable. Denote by $H(x)$ the $r \times r$-matrix with these entries; it
is almost everywhere positive definite, and we have
\[
	\bigabs{s_1 \wedge \dotsb \wedge s_r}_{\det h}^2 = \det H.
\]
Since $\det h$ is actually smooth and flat, we can choose a nowhere vanishing section
$\delta \in H^0(B, \det \shF)$ such that $\abs{\delta}_{\det h} \equiv 1$. We then
have $s_1 \wedge \dotsb \wedge s_r = g \cdot \delta$ for a holomorphic function $g
\in H^0(B, \shO_B)$ with $g(0) = 1$, and 
\[
	\abs{g}^2 = \bigabs{s_1 \wedge \dotsb \wedge s_r}_{\det h}^2 = \det H.
\]
From Hadamard's inequality for semi-positive definite matrices, we obtain
\[
	\abs{g(x)}^2 = \det H(x) \leq \prod_{i=1}^r H_{i,i}(x)
		= \prod_{i=1}^r \bigabs{s_i(x)}_{h,x}^2,
\]
with equality if and only if the matrix $H(x)$ is diagonal. Taking logarithms, we get
\[
	\log \abs{g(x)}^2 \leq \sum_{i=1}^r \log \bigabs{s_i(x)}_{h,x}^2.
\]
This inequality is valid almost everywhere; integrating, we find that
\begin{equation} \label{eq:inequality-2}
	\frac{1}{\mu(B)} \int_B \log \abs{g}^2 \, d\mu \leq
		\sum_{i=1}^r \frac{1}{\mu(B)} \int_B \abs{s_i}_h^2 \, d\mu.
\end{equation}
Now $\log \abs{g}^2$ is plurisubharmonic, and so the mean-value inequality shows that
the left-hand side in \eqref{eq:inequality-2} is greater or equal to $\log
\abs{g(0)}^2 = 0$. At the same time, the right-hand side is less or equal to $0$ by
\eqref{eq:inequality-1}. The conclusion is that all our inequalities are actually
equalities, and so $H(x)$ is almost everywhere equal to the identity matrix of size
$r \times r$. In other words, the sections $s_1, \dotsc, s_r \in H^0(B, \shF)$ are
almost everywhere orthonormal with respect to $h$.

For any holomorphic section $f \in H^0(B, \shFd)$, we therefore have 
\[
	\abs{f}_{\hd}^2 = \sum_{i=1}^r \abs{f \circ s_i}^2
\]
almost everywhere on $B$; because the logarithms of both sides are plurisubharmonic
functions on $B$, the identity actually holds everywhere. The singular hermitian
metric $\hd$ is therefore smooth; but then $h$ is also smooth, and the pair $(E, h)$
is a hermitian flat bundle.

To conclude the proof, we need to argue that $\shF$ is locally free on all of $B$.
The sections $s_1, \dotsc, s_r \in H^0(B, \shF)$ give rise to a morphism of sheaves
\[
	\sigma \colon \shO_B^{\oplus r} \to \shF.
\]
We already know that $\sigma$ is an isomorphism on the open subset $B(\shF)$; by
Hartog's theorem, its inverse extends to a morphism of sheaves
\[
	\tau \colon \shF \to \shO_B^{\oplus r}
\]
with $\tau \circ \sigma = \id$. Because $\shF$ is torsion-free, this forces $\sigma$
to be an isomorphism.
\end{proof}

\begin{note}
Our proof gives a different interpretation for the fact that $(\det E, \det h)$ has
semi-positive curvature. Indeed, without assuming that $\det h$ is smooth and flat,
we have $\det H = \abs{g}^2 e^{-\varphi}$, where $\varphi \colon B \to [-\infty,
+\infty)$ is locally integrable and $\varphi(0) = 0$. The various inequalities above
then combine to give
\[
	0 \leq \frac{1}{\mu(B)} \int_B \varphi \, d\mu,
\]
which is exactly the mean-value inequality for $\varphi$. 
\end{note}

The next theorem is a new result. It says that when $X$ is
compact, all global sections of the dual coherent sheaf $\shFd$ arise from trivial
summands in $\shF$.  Equivalently, every nonzero morphism $\shF \to \OX$ has a
section, which means that $\shF$ splits off a direct summand isomorphic to $\OX$.

\begin{theorem} \label{thm:quotient-split}
Suppose that $X$ is compact and connected. Then for every nonzero $f \in H^0(X,
\shFd)$, there exists a unique global section $s \in H^0(X, \shF)$ such that
$\abs{s}_h$ is a.e.~constant and $f \circ s \equiv 1$.
\end{theorem}

\begin{proof}
Because the singular hermitian metric on $\shF$ has semi-positive curvature, the
function $\log \abs{f}_{\hd}$ is plurisubharmonic on
$X$, hence equal to a nonzero constant. After rescaling the metric, we may assume
without loss of generality that $\abs{f}_{\hd} \equiv 1$. As in the proof of
\theoremref{thm:det-trivial}, we cover $X$ by open subsets that are isomorphic to $B
\subseteq \CC^n$ and are centered at points $x \in X \setminus Z$ where
$\abs{\argbl}_{h,x}$ is finite and positive definite. We shall
argue that there is a unique section of $\shF$ with the desired properties on each
open set of this type; by uniqueness, these sections will then glue together to give
us the global section $s \in H^0(X, \shF)$ that we are looking for.

We may therefore assume without loss of generality that $X = B$, that the origin
belongs to the subset $B \setminus Z$, and that $\abs{\argbl}_{h,0}$ is a hermitian
inner product on the vector space $E_0$. It is easy to see from
\[
	\sup \MENGE{\frac{\abs{f(v)}}{\abs{v}_{h,0}}}{%
		\text{$v \in E_0$ with $\abs{v}_{h,0} \neq 0$}}
	= \abs{f}_{\hd,0} = 1
\]
that there exists a vector $v \in E_0$ with $f(v) = 1$ and $\abs{v}_{h,0} = 1$.
By the minimal extension property, there is a section $s \in H^0(B, \shF)$ such that
\[
	s(0) = v \quad \text{and} \quad
	\frac{1}{\mu(B)} \int_B \abs{s}_h^2 \, d\mu \leq 1.
\]
Now $f \circ s$ is a holomorphic function on $B$, and by definition of $\hd$, we have
\[
	\frac{\abs{f \circ s}}{\abs{s}_h} \leq \abs{f}_{\hd} = 1.
\]
Taking logarithms and integrating, we get
\[
	\frac{1}{\mu(B)} \int_B \log \abs{f \circ s}^2 \, d\mu
		\leq \frac{1}{\mu(B)} \int_B \log \abs{s}_h^2 \, d\mu
		\leq \log \left( \frac{1}{\mu(B)} \int_B \abs{s}_h^2 \, d\mu \right) \leq 0,
\]
using Jensen's inequality along the way. By the mean-value inequality, the left-hand
side is greater or equal to $\log (f \circ s)(0) = 0$, and so once again, all
inequalities must be equalities. It follows that $f \circ s \equiv 1$, and that the
measurable function $\abs{s}_h$ is equal to $1$ almost everywhere. 

It remains to prove the uniqueness statement. Suppose that $s' \in H^0(B, \shF)$ is
another holomorphic section with the property that $f \circ s' \equiv 1$ and
$\abs{s'}_h = 1$ almost everywhere. Outside a set of measure zero, we have
\[
	\abs{s'-s}_h^2 + \abs{s'+s}_h^2 = 2 \abs{s}_h^2 + 2 \abs{s'}_h^2 = 4,
\]
and since $f(s'+s) = 2$, we must have $\abs{s'+s}_h^2 \geq 4$. This implies that
$\abs{s'-s}_h^2 = 0$ almost everywhere, and hence that $s' = s$.
\end{proof}

\section{Pushforwards of relative pluricanonical bundles}
\label{chap:NS}

\subsection{Introduction}

In the previous chapter, we presented a general formalism for constructing
singular hermitian metrics with semi-positive curvature on sheaves of the form
$\fl(\omXY \tensor L)$.  The applications to algebraic geometry come from the fact
that the sheaves $\fl \omXYm$ with $m \geq 2$ naturally fit into this framework. The
main result is the following; see \cite[Corollary~4.2]{BP}, and also \cite[Theorem~1.12]{Tsuji}, \cite[Theorem~4.2.2]{PT}.

\begin{theorem} \label{thm:pluricanonical}
Let $f \colon X \to Y$ be a surjective projective morphism with connected fibers between two
complex manifolds. Suppose that $\fl \omXYm \neq 0$ for some $m \geq 2$.
\begin{aenumerate}
\item The line bundle $\omXY$ has a canonical singular hermitian metric with
semi-positive curvature, called the \define{$m$-th Narasimhan-Simha metric}. 
This metric is continuous on the preimage of the smooth locus of $f$.
\item If $h$ denotes the induced singular hermitian metric on $L =
\omega_{X/Y}^{\otimes(m-1)}$, then
\[
	\fl \bigl( \omXY \tensor L \tensor \shI(h) \bigr) \into \fl \omXYm
\]
is an isomorphism over the smooth locus of $f$.
\end{aenumerate}
\end{theorem}

We can therefore apply \corollaryref{cor:pushforward} and conclude that for any $m
\geq 1$, the torsion-free sheaf $\fl \omXYm$ has a singular hermitian metric with
semi-positive curvature and the minimal extension property. Over the smooth locus of
$f$, this metric is finite and continuous. The minimal extension property
has the following remarkable consequences.

\begin{corollary}
Suppose that $Y$ is compact. 
\begin{aenumerate}
\item If $c_1 \bigl( \det \fl \omXYm \bigr) = 0$ in $H^2(Y, \RR)$, then $\fl \omXYm$
is locally free and the singular hermitian metric on it is smooth and flat.
\item Any nonzero morphism $\fl \omXYm \to \OY$ is split surjective.
\end{aenumerate}
\end{corollary}

\begin{proof}
This follows from \theoremref{thm:det-trivial} and \theoremref{thm:quotient-split}.
\end{proof}

\begin{note}
There are two or three points in the proof where we need to use invariance of
plurigenera. This means that \theoremref{thm:pluricanonical} cannot be used to give a
new proof for the invariance of plurigenera.
\end{note}

\subsection{The absolute case}
\label{par:absolute}

Let us start by discussing the absolute case. Take $X$ to be a smooth projective
variety of dimension $n$. Fix an integer $m \geq 1$ for which the vector space 
\[
	V_m = H^0(X, \omXm)
\] 
of all $m$-canonical forms is nontrivial. Our goal is to construct a singular
hermitian metric on the line bundle $\omX$, with singularities along the base locus
of $V_m$, such that all elements of $V_m$ have bounded norm. We can measure the
\define{length} of an $m$-canonical form $v \in V_m$ by a
real number $\ell(v) \in [0, +\infty)$, defined by the formula
\begin{equation} \label{eq:length}
	\ell(v) = \left( \int_X (c_n^m v \wedge \wbar{v})^{1/m} \right)^{m/2}.
\end{equation}
The constant $c_n = 2^{-n} (-1)^{n^2/2}$ is there to make the expression in
parentheses positive. A more concrete definition is as follows. In local coordinates
$z_1, \dotsc, z_n$, we have an expression
\[
	v = g(z_1, \dotsc, z_n) (\dz_1 \wedge \dotsb \wedge \dz_n)^{\tensor m},
\]
with $g$ holomorphic; the integrand in \eqref{eq:length} is then locally given by
\begin{equation} \label{eq:length-concrete}
	\abs{g}^{2/m} c_n (\dz_1 \wedge \dotsb \wedge \dz_n) 
		\wedge (\dzb_1 \wedge \dotsb \wedge \dzb_n).
\end{equation}
For $m \geq 2$, the length function $\ell$ is not a norm, because the triangle
inequality fails to hold. On the other hand, $\ell$ is continuous on
$V_m$, with $\ell(v) = 0$ iff $v = 0$; we also have $\ell(\lambda v) =
\abs{\lambda} \cdot \ell(v)$ for every $\lambda \in \CC$. 

We can now construct a singular hermitian metric $h_m$ on the line bundle $\omX$
by using the length function $\ell$. Given an element $\xi$ in the fiber of $\omX$
at a point $x \in X$, we define
\[
	\abs{\xi}_{h_m, x}
		= \inf \menge{\ell(v)^{1/m}}{\text{$v \in V_m$ satisfies $v(x) = \xi^{\tensor m}$}}
		\, \in \, [0, +\infty].
\]
In other words, we look for the $m$-canonical form of minimal length whose
value at the point $x$ is equal to the $m$-th power of $\xi$; if $x$ belongs to the
base locus of $V_m$, then $\abs{\xi}_{h_m,x} = +\infty$ for $\xi \neq 0$. We obtain
in this way a singular hermitian metric $h_m$ on the line bundle $\omX$, with
singularities precisely along the base locus of the linear system $V_m$. The
advantage of this construction is that it is completely canonical: there is no need
to choose a basis for $V_m$. 

\begin{note}
Following P\u{a}un and Takayama, we may call $h_m$ the \define{$m$-th Narasimhan-Simha
metric} on the line bundle $\omX$, because Narasimhan and Simha \cite{NS} used this metric in the special case
$\omX$ ample. A similar construction also appears in Kawamata's proof of Iitaka's
conjecture over curves \cite[\S2]{Ka-curve}.
\end{note}

\begin{proposition} \label{prop:NS-absolute}
The Narasimhan-Simha metric $h_m$ on $\omX$ is continuous, has 
singularities exactly along the base locus of $V_m$, and has semi-positive curvature.
\end{proposition}

\begin{proof}
We compute the local weights of $h_m$. Let $z_1, \dotsc, z_n$ be local holomorphic
coordinates on a suitable open subset $U \subseteq X$, and set $s_0 = \dz_1 \wedge
\dotsb \wedge \dz_n$, which is a nowhere vanishing section of $\omX$ on the subset $U$.
Consider the function
\[
	\varphi_m = - \log \abs{s_0}_{h_m}^2 \colon U \to [-\infty, +\infty).
\]
The definition of $h_m$ shows that, for every $x \in U$, 
\begin{equation} \label{eq:NS-weight-1}
	\varphi_m(x) 
		= \frac{2}{m} \sup \MENGE{\log \frac{1}{\ell(v)}}{%
			\text{$v \in V_m$ satisfies $v(x) = s_0(x)^{\tensor m}$}}.
\end{equation}
For each $v \in V_m$, there is a holomorphic function $g_v \colon U \to \CC$ with
$v \restr{U} = g_v \cdot s_0^{\tensor m}$. If $g_v(x) \neq 0$, then the $m$-canonical
form $v/g_v(x)$ contributes to the right-hand side of \eqref{eq:NS-weight-1}, and so
we obtain
\begin{equation} \label{eq:NS-weight-2}
	\varphi_m(x) = \frac{2}{m} \sup \Menge{\log \abs{g_v(x)}}{%
		\text{$v \in V_m$ satisfies $\ell(v) \leq 1$}}.
\end{equation}
We will see in a moment that the supremum is actually a maximum, because the set of
$m$-canonical forms $v \in V_m$ with $\ell(v) \leq 1$ is compact. Evidently,
$\varphi_m(x) = -\infty$ if and
only if $x \in U$ belongs to the base locus of $V_m$.

Now observe that the family of holomorphic functions
\[
	G_m = \menge{g_v \in H^0(U, \OX)}{\text{$v \in V_m$ satisfies $\ell(v) \leq 1$}}
\]
is uniformly bounded on compact subsets. Indeed, the fact that $\ell(v) \leq 1$ gives us
a uniform bound on the $L^{2/m}$-norm of each $g_v$, and then we can argue as in the
proof of \propositionref{prop:Montel}, using the mean-value inequality. By the
$n$-dimensional version of Montel's theorem, the family $G_m$ is equicontinuous;
due to \eqref{eq:NS-weight-2}, our $\varphi_m$ is therefore continuous, as a
function from $U$ into $[-\infty, +\infty)$.

From \eqref{eq:NS-weight-2}, we can also determine the curvature properties of
$h_m$. For each fixed $v \in V_m$, the function $\log \abs{g_v}^{2/m}$ is continuous and
plurisubharmonic, and equal to $-\infty$ precisely on the zero locus of $g_v$. As the
upper envelope of an equicontinuous family of plurisubharmonic functions, $\varphi_m$
is itself plurisubharmonic \cite[Theorem~I.5.7]{Demailly}. This shows that the 
Narasimhan-Simha metric on $\omX$ has semi-positive curvature.
\end{proof}

Another good feature of the Narasimhan-Simha metric is that all $m$-canonical forms
are bounded with respect to this metric. Indeed, if we also use $h_m$ to denote the
induced singular hermitian metric on $\omXm$, then by construction, we have the
pointwise inequality $\abs{v}_{h_m} \leq \ell(v)$ for every $v \in V_m$. In order to
fit the Narasimhan-Simha metric into the framework of \chapref{analytic}, we write
\[
	\omXm = \omX \tensor \omX^{\tensor(m-1)},
\]
and endow the line bundle $L = \omX^{\tensor(m-1)}$ with the singular hermitian
metric $h$ induced by $h_m$. This metric is continuous and has semi-positive
curvature. 

\begin{lemma} \label{lem:NS-shI}
For every $v \in V_m$, we have $\norm{v}_h \leq \ell(v)$.
\end{lemma}

\begin{proof}
We keep the notation introduced during the proof of
\propositionref{prop:NS-absolute}. The weight of $h$ with respect to the section
$s_0^{\tensor(m-1)}$ of the line bundle $\omX^{\tensor(m-1)}$ is
\[
	e^{-(m-1) \cdot \varphi_m},
\]
where $\varphi_m$ is the function defined in \eqref{eq:NS-weight-2}. Now fix an
$m$-canonical form $v \in V_m$ with $\ell(v) = 1$. On the open set
$U$, the integrand in the definition of $\norm{v}_h$ is
\[
	\abs{g_v}^2 e^{-(m-1) \cdot \varphi_m} 
		\cdot c_n (\dz_1 \wedge \dotsb \wedge \dz_n) 
			\wedge (\dzb_1 \wedge \dotsb \wedge \dzb_n).
\]
Because of \eqref{eq:NS-weight-2}, we have $\varphi_m \geq \log \abs{g_v}^{2/m}$, and
therefore
\[
	\abs{g_v}^2 e^{-(m-1) \cdot \varphi_m}
		\leq \abs{g_v}^2 \cdot \abs{g_v}^{-2(m-1)/m} = \abs{g_v}^{2/m}.
\]
Looking back at the definition of $\ell(v)$ in \eqref{eq:length-concrete}, this shows
that $\norm{v}_h \leq \ell(v)$.
\end{proof}

Since $\ell$ is not itself a norm, the inequality will in general be strict.
One useful consequence of \lemmaref{lem:NS-shI} is the identity
\begin{equation} \label{eq:identity}
	H^0 \bigl( X, \omX \tensor L \tensor \shI(h) \bigr) = H^0(X, \omX \tensor L)
		= H^0 \bigl( X, \omXm \bigr).
\end{equation}
Note that the multiplier ideal $\shI(h)$ may well be nontrivial; nevertheless, it
imposes no extra conditions on global sections of $\omXm$.

\subsection{The Ohsawa-Takegoshi theorem for pluricanonical forms}

To analyze how the Narasimhan-Simha metric behaves in families, we will need a
version of the Ohsawa-Takegoshi theorem for $m$-canonical forms. 
Suppose that $f \colon X \to B$ is a holomorphic mapping to the open unit ball $B
\subseteq \CC^r$, with $f$ projective and $f(X) = B$, and such that the central fiber
$X_0 = f^{-1}(0)$ is nonsingular. To simplify the discussion, let us also assume that
$f$ is the restriction of a holomorphic family over a ball of slightly larger radius.
As in \eqref{eq:length}, we have length functions $\ell$ and $\ell_0$ on $X$
respectively $X_0$; because $X$ is not compact, it may happen that $\ell(v) =
+\infty$ for certain $v \in H^0(X, \omXm)$.

\begin{theorem} \label{thm:OTm}
For each $u \in H^0 \bigl( X_0, \omega_{X_0}^{\tensor m} \bigr)$, there is some
$v \in H^0 \bigl( X, \omXm \bigr)$ with 
\[
	\ell(v) \leq \mu(B)^{m/2} \cdot \ell_0(u),
\]
such that the restriction of $v$ to $X_0$ is equal to $u \wedge (df_1 \wedge \dotsb
\wedge df_r)^{\tensor m}$.
\end{theorem}

\begin{proof}
Without loss of generality, we may assume that $\ell_0(u) = 1$. Since $X_0$ is
a projective complex manifold, invariance of plurigenera tells us that the fiber of
the coherent sheaf $\fl \omXm$ at the point $0 \in B$ is equal to $H^0 \bigl( X_0,
\omega_{X_0}^{\tensor m} \bigr)$. Because $B$ is a Stein manifold, we can then
certainly find a section 
\[
	v \in H^0 \bigl( B, \fl \omXm \bigr) = H^0(X, \omXm)
\]
with the correct restriction to $X_0$. By assuming that $f$ comes from a 
morphism to a ball of slightly larger radius, we can also arrange that the quantity
\[
	\ell(v) = \left( \int_X (c_n^m v \wedge \wbar{v})^{1/m} \right)^{m/2}
\]
is finite. Of course, $v$ will not in general satisfy the desired inequality.

The way to deal with this problem is to consider $\omXm = \omX \tensor
\omX^{\tensor(m-1)}$ as an adjoint bundle and to apply the Ohsawa-Takegoshi theorem
to get another extension of smaller length. The section $v \in H^0(X, \omXm)$ induces
a singular hermitian metric on the line bundle $\omXm$, whose curvature is
semi-positive. With respect to a local trivialization 
\[
	\varphi \colon \omXm \restr{U} \to \shO_U,
\]
the weight of this metric is given by $\log\abs{\varphi \circ v}^2$. Endow the line
bundle $\omX^{\tensor(m-1)}$ with the singular hermitian metric whose local weight is 
\[
	\frac{m-1}{m} \log \abs{\varphi \circ v}^2.
\]
It is easy to see that the norm of $u$ with respect to this metric is still equal to
$\ell_0(u) = 1$. \theoremref{thm:OT} says that there exists another
section $v' \in H^0(X, \omXm)$, with the same restriction to $X_0$, whose norm
squared is bounded by $\mu(B)$. To get a useful expression for the norm squared, write 
\[
	v' = F v,
\]
with $F$ meromorphic on $X$ and identically equal to $1$ on $X_0$; then the
inequality in the Ohsawa-Takegoshi theorem takes the form
\[
	\int_X \abs{F}^2 (c_n^m v \wedge \wbar{v})^{1/m} \leq \mu(B).
\]
We can use this to get an upper bound for the quantity
\[
	\ell(v') = \left( \int_X \abs{F}^{2/m} (c_n^m v \wedge \wbar{v})^{1/m} \right)^{m/2}.
\]
To begin with, let us write $(c_n^m v \wedge \wbar{v})^{1/m} = L \, d\mu$, where $L$ is
a nonnegative real-analytic function on $X$, and $d\mu$ is some choice of volume form.
Using H\"older's inequality with exponents $1/m$ and $(m-1)/m$, we have
\[
	\ell(v')^{2/m} = \int_X \abs{F}^{2/m} L \, d\mu 
	\leq \left( \int_X \abs{F}^2 L \, d\mu\right)^{1/m}
		\left( \int_X L \, d\mu\right)^{(m-1)/m},
\]
and therefore $\ell(v') \leq \mu(B)^{1/2} \cdot \ell(v)^{(m-1)/m}$, which we may
rewrite as
\[
	\frac{\ell(v')}{\mu(B)^{m/2}} \leq 
		\left( \frac{\ell(v)}{\mu(B)^{m/2}} \right)^{(m-1)/m}.
\]
Now we iterate this construction to produce an infinite sequence of $m$-canonical
forms $v_0, v_1, v_2, \dotsc \in H^0(X, \omXm)$, all with the correct restriction to
$X_0$. The inequality from above shows that one of two things happens: either
$\ell(v_k) \leq \mu(B)^{m/2}$ for some $k \geq 0$; or $\ell(v_k) > \mu(B)^{m/2}$ for
every $k \in \NN$, and
\[
	\lim_{k \to +\infty} \ell(v_k) = \mu(B)^{m/2}.
\]
If the former happens, we are done. If the latter happens, we apply
\lemmaref{lem:Lp}: it says that a subsequence converges uniformly on compact subsets
to an $m$-canonical form $v \in H^0(X, \omXm)$. Now $v$ satisfies $\ell(v) \leq 
\mu(B)^{m/2}$ (by Fatou's lemma), and its restriction to $X_0$ is of course still
equal to $u \wedge (df_1 \wedge \dotsb \wedge df_r)^{\tensor m}$.
\end{proof}

\begin{lemma} \label{lem:Lp}
Let $X$ be a complex manifold, and let $v_0, v_1, v_2, \dotsc \in H^0(X, \omXm)$ be a 
sequence of $m$-canonical forms such that $\ell(v_k) \leq C$ for every $k \in \NN$.
Then a subsequence converges uniformly on compact subsets to a limit $v \in H^0(X,
\omXm)$.
\end{lemma}

\begin{proof}
With respect to a local trivialization of $\omXm$, we have a sequence of holomorphic
functions whose $L^{2/m}$-norm is uniformly bounded. Using the mean-value inequality,
this implies that the sequence of functions is uniformly bounded on compact subsets;
now apply Montel's theorem to get the desired conclusion.
\end{proof}

\begin{note}
One interesting thing about the proof of \theoremref{thm:OTm} is that it looks very
similar to Viehweg's covering trick (which we used for example in the proof of
\propositionref{summand}).
The advantage of the metric approach is that one can take a limit to obtain a
solution with the same properties as in the case $m = 1$.
\end{note}

\subsection{The relative case}

With the help of \theoremref{thm:OTm}, it is quite easy to analyze the behavior of the
Narasimhan-Simha metric in families. Let us first consider the case of a smooth
morphism $f \colon X \to Y$; as in the statement of \theoremref{thm:pluricanonical},
we assume that $f$ is projective with connected fibers, and that $f(X) = Y$.
Recall that by invariance of plurigenera, the dimension of the space of $m$-canonical forms on
the fiber $X_y = f^{-1}(y)$ is the same for every $y \in Y$. 

The restriction of the relative canonical bundle $\omXY$ to the fiber $X_y$ 
identifies to the canonical bundle $\omega_{X_y}$ of the fiber. We can therefore
apply the construction in \parref{par:absolute} fiber by fiber to produce a singular
hermitian metric $h_m$ on $\omXY$, called the \define{$m$-th relative
Narasimhan-Simha metric}; we shall give a more careful definition of $h_m$ in a
moment. The first result is that $h_m$ is continuous.

\begin{proposition} \label{prop:NS-continuous}
Under the assumptions above, the relative Narasimhan-Simha metric on $\omXY$ is
continuous.
\end{proposition}

\begin{proof}
Once again, this is an application of the Ohsawa-Takegoshi theorem for
pluricanonical forms, which allows us to extend $m$-canonical forms from the fibers
of $f$,  with a uniform upper bound on the length of the extension. After shrinking
$Y$, we can assume that $Y = B$ is the open unit ball in $\CC^r$, with coordinates
$t_1, \dotsc, t_r$. We denote by $V_m = H^0(X, \omXm)$ the (typically
infinite-dimensional) vector space of all $m$-canonical forms on $X$. Given $v \in
V_m$ and a point $y \in Y$, we have
\[
	v \restr{X_y} = v_y \tensor (\dt_1 \wedge \dotsb \wedge \dt_r)^{\tensor m}
\]
for a unique $m$-canonical form $v_y \in H^0(X_y, \omega_{X_y}^{\tensor m})$. We
denote by $\ell(v)$ the length of $v$ on $X$, and by $\ell_y(v_y)$ the length of
$v_y$ on $X_y$. The Ohsawa-Takegoshi theorem for pluricanonical forms (in
\theoremref{thm:OTm}) implies that, possibly after shrinking $Y$, there is a constant
$C \geq 0$ with the following property: 
\begin{equation} \label{eq:property}
\parbox{11cm}{For every $y \in Y$ and every $m$-canonical form $u$ on $X_y$ of length $\leq 1$, \\
there is an $m$-canonical form $v \in V_m$ such that $v_y = u$ and $\ell(v) \leq C$.}
\end{equation}
Now let $n = \dim X$. As the morphism $f$ is smooth, every point in $X$ has an
open neighborhood $U$ with coordinates $z_1, \dotsc, z_{n-r}, t_1,
\dotsc, t_r$. Then $s_0 = \dz_1 \wedge \dotsb \wedge \dz_{n-r}$ gives a local
trivialization of $\omXY$, and we consider the weight function
\[
	\varphi_m = - \log \abs{s_0}_{h_m}^2 \colon U \to [-\infty, +\infty)
\]
of the relative Narasimhan-Simha metric $h_m$. On each fiber, $\varphi_m$ is
given by the formula in \eqref{eq:NS-weight-2}; we can use the Ohsawa-Takegoshi
theorem to obtain a more uniform description. For each $v \in V_m$, we
have
\[
	v \restr{U} = g_v \cdot (s_0 \wedge \dt_1 \wedge \dotsb \wedge \dt_r)^{\tensor m}
\]
for a unique holomorphic function $g_v \colon U \to \CC$. By \eqref{eq:NS-weight-2}
and \eqref{eq:property}, we have
\[
	\varphi_m(x) = \frac{2}{m} \sup \Menge{\log \abs{g_v(x)}}{%
		\text{$v \in V_m$ satisfies $\ell(v) \leq C$ and
		$\ell_y(v_y) \leq 1$}};
\]
where $y = f(x)$. We are going to prove that this defines a
continuous function on $U$.

Fix a point $x \in U$, and let $x_0, x_1, x_2, \dotsc$ be any 
sequence in $U$ with limit $x$. Set $y_k = f(x_k)$ and $y = f(x)$.
For every $k \in \NN$, choose an $m$-canonical form $u_k$ of length $\ell_{y_k}(u_k) =
1$ on the fiber $X_{y_k}$, such that $u_k$ computes $\varphi_m(x_k)$.
Extend $u_k$ to an $m$-canonical form $v_k$ of length $\ell(v_k) \leq C$ on $X$ by
using \eqref{eq:property}; then 
\[
	\varphi_m(x_k) = \frac{2}{m} \log \bigabs{g_{v_k}(x_k)}.
\]
After passing to a subsequence, $v_0, v_1, v_2, \dotsc$ converges uniformly on
compact subsets to an $m$-canonical form $v \in H^0(X, \omXm)$. Since
$\ell_{y_k}(v_{n, y_k}) = 1$, Fatou's lemma shows that $\ell_y(v_y) \leq 1$.
Moreover, the holomorphic functions $g_{v_k}$ converge uniformly on compact subsets
to $g_v$, and therefore
\begin{equation} \label{eq:lim-USC}
	\lim_{k \to +\infty} \varphi_m(x_k) = \frac{2}{m} \log \bigabs{g_v(x)}
		\leq \varphi_m(x).
\end{equation}
On the other hand, we can choose an $m$-canonical form $u'$ of length $\ell_y(u') =
1$ on the fiber $X_y$, such that $u'$ computes $\varphi_m(x)$. Extend $u'$ to an
$m$-canonical form $v'$ of length $\ell(v') \leq C$ on $X$ by using
\eqref{eq:property}; then
\[
	\varphi_m(x) = \frac{2}{m} \log \abs{g_{v'}(x)}. 
\]
Now it is easy to see from the definition of the length function that
$\ell_{y_k}(v_{y_k}')$ tends to $\ell_y(v_y')$ as $k \to +\infty$. In particular, the
$m$-canonical form $v_{y_k}'$ on $X_{y_k}$ has nonzero length for $k \gg 0$, which
means that
\[
	\frac{2}{m} \Bigl( \log \abs{g_{v'}(x_k)} - \log \ell_{y_k}(v_{y_k}') \Bigr)
		\leq \varphi_m(x_k).
\]
Since the left-hand side tends to $\varphi_m(x)$, we obtain
\begin{equation} \label{eq:lim-LSC}
	\varphi_m(x) \leq \liminf_{k \to +\infty} \varphi_m(x_k).
\end{equation}
The two inequalities in \eqref{eq:lim-USC} and \eqref{eq:lim-LSC} together say
that $\varphi_m$ is continuous.
\end{proof}

Next, we prove that $h_m$ has semi-positive curvature -- just as in the case of
adjoint bundles, the proof of this fact is very short, because we know the optimal
value of the constant in \theoremref{thm:OTm}.

\begin{proposition} \label{prop:NS-curvature}
Under the assumptions above, the relative Narasimhan-Simha metric on $\omXY$
has semi-positive curvature.
\end{proposition}

\begin{proof}
Keep the notation introduced during the proof of \propositionref{prop:NS-continuous}.
Because the local weight function $\varphi_m$ is continuous, it
suffices to prove that $\varphi_m$ satisfies the mean-value inequality for
mappings from the one-dimensional unit disk $\Delta$ into $U$. If the image of
$\Delta$ lies in a single fiber, this is okay, because we already know from
\propositionref{prop:NS-absolute} that $\varphi_m$ is plurisubharmonic on each fiber.
So assume from now on that the mapping
from $\Delta$ to $Y$ is non-constant. Since the morphism $f$ is smooth, we can then
make a base change and reduce the problem to the case where $Y = \Delta$ and 
where $i \colon \Delta \into X$ is a section of $f \colon
X \to \Delta$. 

Now let $x_0 = i(0)$ and $X_0 = f^{-1}(0)$, and choose some $u \in
H^0(X_0, \omega_{X_0}^{\tensor m})$ with $\ell_0(u) = 1$ that computes
$\varphi_m(x_0)$. By \theoremref{thm:OTm}, there exists an $m$-canonical form $v \in
H^0(X, \omXm)$ with $v \restr{X_0} = u \wedge df^{\tensor m}$, whose length
satisfies the inequality
\[
	\ell(v) \leq \mu(\Delta)^{m/2} \cdot \ell_0(u) = \pi^{m/2}.
\]
In the notation introduced during the proof of \propositionref{prop:NS-continuous},
we then have
\[
	\varphi_m(x_0) = \frac{2}{m} \log \bigabs{g_v(x_0)}.
\]
If we define $v_y \in H^0(X_y, \omega_{X_y}^{\tensor m})$ by the formula $v \restr{X_y} =
v_y \wedge df^{\tensor m}$, then we have
\[
	\ell(v)^{2/m} = \int_X (c_n^m v \wedge \wbar{v})^{1/m}
		= \int_{\Delta} \ell_y(v_y)^{2/m} \, d\mu.
\]
Now we observe that for almost every $y \in \Delta$, the ratio $v_y / \ell_y(v_y)$ is
an $m$-canonical form on $X_y$ of unit length; by definition of the weight function
$\varphi_m$, we have
\[
	\varphi_m(x) \geq 
		\frac{2}{m} \Bigl( \log \abs{g_v(x)} - \log \ell_y(v_y) \Bigr)
		= \frac{2}{m} \log \abs{g_v(x)} - \log \ell_y(v_y)^{2/m}.
\]
If we now compute the mean value of $\varphi_m \circ i$ over $\Delta$, we find that
\[
	\frac{1}{\pi} \int_{\Delta} \varphi_m(i(y)) \, d\mu \geq 
		\frac{1}{\pi} \int_{\Delta} \frac{2}{m} \log \bigabs{g_v(i(y))} \, d\mu
		- \frac{1}{\pi} \int_{\Delta} \log \ell_y(v_y)^{2/m} \, d\mu.
\]
The first term on the right is greater or equal to $2/m \log \abs{g_v(x_0)} =
\varphi_m(x_0)$, because the function $g_v \circ i$ is holomorphic. To
estimate the remaining integral, note that
\[
	\frac{1}{\pi} \int_{\Delta} \log \ell_y(v_y)^{2/m} \, d\mu
	\leq \log \left( \frac{1}{\pi} \int_{\Delta} \ell_y(v_y)^{2/m} \, d\mu \right)
	= \log \left( \frac{1}{\pi} \cdot \ell(v)^{2/m} \right) \leq 0,
\]
by Jensen's inequality and the fact that $\ell(v) \leq \pi^{m/2}$. 
Consequently, $\varphi_m$ does satisfy the required mean-value inequality, and 
$h_m$ has semi-positive curvature.
\end{proof}

\begin{note}
Compare also Lemma~7 and Lemma~8 in \cite{Ka-curve}.
\end{note}

After these preparations, we can now prove \theoremref{thm:pluricanonical} in
general.

\begin{proof}[Proof of \theoremref{thm:pluricanonical}]
Suppose that $f \colon X \to Y$ is a projective morphism between two complex
manifolds with $f(X) = Y$. Let $Z \subseteq Y$ denote the closed analytic subset
where $f$ fails to be submersive. We already know that the restriction of the line
bundle $\omXY$ to $f^{-1}(Y \setminus Z)$ has a well-defined singular
hermitian metric $h_m$ that is continuous and has semi-positive curvature. To show that $h_m$
extends to a singular hermitian metric with semi-positive curvature on all of
$X$, all we need to prove is that the local weights of $h_m$ remain bounded near
$f^{-1}(Z)$; this is justified by \lemmaref{lem:Riemann-Hartogs}.
P\u{a}un and Takayama \cite[Theorem~4.2.7]{PT} observed that this local boundedness
again follows very easily from the Ohsawa-Takegoshi theorem for pluricanonical forms.

Fix a point $x_0 \in X$ with $f(x_0) \in Z$. Since the problem is local on $Y$, we
may assume that $Y = B$ is the open unit ball in $\CC^r$, with coordinates $t_1,
\dotsc, t_r$, and that $f(x_0) = 0$. On a suitable neighborhood $U$ of the point $x_0$,
we have coordinates $z_1, \dotsc, z_n$; note that because $f$ is most likely not
submersive at $x_0$, we cannot assert that $t_1, \dotsc, t_r$ are part of this
coordinate system. Let $s_0 \in H^0(U, \omXY)$ be a local trivialization of $\omXY$,
chosen so that
\[
	\dz_1 \wedge \dotsb \wedge \dz_n = s_0 \wedge (\dt_1 \wedge \dotsb \wedge \dt_r).
\]
Denote by $\varphi_m$ the weight function of $h_m$ with respect to this
local trivialization:
\[
	\varphi_m(x) = - \log \abs{s_0}_{h_m}^2 \colon U \to [-\infty, +\infty)
\]
For $v \in H^0(X, \omXm)$, we have $v \restr{U} = g_v \cdot (\dz_1 \wedge \dotsb
\wedge \dz_n)^{\tensor m}$ for a holomorphic function $g_v \colon U \to \CC$. As
explained during the proof of \propositionref{prop:NS-continuous}, the
Ohsawa-Takegoshi theorem for pluricanonical forms implies that there is a constant $C
\geq 0$ with the following property: for every $x \in U$, there is some $v \in H^0(X,
\omXm)$ of length $\ell(v) \leq C$ such that
\[
	\varphi_m(x) = \frac{2}{m} \log \abs{g_v(x)}.
\]
For $x$ sufficiently close to $x_0$, there is a positive number $R > 0$ such that $U$
contains the closed ball of radius $R$ centered at $x$. The mean-value inequality and
the fact that $\ell(v) \leq C$ now combine to give us an upper bound for
$\varphi_m(x)$ that depends only on $C$ and $R$, but is independent of the point $x$.
In particular, $\varphi_m$ is uniformly bounded in a neighborhood of the point $x_0
\in f^{-1}(Z)$, and therefore extends uniquely to a plurisubharmonic function on
all of $U$.

The Narasimhan-Simha metric on each fiber $X_y$ with $y \not\in Z$ satisfies
\eqref{eq:identity}; by the Ohsawa-Takegoshi theorem, this means that the inclusion
\[
	\fl \bigl( \omXY \tensor L \tensor \shI(h) \bigr) \into
		\fl \bigl( \omXY \tensor L \bigr) = \fl \omXYm
\]
is an isomorphism over $Y \setminus Z$. Due to \corollaryref{cor:continuous}, the
singular hermitian metric on $\fl \omXYm$ is therefore finite and continuous on $Y
\setminus Z$.
\end{proof}

\providecommand{\bysame}{\leavevmode\hbox to3em{\hrulefill}\thinspace}
\providecommand{\MR}{\relax\ifhmode\unskip\space\fi MR }
\providecommand{\MRhref}[2]{%
  \href{http://www.ams.org/mathscinet-getitem?mr=#1}{#2}
}
\providecommand{\href}[2]{#2}

\end{document}